\documentclass[amssymb,amsfonts,reqno]{amsart}

 \usepackage{xcolor}
 \usepackage[citebordercolor={green}]{hyperref}

\usepackage{color}
\usepackage{latexsym}
\usepackage{amsmath}
\usepackage{amssymb}
\usepackage{amsfonts}
\usepackage{esint}
\usepackage{graphicx}
\usepackage{accents}
\newcommand{\Lie}{{\mathcal{L}}}

\newcommand{\der}{\nabla}

\newcommand{\les}{\lesssim}
\newcommand{\bea}{\begin{eqnarray}}
\newcommand{\eea}{\end{eqnarray}}

\newcommand{\derm}{ { \der^{(\bf{m})}} }

\newcommand{\eps}{{\varepsilon}}\newcommand{\R}{{\mathbb R}}

\newcommand{\E}{{\cal E}}

\newcommand{\la}{\langle}\newcommand{\si}{\sigma}\renewcommand{\b}{\beta}

\newcommand{\cal}{\mathcal}

\def\a{\alpha}\def\ga{\gamma}\def\de{\delta}\def\Si{\Sigma}
\def\bm{\left( \begin{array}{cc}}
\def\endm{\end{array}\right)}\newcommand{\eq}{\end{equation}}
\def\a{\alpha}\def\b{\beta}
\def\ga{\gamma}\def\de{\delta}\def\pa{\partial}

\def \rectangle#1#2{\hbox{\vrule\vbox to #2 {\hrule\hbox to #1{\hfil}\vfil\hrule}\vrule}}

\def\a{\alpha}\def\b{\beta}\def\ga{\gamma}
\def\de{\delta}\def\pa{\partial}

\def\pa{\partial}

\def\beaa{\begin{eqnarray*}}
\def\eeaa{\end{eqnarray*}}
\def\pa{\partial}
\def\a{{\alpha}}
\def\b{{\beta}}
\def\ga{\gamma}

\def\de{\delta}

\def\eps{\epsilon}
\def\la{\lambda}

\def\si{\sigma}
\def\Si{\Sigma}

\def\Om{\Omega}

\def\g{{\bf g}}

\def\SSS{{\Bbb S}}

\def\R{{\mathbb R}}

\def\12{\frac{1}{2}}

\def\bep{\begin{proposition}}
\def\eep{\end{proposition}}

\def\4{\frac{1}{4}}

\def\12{\frac{1}{2}}

\def\bep{\begin{proposition}}
\def\eep{\end{proposition}}

\def\bm#1{\boldsymbol{#1}} 

\def\build#1_#2^#3{\mathrel{\mathop{\kern 0pt#1}\limits_{#2}^{#3}}}
\def\4{\frac{1}{4}}

\def\<{\langle}
\def\>{\rangle}

\theoremstyle{plain}
\newtheorem{theorem}{Theorem}
\newtheorem{proposition}{Proposition}
\newtheorem{lemma}{Lemma}
\newtheorem{corollary}{Corollary}

\theoremstyle{remark}
\newtheorem{remark}{Remark}

\theoremstyle{definition}
\newtheorem{definition}{Definition}
\numberwithin{equation}{section}
\numberwithin{proposition}{section}
\numberwithin{definition}{section}
\numberwithin{lemma}{section}
\numberwithin{corollary}{section}
\numberwithin{remark}{section}
\begin{document}
\include{psfig}
\title[Energy Estimates for Einstein-Yang-Mills]{Energy estimates for the Einstein-Yang-Mills fields and applications}
\author{Sari Ghanem}
\address{University of Lübeck}
\email{sari.ghanem@uni-luebeck.de}
\maketitle

\begin{abstract}
We prove exterior energy estimates for tensorial non-linear wave equations, where the background metric is a perturbation of the Minkowski space-time, and where the derivatives are the Minkowski covariant derivatives. We obtain bounds in the exterior region of the Minkowski space-time, for the weighted $L^2$ norm on each component, separately, of the covariant derivative of the tensorial solutions, and we also control a space-time integral in the exterior of the covariant tangential derivatives of the solutions. As a special application, we use here these energy estimates to prove the exterior stability of the Minkowski space-time, $\R^{1+4}$\,, as solution to the coupled Einstein-Yang-Mills system associated to any compact Lie group $G$\,, in the Lorenz gauge and in wave coordinates. The bounds in the exterior for the $L^2$ norm on the covariant derivatives of each component, separately, of the tensor solution, as well as the bound on the space-time integral of the covariant tangential derivatives, are motivated by a problem that we will address in a paper that follows to prove the exterior stability of the $(1+3)$-Minkowski space-time for perturbations governed by the Einstein-Yang-Mills equations.

\end{abstract}

\setcounter{page}{1}
\pagenumbering{arabic}


\section{Introduction}\label{Introduction}

This is the second paper in a series of three, where we study the non-linear stability of the Minkowski space-time solution to the coupled Einstein-Yang-Mills equations. In this paper, we prove exterior energy estimates and apply them, as a special case, to prove the exterior stability of the $(1+4)$-Minkowski space-time solution to the Einstein-Yang-Mills system. However, our energy estimates are mostly motivated by the next paper to prove the non-linear exterior stability of the $(1+3)$-Minkowski space-time.

First, we prove exterior energy estimates for a system of coupled non-linear covariant wave equations. More precisely, we consider that we are given a fixed system of coordinates, namely $(t, x^1, \ldots, x^n)$\;, that is not necessarily wave coordinates, yet, for our application to the proof of stability of Minkowski, this system will ultimately be chosen to be that of wave coordinates. In this fixed system of coordinates, we define $m$ to be the Minkowski metric $(-1, +1, \ldots, +1)$ and we define $\derm$ to be the covariant derivative associated to the metric $m$ (see Definition \eqref{definitionofMinkowskiandcovariantderivtaiveofMinkwoforafixedgivensystemofcoordinates}). We consider any arbitrary curved space-time $(\cal M, \g) $\;, with a smooth Lorentzian metric $\g$\;, which will ultimately be, in our application to the non-linear exterior stability problem, our unknown Lorentzian manifold solution to the fully coupled Einstein-Yang-Mills system.

We study the following system of non-linear covariant tensorial wave equations for $\Phi$ on $(\cal M, \g) $\;, where the initial data for the hyperbolic Cauchy problem is given on an initial Cauchy hypersurface $\Sigma$\;, and where $V$ is any vector,
\bea\label{nonlinearsystemoftensorialwaveequations}
g^{\a\b} \derm_{\a } \derm_\b \Phi_{V} = S_{V} \; .
\eea
Here the metric $\g$ is a perturbation of Minkowski in the following sense: if we define (see Definition \ref{definitionofbigHandsmallhandrecallofdefinitionofMinkowskimetricminrelationtowavecoordiantesasreminder}),
\bea
H^{\mu\nu} &:=& g^{\mu\nu}-m^{\mu\nu} \;,
\eea
where $m^{\mu\nu}$ is the inverse of the Minkowski metric $m_{\mu\nu}$\;, that is defined to be $(-1, +1, \ldots, +1)$ in our chosen system of coordinates $(x^0, x^1, \ldots, x^n)$, where here $x^0 = t$\;, then we assume in our energy estimates that
\bea\label{boundednessoftheperturbationbigHbyaconstant}
 \sum_{\mu, \nu = 0}^{n}   | H_{\mu\nu} | < \frac{1}{n}\; .
\eea
This condition on the perturbation $H$ would make the boundary terms of our energy estimates “look like” the $L^2$ norm of the covariant gradient $ \derm \Phi_{V}  $ of the solution of our system \eqref{nonlinearsystemoftensorialwaveequations} (see Lemma \ref{howtogetthedesirednormintheexpressionofenergyestimate}).

The goal of our energy estimates is to prove an exterior energy estimate that would allow us to control the $L^2$ norm of $\derm \Phi_{V}$ where the integral would be taken on a hypersurface that is the intersection of $t = constant$\;,  in our fixed system of coordinates $(t, x^1, \ldots, x^n)$\;, with the complement of the future domain of dependance for the metric $m$ of a compact $\cal K \subset \Sigma$\;. This is what we mean by exterior energy estimate: they are bounds on $L^2$ norm in domains that are exterior to the future domain of dependance of $\cal K \subset \Sigma$\;.

We consider a specific non-symmetric tensor (see Definition \ref{defofthestreessenergymomentumtensorforwaveequationhere}) which we contract with a weighted vector (see \eqref{The weightedvectorproportionaltodt}) to get a weighted conservation law (see Lemma \ref{weightedconservationlawintheexteriorwiththeenergymomuntumtensorcontarctedwithvectordt}). Here the weightes are defined in Definitions \ref{defoftheweightw}, \ref{defwidehatw} and \ref{defwidetildew}. Thus, we get a weighted conservation law (see Corollary \ref{WeightedenergyestimateusingHandnosmallnessyet}) that has also a quantity that controls the tangential derivatives of our unknown solution $\Phi_{V}$ (see Lemma \ref{expressionofTttplusTtr}), thanks to the non-vanishing weight $\widehat{w}^\prime$ (see Lemma \ref{equaivalenceoftildewandtildeandofderivativeoftildwandderivativeofhatw}). 

The fact the metric $g$ is assumed to be close enough to the Minkowski metric (see \eqref{boundednessoftheperturbationbigHbyaconstant}), allows us to translate these conservation laws into \textit{weighted} energy estimates as in Corollary \ref{TheenerhyestimatewithtermsinvolvingHandderivativeofHandwithwandhatw}. Thanks to the non-trivial weight $\widehat{w}$ (see Definition \ref{defwidehatw}), we have control on the weighted space-time integral of the covariant tangential derivatives of $\Phi_V$\;, namely
\bea
\notag
 \int_{t_1}^{t_2}  \int_{\Sigma^{ext}_{\tau} }  \Big(    \frac{1}{2} \Big(  | \derm_t  \Phi_{V} + \derm_r \Phi_{V} |^2  +  \de^{ij}  | ( \derm_i - \frac{x_i}{r} \derm_{r}  )\Phi_{V} |^2 \Big)  \cdot d\tau \cdot \widehat{w}^\prime (q) d^{n}x \;,
\eea
 a control that will be used in the next paper for the non-linear stability of the $(1+3)$-Minkowski space-time governed by the coupled Einstein-Yang-Mills system. Also, thanks to the definition of our tensor $T$ in Definition \ref{defofthestreessenergymomentumtensorforwaveequationhere}, we get control on the $L^2$ norm on each component, namely,
  \bea
  \notag
     \int_{\Sigma^{ext}_{t} }  |\derm \Phi_{V} |^2     \cdot w(q)  \cdot d^{n}x   \; .
  \eea
 The fact that we can separate the controls for the $L^2$ norm on each component $\derm \Phi_{V} $ is necessary for our next paper that treats the case of $n=3$.
 
 We showed in \cite{G4} that the stability of the Minkowski space-time in the Lorenz gauge and in wave coordinates, solution to the Einstein-Yang-Mills equations, could be recasted as the study of non-linear wave equations, in the form of \eqref{nonlinearsystemoftensorialwaveequations}, on both the Yang-Mills potential $A$ and the perturbation metric $h$ (see Definition \ref{definitionofbigHandsmallhandrecallofdefinitionofMinkowskimetricminrelationtowavecoordiantesasreminder}) -- we refer the reader to the introduction of our previous paper \cite{G4}. 
 
 In fact, the Einstein-Yang-Mills equations are
\bea\label{EYM}
R_{ \mu \nu} - \frac{1}{2} g_{\mu\nu} R &=& 8\pi \cdot T_{\mu\nu} \; , 
\eea
where
\bea\label{definitionofstressenergymomenturmYangMillsmatter}
 T_{\mu\nu} =  \frac{1}{4 \pi} \cdot  ( < F_{\mu\b}, F_{\nu}^{\; \; \b}> - \frac{1}{4} g_{\mu\nu } < F_{\a\b },F^{\a\b } > ) \; ,
 \eea
and where $F$ is the Yang-Mills curvature, that is a two-form defined by
\bea\label{expF}
F_{\a\b} = \der_{\a}A_{\b} - \der_{\b}A_{\a} + [A_{\a},A_{\b}] \; ,
\eea
where $A$ is valued in the Lie algebra $\cal G$ associated to any compact Lie group $G$\;, and where $\der$ is the Levi-Civita covariant derivative associated to the unknown metric $\g$\;.

However, we showed in \cite{G4}, that the Einstein-Yang-Mills equations in the Lorenz gauge and in wave coordinates, could be written as follows, 

   \bea
  \notag
 && g^{\la\mu} \derm_{\la}   \derm_{\mu}   A_{\si}     \\
   \notag
 &=&    ( \derm_{\si}  h^{\a\mu} )  \cdot (  \derm_{\a}A_{\mu} )  \\
    \notag
 &&   +   \frac{1}{2}    \big(   \derm^{\mu} h^{\nu}_{\, \, \, \si} +  \derm_\si h^{\nu\mu} -   \derm^{\nu} h^{\mu}_{\,\,\, \si}  \big)   \cdot  \big( \derm_{\mu}A_{\nu} -  \derm_{\nu}A_{\mu}  \big) \\
 \notag
&&     +   \frac{1}{2}    \big(   \derm^{\mu} h^{\nu}_{\, \, \, \si} +  \derm_\si h^{\nu\mu} -   \derm^{\nu} h^{\mu}_{\,\,\, \si}  \big)   \cdot   [A_{\mu},A_{\nu}] \\
 \notag
 && -  \big(  [ A_{\mu}, \derm^{\mu} A_{\si} ]  +    [A^{\mu},  \derm_{\mu}  A_{\si} - \derm_{\si} A_{\mu} ]    +    [A^{\mu}, [A_{\mu},A_{\si}] ]  \big)  \\
 \notag
  && + O( h \cdot  \derm  h \cdot  \derm A) + O( h \cdot  \derm h \cdot  A^2) + O( h \cdot  A \cdot \derm A) + O( h \cdot  A^3) \; ,
  \eea
 and
  \bea
\notag
 && g^{\alpha\beta}\derm_\alpha \derm_\beta h_{\mu\nu} \\
 \notag
  &=& P(\derm_\mu h,\derm_\nu h)  +  Q_{\mu\nu}(\derm h,\derm h)   + G_{\mu\nu}(h)(\derm h,\derm h)  \\
\notag
 &&   -4     <   \derm_{\mu}A_{\b} - \derm_{\b}A_{\mu}  ,  \derm_{\nu}A^{\b} - \derm^{\b}A_{\nu}  >    \\
 \notag
 &&   + m_{\mu\nu }       \cdot  <  \derm_{\a}A_{\b} - \derm_{\b}A_{\a} , \derm_{\a} A^{\b} - \derm^{\b}A^{\a} >   \\
 \notag
&&           -4  \cdot  \big( <   \derm_{\mu}A_{\b} - \derm_{\b}A_{\mu}  ,  [A_{\nu},A^{\b}] >   + <   [A_{\mu},A_{\b}] ,  \derm_{\nu}A^{\b} - \derm^{\b}A_{\nu}  > \big)  \\
\notag
&& + m_{\mu\nu }    \cdot \big(  <  \derm_{\a}A_{\b} - \derm_{\b}A_{\a} , [A^{\a},A^{\b}] >    +  <  [A_{\a},A_{\b}] , \derm^{\a}A^{\b} - \derm^{\b}A^{\a}  > \big) \\
\notag
 &&  -4     <   [A_{\mu},A_{\b}] ,  [A_{\nu},A^{\b}] >      + m_{\mu\nu }   \cdot   <  [A_{\a},A_{\b}] , [A^{\a},A^{\b}] >  \\
 \notag
     && + O \big(h \cdot  (\derm A)^2 \big)   + O \big(  h  \cdot  A^2 \cdot \derm A \big)     + O \big(  h   \cdot  A^4 \big)  \;  , 
\eea
where $P$\;, $Q$ and $G$, as well as the big $O$ notation, are defined in \cite{G4}.

 Hence, if we look at the Lie derivatives of the source terms which appear for the Yang-Mills potential, where here $Z^J$ is any product of length $|J|$ of Minkowski vector fields, as explained in \cite{G4}, we have the following bound,
 
    \bea\label{liederivativeindirectionofMinkowskiforsourcetermsforA}
   \notag
 && |  \Lie_{Z^I}   ( g^{\la\mu} \derm_{\la}   \derm_{\mu}   A  )   | \\
    \notag
   &\leq& \sum_{|J| +|K|+|L|+ |M| \leq |I|} \big( \; O(  |   \derm (\Lie_{Z^J}  h) | \cdot |    \derm (\Lie_{Z^K} A ) |  )  \\
      \notag
   &&     +  O(   |   \derm ( \Lie_{Z^J} h) | \cdot |  \Lie_{Z^K}   A  | \cdot |    \Lie_{Z^L}  A |   )    \\
   \notag
   &&   +   O( |\Lie_{Z^J} A   | \cdot |    \derm ( \Lie_{Z^K}  A) |  )    + O(   |    \Lie_{Z^J} A | \cdot   |  \Lie_{Z^K} A  | \cdot  |   \Lie_{Z^L} A | )    \\
           \notag
      &&         + O( | \Lie_{Z^J}   h | \cdot | \derm ( \Lie_{Z^K}  h) | \cdot | \derm (\Lie_{Z^L}  A) | )  \\
       \notag
      && +  O( | \Lie_{Z^J}  h | \cdot |   \derm (\Lie_{Z^K}  h ) | \cdot  | \Lie_{Z^L}  A | \cdot \Lie_{Z^M} A |  \\
       \notag
      && + O( | \Lie_{Z^J} h  | \cdot |  \Lie_{Z^K} A | \cdot |  \derm (\Lie_{Z^L} A) | )   + O( | \Lie_{Z^J} h | \cdot | \Lie_{Z^K} A | \cdot | \Lie_{Z^L} A | \cdot | \Lie_{Z^M} A | ) \; \big) \; . \\
  \eea

Unlike the case of the Einstein vacuum equations, already in 4 space-dimensions, i.e. for $n=4$\;, we see that the non-linear structure of the Einstein-Yang-Mills potential $A$, exhibits terms such that $| A   | \cdot |    \derm  A |$ and  $|  A |^3$\;, which are troublesome in the interior region, inside the outgoing light cone prescribed by $q:= r-t < 0$\;, where, $t$ and $r:= |x|$ are the time and space wave coordinates of the Einstein-Yang-Mills equations in the Lorenz gauge. Here, we chose to fix our system of coordinates as being that of wave coordinates and then we defined the Minkowski space-time to be the Minkowski metric in this system of wave coordinates: thus, these are the time and space coordinates of our Minkowski space-time.

These “troublesome” terms, namely $| A   | \cdot |    \derm  A |$ and $|  A |^3$\;, exhibit factors for $|A|$ which are not integrable in the interior. Indeed, since we are working with wave equations and therefore the energy is at the level of $\derm A$ (see \eqref{definitionoftheweightedenergy}), the Grönwall inequality that we would like to establish is a one that is at the level of the norm of the gradient of the Lie derivatives of the fields. Hence, trying to use a Hardy type inequality (as the one given in Corollary \ref{HardytypeinequalityforintegralstartingatROm}), in order to transform $A$ into $\derm A$\;, we encounter a problem in the interior region that is already there for $n=4$\;.

As mentioned earlier, the energy that we defined involves only the gradient of the fields, and not the field itself (see \eqref{definitionoftheweightedenergy}). One needs to convert a control on the zero-derivatives in the source terms into a control on the gradient of the fields using a Hardy type inequality. Yet, a Hardy type inequality says that an $L^2$ norm of the field with a weight of $\frac{1}{(1+|q|)^2}$ could be converted into an $L^2$ norm of the gradient of the field, which is the one that interests us. However, in the case of $n\geq 4$\;, in order to close the argument to bound the higher order energy, one needs to control $(1+t )^{1+\la} \cdot |  \Lie_{Z^I}   ( g^{\la\mu} \derm_{\la}   \derm_{\mu}   A  )   |^2$\;, with $\la > 0$\;, where $\Lie_{Z^I}$ are the Minkowski Lie derivatives, in order to get a uniform bound on the energy. This implies that a factor of the field $A$ should be $\frac{1}{(1+t)^{1+\la}\cdot (1+|q|)}$\;. This way, we would have
\bea
\notag
(1+t )^{1+\la} \cdot \Big( \frac{1}{(1+t)^{1+\la}\cdot (1+|q|)} \cdot |A|  \Big)^2 &=& \frac{1}{(1+t)^{1+\la}\cdot (1+|q|)^2} \cdot |A|^2 \; . \\
\eea
Then, by using Hardy inequality, as explained, we could then obtain a control on the integral of the above quantity by an integral of this quantity $ \frac{1}{(1+t)^{1+\la}} \cdot |\derm A|^2$\;, which when appears in a Grönwall type inequality, allows one to conclude that the energy will indeed remain bounded. However, looking at the source terms which appear in the wave equation on $A$ (see \eqref{liederivativeindirectionofMinkowskiforsourcetermsforA}), we see that in trying to control the non-linear structure to close a Grönwall type inequality on the energy, we are confronted to terms either of this type
\bea
         \notag
 && \sum_{|J| +|K|\leq |I|}   |\Lie_{Z^J} A   | \cdot |    \derm ( \Lie_{Z^K}  A) | \\
 \notag
  &\les& \Big( E (   \lfloor \frac{|I|}{2} \rfloor + \lfloor  \frac{n}{2} \rfloor  + 1)  \cdot \begin{cases}  \frac{\eps }{(1+t+|q|)^{\frac{(n-1)}{2}-\delta} (1+|q|)^{1+\gamma}},\quad\text{when }\quad q>0\;,\\
           \notag
      \frac{\eps  }{(1+t+|q|)^{\frac{(n-1)}{2}-\delta}(1+|q|)^{\frac{1}{2} }}  \,\quad\text{when }\quad q<0 \;, \end{cases} \Big) \\
         \notag
      && \cdot \sum_{|K|\leq |I|}   |  \Lie_{Z^K} A | \\
        \notag
      && +  \Big( E (   \lfloor \frac{|I|}{2} \rfloor + \lfloor  \frac{n}{2} \rfloor  + 1)  \cdot \begin{cases}  \frac{\eps }{(1+t+|q|)^{\frac{(n-1)}{2}-\delta} (1+|q|)^{\gamma}},\quad\text{when }\quad q>0\; ,\\
           \notag
      \frac{\eps \cdot (1+|q|)^{\frac{1}{2} }  }{(1+t+|q|)^{\frac{(n-1)}{2}-\delta}}  \,\quad\text{when }\quad q<0 \; , \end{cases} \Big) \\
         \notag
      && \cdot \sum_{|K|\leq |I|}   | \derm ( \Lie_{Z^K} A ) | \; ,  \\
  \eea
or of this type
\bea  
         \notag
 && \sum_{|J| +|K| + |L| \leq |I|}  |    \Lie_{Z^J} A | \cdot   |  \Lie_{Z^K} A  | \cdot  |   \Lie_{Z^L} A |  \\
          \notag
       &\les& \Big( E (   \lfloor \frac{|I|}{2} \rfloor + \lfloor  \frac{n}{2} \rfloor  + 1)  \cdot \begin{cases}  \frac{\eps }{(1+t+|q|)^{(n-1)-2\delta} (1+|q|)^{2\gamma}},\quad\text{when }\quad q>0\; ,\\
           \notag
      \frac{\eps \cdot (1+ |q|)  }{(1+t+|q|)^{(n-1)-2\delta}}  \,\quad\text{when }\quad q<0 \; , \end{cases} \Big)  \\
               \notag
      && \cdot \sum_{|K|\leq |I|}   |  \Lie_{Z^K} A  | \; . \\
  \eea
The decaying factors here are the ones which arise from a weighted version of the Klainerman-Sobolev inequality, and $E (   \lfloor \frac{|I|}{2} \rfloor + \lfloor  \frac{n}{2} \rfloor  + 1)$ is the constant that bounds the higher order energy norm used in our bootstrap argument for $\lfloor \frac{|I|}{2} \rfloor + \lfloor  \frac{n}{2} \rfloor  + 1$ derivatives of the fields, and for now $\de = 0$ since we are assuming that the energy is uniformly bounded. However, as we can see, both of these terms do not have the right factor in the interior region for $n=4$\;, as they enter as $\frac{1}{(1+t) \cdot (1+|q|)}$ factors for $|A|$ instead of $\frac{1}{(1+t)^{1+\la}\cdot (1+|q|)}$, $\la > 0$\;.
 
 One could think about fixing the problem by allowing a polynomial growth of the energy by a rate of $(1+t)^\de$\;, if $\de > 0$\;, and therefore, one could relax the need to have the factor in the Grönwall lemma to be integrable but to become instead $\frac{1}{(1+t)}$\;. Would this fix the problem? As discussed above, in order to effectively use a Hardy type inequality, one would then need a factor of $\frac{1}{(1+t) \cdot (1+|q|)}$ in front of the field $A$. However, still under such a relaxed bootstrap assumption, the factors which appear in front of $A$ generated from the terms $| A   | \cdot |    \derm  A |$ and  $|  A |^3$ are not good enough in the interior region, as they enter as  $\frac{1}{(1+t)^{1-\de}\cdot (1+|q|)}$ factor for $|A|$\;.  
 
 Naively, one could try to counter the problem in the interior by putting a weight in the interior, say of $(1+|q|)^\b$\;, with $\b >0$ for $q <0$\;, in the weighted higher order energy so as to obtain a better control in the interior region when using the Klainerman-Sobolev inequality, as we did for the exterior region. While this seems an interesting idea, in reality doing so, would change the sign of the derivative of the weight at $q=0$\;, since the weight grows as, say  $(1+|q|)^\a$\;, with $\a >0$ for $q \geq 0$\;, and as  $(1+|q|)^\b$\;, with $\b >0$ for $q <0$\;, which means that the derivative of the weight is negative in the interior region. This implies that a space-time integral enters with the wrong sign (the negative sign generating from the growing weight in the interior) in the energy estimate that we would like to use in order to establish the Grönwall inequality (see Lemma \ref{TheenerhyestimatewithtermsinvolvingHandderivativeofH} and Corollary \ref{TheenerhyestimatewithtermsinvolvingHandderivativeofHandwithwandhatw}). Therefore, we can only introduce a weight in the exterior region, and as much weight we want, which would allow us to gain more decay in $|q|$\;, yet only in the exterior region.
 
 We could then proceed in the exterior to obtain a suitable Grönwall type ineuqlaity on the energy (see Lemma \ref{Gronwallinequalityintheexteriorontheenergyfornequal4}) that would allow us to close the bootstrap argument that we started, as we do in Proposition \ref{Thepropositiononclosingthebootstrapargumenttoactuallyboundtheenergy}. We will then prove the following theorem.

\subsection{The statements}\

We will prove the energy estimate given in Corollary \ref{TheenerhyestimatewithtermsinvolvingHandderivativeofHandwithwandhatw} and as a special application for $n=4$\;, we will use this energy estimate to prove Proposition \ref{Thepropositiononclosingthebootstrapargumenttoactuallyboundtheenergy}. Thus, based on the set-up detailed in our previous paper \cite{G4}, we would have by then proved the following theorem.

\begin{theorem}
Let $n\geq 4$\;. Assume that we are given an initial data set $(\Sigma, \overline{A}, \overline{E}, \overline{g}, \overline{k})$ for \eqref{EYM}. We assume that $\Sigma$ is diffeomorphic to $\R^n$\;. Then, there exists a global system of coordinates $(x^1, ..., x^n) \in \R^n$ for $\Sigma$\;. We define
\bea
r := \sqrt{ (x^1)^2 + ...+(x^n)^2  }\;.
\eea
Furthermore, we assume that the data $(\overline{A}, \overline{E}, \overline{g}, \overline{k}) $ is smooth and asymptotically flat.

Let $\de_{ij}$ be the Kronecker symbol, and let $\overline{h}_{ij} $ be defined in this system of coordinates $x^i$\;, by
 \bea
\overline{h}_{ij} := \overline{g}_{ij} - \de_{ij} \; .
\eea

We then define the weighted $L^2$ norm on $\Sigma$\;, namely $\overline{\E}_N$\;, for $\ga > 0$\;, by
 \bea\label{definitionoftheenergynormforinitialdata}
 \notag
&& \overline{\E}_N \\
 \notag
&:=&  \sum_{|I|\leq N} \big(   \| (1+r)^{1/2 + \ga + |I|}   \overline{D} (  \overline{D}^I  \overline{A}    )  \|_{L^2 (\Sigma)} +  \|(1+r)^{1/2 + \ga + |I|}    \overline{D}  ( \overline{D}^I \overline{h}   )  \|_{L^2 (\Sigma)} \big) \\
\notag
&:=&  \sum_{|I|\leq N}      \big(   \sum_{i=1}^{n}  \|(1+r)^{1/2 + \ga + |I|}     \overline{D} (  \overline{D}^I  \overline{A_i}    )  \|_{L^2 (\Sigma)} +  \sum_{i, j =1}^{n}  \|(1+r)^{1/2 + \ga + |I|}     \overline{D}  ( \overline{D}^I \overline{h}_{ij}   )  \|_{L^2 (\Sigma)} \big) \; , \\
\eea
where the integration is taken on $\Sigma$ with respect to the Lebesgue measure $dx_1 \ldots dx_n$\;, and where $\overline{D} $ is the Levi-Civita covariant derivative associated to the given Riemannian metric $\overline{g}$\;.

We also assume that the initial data set $(\Sigma, \overline{A}, \overline{E}, \overline{g}, \overline{k})$ satisfies the Einstein-Yang-Mills constraint equations, namely
\bea
\notag
  \mathcal{R}+ \overline{k}^i_{\,\, \, i} \overline{k}_{j}^{\,\,\,j}  -  \overline{k}^{ij} \overline{k}_{ij}   &=&    \frac{4}{(n-1)}   < \overline{E}_{i}, \overline{E}^{ i}>   \\
 \notag
 && +  < \overline{D}_{i}  \overline{A}_{j} - \overline{D}_{j} \overline{A}_{i} + [ \overline{A}_{i},  \overline{A}_{j}] ,\overline{D}^{i}  \overline{A}^{j} - \overline{D}^{j} \overline{A}^{i} + [ \overline{A}^{i},  \overline{A}^{j}] >  \;  ,\\
 \notag
\overline{D}_{i} \overline{k}^i_{\,\,\, j}    - \overline{D}_{j} \overline{k}^i_{\, \,\,i}  &=&  2 < \overline{E}_{i}, \overline{D}_{j}  \overline{A}^{i} - \overline{D}^{i} \overline{A}_{j} + [ \overline{A}_{j},  \overline{A}^{i}]  >  \;  ,\\
\notag
\overline{D}^i \overline{E}_{ i} + [\overline{A}^i, \overline{E}_{ i} ]  &=& 0  \;  . \\
\eea

For any $n \geq 4$\;, and for any $N \geq 2 \lfloor  \frac{n}{2} \rfloor  + 2$\;, there exists a constant $ \overline{c} (N, \ga)$ depending on $N$ and on $\ga$\;, such that if 

\bea\label{Assumptiononinitialdataforglobalexistenceanddecay}
\overline{\E}_N \leq   \overline{c} ( N, \ga) \; ,
\eea
then there exists a solution $(\cal{M}, A, g)$ to the Cauchy problem for the fully coupled Einstein-Yang-Mills equations \eqref{EYM} in the future of the whole causal complement of any compact $\cal{K} \subset \Sigma$\;, converging to the null Yang-Mills potential and to the Minkowski space-time in the following sense: if we define the metric $ m_{\mu\nu}$ to be the Minkowski metric in wave coordinates $(x^0, x^1, \ldots, x^n)$\, and define $t = x^0$\;, and if we define in this system of wave coordinates 
\bea
h_{\mu\nu} := g_{\mu\nu} - m_{\mu\nu}  \;  ,
\eea
then, for $\overline{h}_{ij} $ and $\overline{A}_i$ decaying sufficiently fast as exhibit in Proposition \ref{Thepropositiononclosingthebootstrapargumenttoactuallyboundtheenergy}, we have the following estimates on $h$\;, and on $A$ in the Lorenz gauge, for the norm constructed using wave coordinates, by taking the sum over all indices in wave coordinates. That there exists a constant $E(N)$\;, that depends on $\overline{c} (N, \ga)$\;, such that for all $|I| \leq N -  \lfloor  \frac{n}{2} \rfloor  - 1$\;, we have the following estimates in the whole complement of the future causal of the compact $\cal{K}  \subset \Sigma$\;, 

\bea
 \notag
  && \sum_{\mu= 0}^{n} |\derm  ( \Lie_{Z^I}  A_{\mu} ) (t,x)  |     +   \sum_{\mu, \nu = 0}^{n}  |\derm  ( \Lie_{Z^I}  h_{\mu\nu} ) (t,x)  |   \\
   \notag
    &\leq&  C(\cal{K} )  \cdot \frac{E^{ext} (N) }{(1+t+|r-t|)^{\frac{(n-1)}{2}} (1+|r-t|)^{1+\gamma}} \; ,\\
      \eea
and
 \bea
 \notag
  \sum_{\mu= 0}^{n}  |\Lie_{Z^I} A_{\mu} (t,x)  | +   \sum_{\mu, \nu = 0}^{n}  |\Lie_{Z^I}  h_{\mu\nu} (t,x)  |  &\leq&C(\cal{K} )     \cdot c(\ga)  \cdot \frac{E^{ext} (N) }{(1+t+|r-t|)^{\frac{(n-1)}{2}} (1+|r-t|)^{\gamma}} \; , \\
      \eea
            where $Z^I$ are the Minkowski vector fields.

      In particular, the gauge invariant norm on the Yang-Mills curvature decays as follows, for all $|I| \leq N -  \lfloor  \frac{n}{2} \rfloor  - 1$\;,
 \bea
 \notag
 \sum_{\mu, \nu = 0}^{n}  |\Lie_{Z^I} F_{\mu\nu}  (t,x) |  &\leq&C(\cal{K} )  \cdot \frac{E^{ext} (N)}{(1+t+|r-t|)^{\frac{(n-1)}{2}} (1+|r-t|)^{1+\gamma}}  \\
  \notag
&& +     C(\cal{K} )   \cdot c(\ga)  \cdot \frac{ E^{ext} (N) }{(1+t+|r-t|)^{(n-1)} (1+|r-t|)^{2\gamma}} \; . \\
      \eea
      
      Furthermore, if one defines $w$ as follows, 
\bea
w(q):=\begin{cases} (1+|r-t|)^{1+2\gamma} \quad\text{when }\quad r-t>0 \;, \\
         1 \,\quad\text{when }\quad r-t<0 \; , \end{cases}
\eea
and if we define  $\Sigma_t^{ext} (\cal K)$ as being the time evolution in wave coordinates of $\Sigma$ in the future of the causal complement of $\cal{K}$\;, then for all time $t$\;, we have
\bea\label{theboundinthetheoremonEnbyconstantEN}
\notag
&& \E _{N}^{ext} (\cal K) (t) \\
\notag
&:=&  \sum_{|J|\leq N} \big( \|w^{1/2}   \derm ( \Lie_{Z^J} h   (t,\cdot) )  \|_{L^2 (\Sigma_t^{ext} (\cal K)) } +  \|w^{1/2}   \derm ( \Lie_{Z^J}  A   (t,\cdot) )  \|_{L^2 (\Sigma_t^{ext} (\cal K))  }  \big) \\
&\leq& E^{ext}(N) \; .
\eea
     
More precisely, for any constant $E^{ext}(N)$\;, there exist two constants, a constant $c_1$ that depends on $\ga > 0$ and on $n \geq 5$, and a constant $c_2$ (to bound $ \overline{\E}_N (0)$ defined in \eqref{definitionoftheenergynormforinitialdata}), that depends on $E^{ext} (N)$\,, on $N \geq 2 \lfloor  \frac{n}{2} \rfloor  + 2$ and on $w$ (i.e. depends on $\gamma$), such that if
\bea\label{AssumptiononinitialdataforcertainfirstminimalLiederivativesglobalexistenceanddecay}
 \overline{\E}_{ ( \lfloor  \frac{n}{2} \rfloor  +1)} (0)  \leq c_1(\ga, n ) \; ,
\eea
and if
\bea\label{Assumptiononinitialdataforglobalexistenceanddecay}
 \overline{\E_N} (0) \leq c_2 (E^{ext}(N), N, \ga) \; ,
\eea
then, we have for all time $t$\,, 
\bea\label{Theboundontheglobaleenergybyaconstantthatwechoose}
 \E_{N}^{ext} (\cal K)  (t) \leq E^{ext}(N) \; .
\eea

\end{theorem}

 \section{Energy estimates in the exterior region}

 The goal is to prove exterior stability of the Minkowski space-time for $n=4$ in this paper, and for $n=3$ in the paper that follows. We note that exterior stability for $n\geq 5$ is implied from the above, and we have already prove global stability for for $n\geq 5$ in a previous paper. For this, we will derive exterior energy estimates and we will then use a Klainerman-Sobolev inequality in the exterior and make a continuity argument, as before, yet on the energy in the exterior.

\subsection{The definitions and notations}\

\begin{definition}\label{definitionofMinkowskiandcovariantderivtaiveofMinkwoforafixedgivensystemofcoordinates}
Let  $(x^0, x^1, \ldots, x^n)$ be a fixed system of coordinates (which will ultimately be chosen to be a system of wave coordinates), which we shall also write, sometimes, as $(t, x^1, \ldots, x^n)$\;, where $t = x^0$\;. We define $m$ be the Minkowski metric $(-1, +1, \ldots, +1)$ in this fixed system of coordinates. We define $\derm$ be the covariant derivative associate to the metric $m$\;..

\end{definition}

\begin{definition}
For an arbitrary tensor or arbitrary order, say $K_{\a\b}$, either valued in the Lie algebra associated to the Lie group $G$ or a scalar, we define
 \bea
 \notag
|  K |^2 &:=&  \sum_{\a,\; \b \in  \{t, x^1, \ldots, x^n \}} |  K_{\a\b} |^2   \, .
\eea
Since $\derm  K_{UV}$ is tensor, and $\derm  K$ is a tensor of higher order, the definition of the norm gives
 \bea
 \notag
|  \derm  K_{UV} |^2 &=&  \sum_{\a,\; \b,\, \mu \in  \{t, x^1, \ldots, x^n \}} |  \derm_{\mu} K_{UV} |^2   \, ,
\eea
and
 \bea
 \notag
|  \derm  K |^2 &=&  \sum_{\a,\; \b,\, \mu \in  \{t, x^1, \ldots, x^n \}} |  \derm_{\mu} K_{\a\b} |^2   \, .
\eea
Note that this definition coincides with the definition of the norm that we gave in \cite{G4}, although we introduced it there in a more general fashion using contractions with a constructed euclidian metric.
\end{definition}

 \begin{definition}\label{defofthestreessenergymomentumtensorforwaveequationhere}
 Let either $\Phi$ be a tensor either valued in the Lie algebra $\cal G$, associated to the Loe group $G$\;, or a scalar. In particular, we will have in this paper either $\Phi_{V} = A_{V}$ or $\Phi_{UV} = h_{UV}$\;.
 We consider the following non-symmetric tensor for wave equations. When $\Phi =A$\;, we define
\bea\label{definitionofthewavestreessenergymomentumtensorfortensorphiVwhereVisanyvector}
\notag
T^{(\bf{g}) \; \mu}_{\;\;\;\;\;\;\;\;\;\;\nu} (\Phi_{V} )    =  g^{\mu\a}< \derm_\a \Phi_{V} ,   \derm_\nu \Phi_{V} > - \frac{1}{2} m^{\mu}_{\;\;\;\nu}  \cdot g^{\a\b}  < \derm_\a \Phi_{V} ,   \derm_\b \Phi_{V} > \;, \\
 \eea
 and when $\Phi = h$\;, we define
 \bea\label{definitionofthewavestreessenergymomentumtensorfortensorphiVwhereUVareanyvector}
\notag
T^{(\bf{g}) \; \mu}_{\;\;\;\;\;\;\;\;\;\;\nu} (\Phi_{UV} )    =  g^{\mu\a}< \derm_\a \Phi_{UV}  ,   \derm_\nu \Phi_{UV}  > - \frac{1}{2} m^{\mu}_{\;\;\;\nu}  \cdot g^{\a\b}  < \derm_\a \Phi_{UV}  ,   \derm_\b \Phi_{UV}  > \;, \\
 \eea
where we raise index with respect to the Minkowski metric $m$\,, defined to be Minkowski in wave coordinates. We consider $\Phi$ to be a field decaying fast enough at spatial infinity, so that there is no contribution at null infinity.

For simplicity of notation, we will write from now on, either $\Phi$ to say $\Phi_V$\;, or to say $\Phi_{UV}$ where $U\,,\, V$ are any vectors, which in this paper will be ultimately chosen to be wave coordinates. We will therefore write
\bea
\notag
T^{(\bf{g}) \; \mu}_{\;\;\;\;\;\;\;\;\;\;\nu} (\Phi )    =  g^{\mu\a}< \derm_\a \Phi ,   \derm_\nu \Phi > - \frac{1}{2} m^{\mu}_{\;\;\;\nu}  \cdot g^{\a\b}  < \derm_\a \Phi ,   \derm_\b \Phi > \;, \\
 \eea
 however, we will keep in mind, that we actually have \eqref{definitionofthewavestreessenergymomentumtensorfortensorphiVwhereVisanyvector} or \eqref{definitionofthewavestreessenergymomentumtensorfortensorphiVwhereUVareanyvector} for any vectors $U\;,\; V$.
\end{definition}

\begin{definition}\label{definitionofLtilde}
We define $ \widetilde{L}^{\nu} $ as a vector at a point in the space-time, such that  $T_{\widetilde{L} t}^{(\bf{g})} \geq 0$\;. We note that such a definition does not give a unique vector, however, we will use this to construct $N_{t_1}^{t_2} (q_0) $ in Definition \ref{definitionofNandofNtruncatedbtweentwots} and to define $ \hat{L}^{\nu} $ in Definition \ref{definitionofwidetildeLfordivergencetheoremuse}.
 \end{definition}

\begin{definition}\label{definitionofNandofNtruncatedbtweentwots}
We define $N_{t_1}^{t_2} (q_0) $ as a boundary made by the following:

For a $t_N \geq 0 $\;, we take a curve in the plane $(t, r)$\;, at an $\Om \in \SSS^{n-1}$ fixed, that starts at $(t=t_N\,,\, r=0) $ and extends to the future of $\Sigma_{t_0}$\;, and where $ \widetilde{L}^{\nu} $ (defined in Definition \ref{definitionofLtilde}) is orthogonal (with respect to the Minkowski metric $m$) to that curve at each point. For a $t_N \geq 0 $ large enough, depending on $q_0$\;, we have under the bootstrap assumption, and therefore under the a priori estimates in Lemmas \ref{aprioriestimatesongradientoftheLiederivativesofthefields} and \ref{aprioriestimatefrombootstraponzerothderivativeofAandh1}, that there exits such a curve that is contained in the region $\{(t, x) \;|\; q:= r-t \leq q_0 \}$\;.

We define $N (q_0)$ as the product of $\SSS^{n-1}$ and of such future inextendable curve, and we define $N_{t_1}^{t_2} (q_0) $ as $N (q_0) $ truncated between $t_1$ and $t_2$\;. Thus,  $N (q_0)$ depends on the choice of the starting point $(t_N, 0)$ that is the tip of $N (q_0)$\;, however we write $N(q_0)$ to refer to a choice of such $N(q_0)$\;. To lighten the notation, we write $N$\;, instead of $N(q_0)$\;, where a $q_0$ has been chosen fixed. 

\end{definition} 

\begin{definition} Let $\overline{C}$ be the exterior region defined as the following:

The boundary $N$ (given in Definition \eqref{definitionofNandofNtruncatedbtweentwots}) separates the space-time into two regions: one interior region that we shall call $C$ where space-like curves are contained in a compact region, and the other region is the complement of $C$\;, where space-like curves can go to spatial infinity. We define $\overline{C}$ as the complement of $C$\;. 

\end{definition} 

\begin{definition}\label{definitionoftheexteriorslicesigmat}
 We define for a given fixed $t$\;,
   \bea
\Sigma^{ext}_{t}  &:=&  \Sigma_t  \cap \overline{C} \, .
\eea
\end{definition}

\begin{definition} 

We define $n^{(\bf{m}), \nu}_{\Sigma}$ as the unit orthogonal vector (for the metric $m$) to the hypersurface $\Sigma^{ext}_{t}$\;, and $dv^{(\bf{m})}_\Sigma$ as the induced volume form on $\Sigma_t$\;.

We denote by $n^{(\bf{m}), \nu}_{N}$ an orthogonal vector (for the metric $m$) to the hypersurface $N_{t_1}^{t_2}$\;,  and by $dv^{(\bf{m})}_N$ a volume form on $N_{t_1}^{t_2} $ such that the divergence theorem applies.

\end{definition} 

\begin{definition}\label{definitionofwidetildeLfordivergencetheoremuse}
We define $ \hat{L}^{\nu} $ as a vector proportional to $ \widetilde{L}^{\nu} $ (defined in Definition \ref{definitionofLtilde}), and therefore orthogonal (with respect to the Minkowski metric $m$) to $N$ (that is defined in Definition \ref{definitionofNandofNtruncatedbtweentwots}), and such that $\hat{L}$ is oriented,  and with Euclidian length, in such a way that the stated divergence theorem in what follows hold true: see, for instance, Lemmas \ref{conservationlawwithoutcomputingdivergenceofTmut} and \ref{weightedconservationlawintheexteriorwiththeenergymomuntumtensorcontarctedwithvectordt} and Corollary \ref{WeightedenergyestimateusingHandnosmallnessyet}.
 \end{definition}

\begin{remark}
Based on the construction of $N(q_0)$ in Definition \ref{definitionofNandofNtruncatedbtweentwots}, we have that the exterior region $\overline{C} $ includes the region $\{(t, x) \;|\; q:= r- t \geq q_0 \}$\;.
\end{remark}

    \begin{definition}\label{definitionofbigHandsmallhandrecallofdefinitionofMinkowskimetricminrelationtowavecoordiantesasreminder}
We define $H$ as the 2-tensor given by
\bea
H^{\mu\nu} &:=& g^{\mu\nu}-m^{\mu\nu} \;,
\eea
where $m^{\mu\nu}$ is the inverse of the Minkowski metric $m_{\mu\nu}$\;, defined in Definition \ref{definitionofMinkowskiandcovariantderivtaiveofMinkwoforafixedgivensystemofcoordinates}. 
In addition, we define
\bea
h_{\mu\nu} &:=& g_{\mu\nu} - m_{\mu\nu} \; , \\
h^{\mu\nu} &:=& m^{\mu\mu^\prime}m^{\nu\nu^\prime}h_{\mu^\prime\nu^\prime} \;.
\eea
\end{definition}

\subsection{Conservation laws for wave equations}\
 
\begin{lemma}\label{generalconservationlawwithanyvectorfieldX}
For a vector field $X^{\nu}$\,, we have
 \bea
 \notag
&& \int_{t_1}^{t_2}  \int_{\Sigma^{ext}_{\tau}} \der^{\mu} \big( X^{\nu}T_{\mu\nu}^{(\bf{g})} \big)  \cdot dv^{(\bf{m})} \\
 \notag
 &=&  \int_{t_1}^{t_2}  \int_{\Sigma^{ext}_{\tau} } \Big( \big( \derm^{\mu} X^{\nu}  \big) \cdot T_{\mu\nu}^{(\bf{g})}  + < g^{\mu\a} \derm_{\mu } \derm_\a \Phi_{V}  ,   \derm_\nu \Phi_{V}  >\\
  \notag
&& +( \derm_{\mu } H^{\mu\a} ) \cdot < \derm_\a \Phi_{V}  ,   \derm_\nu \Phi_{V}  > - \frac{1}{2} m^{\mu}_{\;\;\;\nu}  \cdot  ( \derm_{\mu } H^{\a\b} ) \cdot < \derm_\a \Phi_{V}  ,   \derm_\b \Phi_{V}  >   \Big) \cdot dv^{(\bf{m})}  \\
  \notag
 &=&   \int_{\Sigma^{ext}_{t_1} }  \big( X^{\mu}T_{\mu\nu}^{(\bf{g})} \big)  n^{(\bf{m}), \nu}_{\Sigma} \cdot dv^{(\bf{m})}_\Sigma  -  \int_{\Sigma^{ext}_{t_2}}  \big( X^{\mu}T_{\mu\nu}^{(\bf{g})} \big)  n^{(\bf{m}), \nu}_{\Sigma} \cdot dv^{(\bf{m})}_\Sigma \\
 && -  \int_{N_{t_1}^{t_2} }  \big( X^{\mu}T_{\mu\nu}^{(\bf{g})} \big)  n^{(\bf{m}), \nu}_{N} \cdot dv^{(\bf{m})}_N \; .
 \eea
 where the tensor $T_{\mu\nu}$ is defined in Definition \ref{defofthestreessenergymomentumtensorforwaveequationhere}.
\end{lemma}

\begin{proof}

Contracting the stress-energy-momentum tensor with a vector field $X^{\nu}$\,, and applying the divergence theorem to $X^{\nu}T_{\mu\nu}$\,, one gets
 
 \bea
 \notag
&&  \int_{\Sigma^{ext}_{t_2} }  \big( X^{\mu}T_{\mu\nu}^{(\bf{g})} \big)  n^{(\bf{g}), \nu}_{\Sigma} \cdot dv^{(\bf{g})}_\Sigma +   \int_{N_{t_1}^{t_2} }  \big( X^{\mu}T_{\mu\nu}^{(\bf{g})} \big)  n^{(\bf{g}), \nu}_{N} \cdot dv^{(\bf{g})}_N + \int_{t_1}^{t_2}  \int_{\Sigma^{ext}_{\tau} (q)} \der^{\mu} \big( X^{\nu}T_{\mu\nu}^{(\bf{g})} \big)  \cdot dv^{(\bf{g})} \\
 \notag
 &=&   \int_{\Sigma^{ext}_{t_1} }  \big( X^{\mu}T_{\mu\nu}^{(\bf{g})} \big)  n^{(\bf{g}), \nu}_{\Sigma} \cdot dv^{(\bf{g})}_\Sigma  \; . \\
 \eea

We then compute 
\bea
\derm^{\mu }T^{(\bf{g})}_{\mu\nu} &:=& m^{\mu\la}\derm_{\la }T_{\mu\nu} =\derm_{\mu } T^{\mu}_{\;\;\;\nu} \; .
\eea

We get
\beaa
&& \derm_{\mu }T^{(\bf{g}) \; \mu}_{\;\;\;\;\;\;\;\;\;\;\nu}  \\
&=& ( \derm_{\mu } g^{\mu\a} ) \cdot < \derm_\a \Phi ,   \derm_\nu \Phi > - \frac{1}{2} m^{\mu}_{\;\;\;\nu}  \cdot  ( \derm_{\mu } g^{\a\b} ) \cdot < \derm_\a \Phi ,   \derm_\b \Phi > \\
&& + g^{\mu\a}< \derm_{\mu } \derm_\a \Phi ,   \derm_\nu \Phi > +  g^{\mu\a}< \derm_\a \Phi ,  \derm_{\mu } \derm_\nu \Phi > \\
&&  - \frac{1}{2} m^{\mu}_{\;\;\;\nu}  \cdot g^{\a\b}  <\derm_{\mu } \derm_\a \Phi ,   \derm_\b \Phi >  - \frac{1}{2} m^{\mu}_{\;\;\;\nu}  \cdot g^{\a\b}  < \derm_\a \Phi , \derm_{\mu }  \derm_\b \Phi > \\
&& \text{(where we used the fact that $\derm m = 0$)} \\
&=& ( \derm_{\mu } g^{\mu\a} ) \cdot < \derm_\a \Phi ,   \derm_\nu \Phi > - \frac{1}{2} m^{\mu}_{\;\;\;\nu}  \cdot  ( \derm_{\mu } g^{\a\b} ) \cdot < \derm_\a \Phi ,   \derm_\b \Phi > \\
&& + < g^{\mu\a} \derm_{\mu } \derm_\a \Phi ,   \derm_\nu \Phi > +  g^{\mu\a}< \derm_\a \Phi ,  \derm_{\mu } \derm_\nu \Phi > \\
&&  - m^{\mu}_{\;\;\;\nu}  \cdot g^{\a\b}  <\derm_{\mu } \derm_\a \Phi ,   \derm_\b \Phi > \\
&& \text{(using the symmetry of the metric $g$)} \\
&=& ( \derm_{\mu } g^{\mu\a} ) \cdot < \derm_\a \Phi ,   \derm_\nu \Phi > - \frac{1}{2} m^{\mu}_{\;\;\;\nu}  \cdot  ( \derm_{\mu } g^{\a\b} ) \cdot < \derm_\a \Phi ,   \derm_\b \Phi > \\
&& + < g^{\mu\a} \derm_{\mu } \derm_\a \Phi ,   \derm_\nu \Phi > +  g^{\mu\a}< \derm_\a \Phi ,  \derm_{\mu } \derm_\nu \Phi > \\
&&  -   g^{\a\b}  <\derm_{\nu } \derm_\a \Phi ,   \derm_\b \Phi > \; .
\eeaa

Now, we can compute in wave coordinates $\derm_{\nu } \derm_\a \Phi  $, and if the end result for $\derm^{\mu }T_{\mu\nu}$ gives a tensor in $\nu$, then the identity that we obtain will be true independently of the system of coordinates. In wave coordinates, the Christoffel symbols are vanishing and therefore, the two derivates commute, i.e.
\beaa
\derm_{\nu } \derm_\a \Phi  = \derm_{\a } \derm_\nu \Phi \; .
\eeaa
Thus, in wave coordinates
\beaa
&& \derm^{\mu }T^{(\bf{g})}_{\mu\nu}  \\
&=&( \derm_{\mu } g^{\mu\a} ) \cdot < \derm_\a \Phi ,   \derm_\nu \Phi > - \frac{1}{2} m^{\mu}_{\;\;\;\nu}  \cdot  ( \derm_{\mu } g^{\a\b} ) \cdot < \derm_\a \Phi ,   \derm_\b \Phi > \\
&& + < g^{\mu\a} \derm_{\mu } \derm_\a \Phi ,   \derm_\nu \Phi > +  g^{\mu\a}< \derm_\a \Phi ,  \derm_{\mu } \derm_\nu \Phi > \\
&&  -   g^{\a\b}  <\derm_{\a } \derm_\nu \Phi ,   \derm_\b \Phi > \\
&=&( \derm_{\mu } g^{\mu\a} ) \cdot < \derm_\a \Phi ,   \derm_\nu \Phi > - \frac{1}{2} m^{\mu}_{\;\;\;\nu}  \cdot  ( \derm_{\mu } g^{\a\b} ) \cdot < \derm_\a \Phi ,   \derm_\b \Phi > \\
&& + < g^{\mu\a} \derm_{\mu } \derm_\a \Phi ,   \derm_\nu \Phi > \; .
\eeaa

 Since the end result is a tensor in $\nu$, we obtain
 \beaa
\derm^{\mu }T^{(\bf{g})}_{\mu\nu} &=&( \derm_{\mu } g^{\mu\a} ) \cdot < \derm_\a \Phi ,   \derm_\nu \Phi > - \frac{1}{2} m^{\mu}_{\;\;\;\nu}  \cdot  ( \derm_{\mu } g^{\a\b} ) \cdot < \derm_\a \Phi ,   \derm_\b \Phi > \\
&& + < g^{\mu\a} \derm_{\mu } \derm_\a \Phi ,   \derm_\nu \Phi > \; .
\eeaa

 Contracting the stress-energy-momentum tensor with respect to the first index, with a vector field $X^{\nu}$, and computing the covariant divergence of $X^{\nu}T_{\mu\nu}$, one gets
\beaa
\derm^{\mu} \big( X^{\nu}T_{\mu\nu}^{(\bf{g})} \big) &=& \big( \derm^{\mu} X^{\nu}  \big) \cdot T_{\mu\nu}^{(\bf{g})} +\big( X^{\nu}  \big) \cdot \derm^{\mu} T_{\mu\nu}^{(\bf{g})} \\
&=& \derm^{\mu} X^{\nu}  \cdot T_{\mu\nu}^{(\bf{g})} +  \derm^{\mu }T^{(\bf{g})}_{\mu X}  \; .
\eeaa
We can then write 
 \bea
 \notag
&&\derm^{\mu} \big( X^{\nu}T_{\mu\nu}^{(\bf{g})} \big) \\
\notag
&=& \big( \derm^{\mu} X^{\nu}  \big) \cdot T_{\mu\nu}^{(\bf{g})} + < g^{\mu\a} \derm_{\mu } \derm_\a \Phi ,   \derm_\nu \Phi >\\
 \notag
&& +( \derm_{\mu } g^{\mu\a} ) \cdot < \derm_\a \Phi ,   \derm_\nu \Phi > - \frac{1}{2} m^{\mu}_{\;\;\;\nu}  \cdot  ( \derm_{\mu } g^{\a\b} ) \cdot < \derm_\a \Phi ,   \derm_\b \Phi > \; . \\ 
\eea

 This leads to
 
 \bea
 \notag
&& \int_{t_1}^{t_2}  \int_{\Sigma^{ext}_{\tau}} \der^{\mu} \big( X^{\nu}T_{\mu\nu}^{(\bf{g})} \big)  \cdot dv^{(\bf{m})} \\
 \notag
 &=&  \int_{t_1}^{t_2}  \int_{\Sigma^{ext}_{\tau} } \Big( \big( \derm^{\mu} X^{\nu}  \big) \cdot T_{\mu\nu}^{(\bf{g})}  + < g^{\mu\a} \derm_{\mu } \derm_\a \Phi ,   \derm_\nu \Phi >\\
  \notag
&& +( \derm_{\mu } g^{\mu\a} ) \cdot < \derm_\a \Phi ,   \derm_\nu \Phi > - \frac{1}{2} m^{\mu}_{\;\;\;\nu}  \cdot  ( \derm_{\mu } g^{\a\b} ) \cdot < \derm_\a \Phi ,   \derm_\b \Phi >   \Big) \cdot dv^{(\bf{m})}  \\
  \notag
 &=&   \int_{\Sigma^{ext}_{t_1} }  \big( X^{\mu}T_{\mu\nu}^{(\bf{g})} \big)  n^{(\bf{m}), \nu}_{\Sigma} \cdot dv^{(\bf{m})}_\Sigma  -  \int_{\Sigma^{ext}_{t_2}}  \big( X^{\mu}T_{\mu\nu}^{(\bf{g})} \big)  n^{(\bf{m}), \nu}_{\Sigma} \cdot dv^{(\bf{m})}_\Sigma \\
 && -  \int_{N_{t_1}^{t_2} }  \big( X^{\mu}T_{\mu\nu}^{(\bf{g})} \big)  n^{(\bf{m}), \nu}_{N} \cdot dv^{(\bf{m})}_N \; .
 \eea

Using the definition of $H^{\mu\nu}:=g^{\mu\nu}-m^{\mu\nu}$, we get the result
\end{proof}

We recall that $ \hat{L}^{\nu} $ is defined as in Definition \ref{definitionofwidetildeLfordivergencetheoremuse}.
 
\begin{lemma}\label{conservationlawwithoutcomputingdivergenceofTmut}
We have
  \bea
  \notag
   &&   \int_{\Sigma^{ext}_{t_2} } \Big(  - \frac{1}{2} g^{t t}< \derm_t \Phi_{V} ,   \derm_t \Phi_{V} >   + \frac{1}{2} g^{j i}  < \derm_j \Phi_{V} ,   \derm_i \Phi_{V} > \Big)    \cdot d^{n}x \\
   \notag
&& +  \int_{N_{t_1}^{t_2} }  \big( T_{\hat{L} t}^{(\bf{g})} \big)  \cdot dv^{(\bf{m})}_N \\
     \notag
   &=&   \int_{\Sigma^{ext}_{t_1} }   \Big(  - \frac{1}{2} g^{t t}< \derm_t \Phi_{V} ,   \derm_t \Phi_{V} >   + \frac{1}{2} g^{j i}  < \derm_j \Phi_{V},   \derm_i \Phi_{V} >  \Big)  \cdot d^{n}x  \\
     \notag
     &&- \int_{t_1}^{t_2}  \int_{\Sigma^{ext}_{\tau} } \Big( \derm^{\mu }T_{\mu t}^{(\bf{g})}     \Big) \cdot d\tau d^{n}x   \; .\\
 \eea

\end{lemma}

\begin{proof}
 
 Considering the metric $m$, we know by definition of $m$ being the Minkowski metric in wave coordinate $\{t, x^1, \ldots, x^n \}$, that for $X = \frac{\pa}{\pa t} $, we then have
 \beaa
 \big( \derm^{\mu}  \big( \frac{\pa}{\pa t}  \big)^{\nu} \big) \cdot T_{\mu\nu}^{(\bf{g})}   &=&  0 \; , \\
n^{(\bf{m}), \nu}_{\Sigma}  &=&  \big( \frac{\pa}{\pa t}  \big)^{\nu} \; , \\
 dv^{(\bf{m})}_\Sigma &=& dx^1 \ldots dx^n  := d^{n} x \; .
 \eeaa

Hence, the conservation law in Lemma \ref{generalconservationlawwithanyvectorfieldX}, obtained through the divergence theorem for the non-symmetric tensor $T_{\mu t}$, gives
    
 \bea\label{conservationlawforvectorfieldt}
 \notag
  \int_{t_1}^{t_2}  \int_{\Sigma^{ext}_{\tau} } \Big(   \derm^{\mu }T_{\mu t}^{(\bf{g})}    \Big) \cdot d\tau d^{n}x  &=&   \int_{\Sigma^{ext}_{t_2} }  T_{ t t}^{(\bf{g})}   \cdot d^{n}x -  \int_{\Sigma^{ext}_{t_1} }   T_{t t}^{(\bf{g})} \cdot d^{n}x  \\
  \notag
 && -  \int_{N_{t_1}^{t_2} }  \big( T_{\hat{L} t}^{(\bf{g})} \big)  \cdot dv^{(\bf{m})}_N \; .\\
 \eea

We compute
 \beaa
\notag
  T_{ t t}^{(\bf{g})}  = m_{\a t } {T^{(\bf{g})}}^{\a}_{\;\;\; t} = m_{t t } {T^{(\bf{g})}}^{t}_{\;\;\; t} = - {T^{(\bf{g})}}^{t}_{\;\;\; t} \; .
 \eeaa
 
We compute further,
 \bea\label{evlautingTttforourenergymomentumtensor}
\notag
&& {T^{(\bf{g})}}^{t}_{\;\;\; t} \\
\notag
 &=&  g^{t\a}< \derm_\a \Phi ,   \derm_t \Phi > - \frac{1}{2} m^{t}_{\;\;\; t}  \cdot g^{\a\b}  < \derm_\a \Phi ,   \derm_\b \Phi > \; \\
 \notag
&=&  g^{t\a}< \derm_\a \Phi ,   \derm_t \Phi > + \frac{1}{2} m_{t t}  \cdot \big( g^{t \b}  < \derm_t \Phi ,   \derm_\b \Phi > +g^{j \b}  < \derm_j \Phi ,   \derm_\b \Phi > \big)   \\
\notag
&=&  \frac{1}{2} g^{t\a}< \derm_\a \Phi ,   \derm_t \Phi > - \frac{1}{2} g^{j \b}  < \derm_j \Phi ,   \derm_\b \Phi >     \\
\notag
&=&  \frac{1}{2} g^{t t}< \derm_t \Phi ,   \derm_t \Phi >   + \frac{1}{2} g^{tj}< \derm_j \Phi ,   \derm_t \Phi >  - \frac{1}{2} g^{j t}  < \derm_j \Phi ,   \derm_t \Phi > \\
\notag
&&  - \frac{1}{2} g^{j i}  < \derm_j \Phi ,   \derm_i \Phi >   \; . \\
 \eea

Consequently, the conservation law \ref{conservationlawforvectorfieldt}, with the vector field $X= \frac{\pa}{ \pa t}$\,, gives the stated result.
 
 \end{proof}
 
 \begin{corollary}
 We have
    \bea
  \notag
   &&   \int_{\Sigma^{ext}_{t_2} } \Big(  - \frac{1}{2} g^{t t}< \derm_t \Phi_{V} ,   \derm_t \Phi_{V} >   + \frac{1}{2} g^{j i}  < \derm_j \Phi_{V} ,   \derm_i \Phi_{V} > \Big)    \cdot d^{n}x \\
   \notag
&& +  \int_{N_{t_1}^{t_2} }  \big( T_{\hat{L} t}^{(\bf{g})} \big)  \cdot dv^{(\bf{m})}_N \\
     \notag
   &=&   \int_{\Sigma^{ext}_{t_1} }   \Big(  - \frac{1}{2} g^{t t}< \derm_t \Phi_{V} ,   \derm_t \Phi_{V} >   + \frac{1}{2} g^{j i}  < \derm_j \Phi_{V} ,   \derm_i \Phi_{V} >  \Big)  \cdot d^{n}x  \\
     \notag
     &&- \int_{t_1}^{t_2}  \int_{\Sigma^{ext}_{\tau} } \Big(  < g^{\mu\a} \derm_{\mu } \derm_\a \Phi_{V} ,   \derm_t \Phi_{V} >  \\
       \notag
&& +  \frac{1}{2} \cdot   ( \derm_{t } g^{t \a} ) \cdot < \derm_\a \Phi_{V} ,   \derm_t \Phi_{V} > +( \derm_{j } g^{j \a} ) \cdot < \derm_\a \Phi_{V} ,   \derm_t \Phi_{V} > \\
 \notag
&& - \frac{1}{2} \cdot  ( \derm_{t } g^{j \b} ) \cdot < \derm_j \Phi_{V} ,   \derm_\b \Phi_{V} >     \Big) \cdot d\tau d^{n}x   \; .\\
 \eea
 
 \end{corollary}
 
 \begin{proof}
 We compute
  \bea\label{expressionofdivergenceofTmutwithoutdecompisinginaframe}
  \notag
&& \derm^{\mu }T_{\mu t}^{(\bf{g})}   \\
\notag
 &=& < g^{\mu\a} \derm_{\mu } \derm_\a \Phi ,   \derm_t \Phi >  +( \derm_{\mu } g^{\mu\a} ) \cdot < \derm_\a \Phi ,   \derm_t \Phi > \\
 \notag
&&- \frac{1}{2} m^{\mu}_{\;\;\; t}  \cdot  ( \derm_{\mu } g^{\a\b} ) \cdot < \derm_\a \Phi ,   \derm_\b \Phi >  \; .\\
 \eea

Since $ m^{\mu}_{\;\;\; t} = m^{\mu\a} \cdot m_{\a t}  =m^{\mu t} \cdot m_{t t} = - m^{\mu t} $, we compute further by decomposing the sum in wave coordinates,
\bea\label{evaluationofthecovariantdivergenceofourenergymomentumtensor}
\notag
&& \derm^{\mu }T_{\mu t}^{(\bf{g})}    \\
\notag
 &=& < g^{\mu\a} \derm_{\mu } \derm_\a \Phi ,   \derm_t \Phi >  \\
\notag
&& +( \derm_{t } g^{t \a} ) \cdot < \derm_\a \Phi ,   \derm_t \Phi > +( \derm_{j } g^{j \a} ) \cdot < \derm_\a \Phi ,   \derm_t \Phi > \\
 \notag
&&+ \frac{1}{2} m^{t t} \cdot  ( \derm_{t} g^{\a\b} ) \cdot < \derm_\a \Phi ,   \derm_\b \Phi >  \\ 
\notag
&=& < g^{\mu\a} \derm_{\mu } \derm_\a \Phi ,   \derm_t \Phi >  \\
\notag
&& +( \derm_{t } g^{t \a} ) \cdot < \derm_\a \Phi ,   \derm_t \Phi > +( \derm_{j } g^{j \a} ) \cdot < \derm_\a \Phi ,   \derm_t \Phi > \\
 \notag
&& - \frac{1}{2} \cdot  ( \derm_{t } g^{t\b} ) \cdot < \derm_t \Phi ,   \derm_\b \Phi > - \frac{1}{2} \cdot  ( \derm_{t } g^{j \b} ) \cdot < \derm_j \Phi ,   \derm_\b \Phi > \\ 
\notag
&=& < g^{\mu\a} \derm_{\mu } \derm_\a \Phi ,   \derm_t \Phi >  \\
\notag
&& +  \frac{1}{2} \cdot   ( \derm_{t } g^{t \a} ) \cdot < \derm_\a \Phi ,   \derm_t \Phi > +( \derm_{j } g^{j \a} ) \cdot < \derm_\a \Phi ,   \derm_t \Phi > \\
 \notag
&& - \frac{1}{2} \cdot  ( \derm_{t } g^{j \b} ) \cdot < \derm_j \Phi ,   \derm_\b \Phi > \, .\\
\eea

 Injecting in Lemma \ref{conservationlawwithoutcomputingdivergenceofTmut}, we obtain the desired result.

 \end{proof}

 \subsection{The weighted energy estimate in the exterior for $g^{\mu\nu} \derm_{\mu} \derm_{\nu} \Phi$ }\

\begin{definition}\label{definitionoftheparmeterq}
We define
\bea
q := r - t \; ,
\eea
where $t$ and $r$ are defined using the wave coordinates as explained in \cite{G4}.
\end{definition}

\begin{definition}\label{defoftheweightw}
We define
\beaa
w(q):=\begin{cases} 1+|q|)^{1+2\gamma} \quad\text{when }\quad q>0 , \\
         1 \,\quad\text{when }\quad q<0 . \end{cases}
\eeaa
for some $\gamma > 0$\,.

\end{definition}

\begin{definition}\label{defwidehatw}
We define $\widehat{w}$ by 
\beaa
\widehat{w}(q)&:=&\begin{cases} (1+|q|)^{1+2\gamma} \quad\text{when }\quad q>0 , \\
        (1+|q|)^{2\mu}  \,\quad\text{when }\quad q<0 , \end{cases} \\
        &=&\begin{cases} (1+ q)^{1+2\gamma} \quad\text{when }\quad q>0 , \\
      (1 - q)^{2\mu}  \,\quad\text{when }\quad q<0 ,\end{cases} 
\eeaa

for $\ga > 0$ and $\mu < 0$\,. Note that the factor $\mu \neq 0$ is constructed so that for $q<0$, the derivative $\frac{\pa \widehat{w}}{\pa q}$ is non-vanishing, so as to have a control on certain tangential derivatives, which is needed for the case $n=3$, which we will treat in a paper the follows (see Corollary \ref{TheenerhyestimatewithtermsinvolvingHandderivativeofHandwithwandhatw}) -- yet, we will not use this control here in the case of $n=4$. This being said, note that the definition of $\widehat{w}$\,, is also so that for $\ga \neq - \frac{1}{2} $ and $\mu \neq 0$ (which is assumed here), we would have 
\beaa
\widehat{w}^{\prime}(q) \sim \frac{\widehat{w}(q)}{(1+|q|)} \; ,
\eeaa
(see Lemma \ref{derivativeoftildwandrelationtotildew}) -- this is will determine the kind of control that we will have on the tangential derivatives, control that we will use in the next paper for space-dimension $n=3$. 

\end{definition}

\begin{remark}
We take $\mu < 0$ (instead of $\mu > 0$), because we want the derivative $\frac{\pa \widehat{w}}{\pa q} > 0$\,, as we will see that this is what we need in order to obtain an energy estimate on the fields (see Corollary \ref{TheenerhyestimatewithtermsinvolvingHandderivativeofHandwithwandhatw}). In other words, $\mu < 0$ is a necessary condition to ensure that $\widehat{w}^{\prime} (q)$ enters with the right sign in the energy estimate (see \eqref{estimatetoreferencetoexplainthedefinitionoftildewandhatw}).
\end{remark}

\begin{definition}\label{defwidetildew}
We define $\widetilde{w}$ by 
\beaa
\widetilde{w} ( q)&:=&  \widehat{w}(q) + w(q) \\
&:=&\begin{cases} 2 (1+|q|)^{1+2\gamma} \quad\text{when }\quad q>0 , \\
       1+  (1+|q|)^{2\mu}  \,\quad\text{when }\quad q<0 . \end{cases} \\    
\eeaa
Note that the definition of $\widetilde{w}$ is constructed so that Lemma \ref{equaivalenceoftildewandtildeandofderivativeoftildwandderivativeofhatw} holds, which we need in order to obtain \eqref{estimatetoreferencetoexplainthedefinitionoftildewandhatw}.
\end{definition}

\begin{lemma}\label{equaivalenceoftildewandtildeandofderivativeoftildwandderivativeofhatw}
We have
\beaa
\widetilde{w}^{\prime}  &\sim & \widehat{w}^{\prime } \; .
\eeaa
Furthermore, for $\mu < 0$\,, we have
\beaa
\widetilde{w} ( q)& \sim & w(q) \; .
\eeaa

\end{lemma}

\begin{proof}
We compute the derivative with respect to $q$\,,
\beaa
\widetilde{w}^{\prime} &=&   \widehat{w}^{\prime} ( q) + w^{\prime} ( q) \\
&=& \begin{cases} 2 \cdot \widehat{w}^{\prime} ( q) \quad\text{when }\quad q>0 , \\
        \widehat{w}^{\prime }( q)  \,\quad\text{when }\quad q<0 . \end{cases} \\ 
\eeaa
Consequently,
\beaa
\widetilde{w}^{\prime}  &\sim & \widehat{w}^{\prime } \; .
\eeaa

Now, on one hand, since $ \widehat{w} \geq 0$, we have 
\beaa
\widetilde{w} ( q)&\geq&w(q) \; .
\eeaa
On the other hand, since $\mu < 0$\,, we have
\beaa
\widetilde{w} ( q)& = & \begin{cases} 2 (1+|q|)^{1+2\gamma} \quad\text{when }\quad q>0 , \\
       1+  (1+|q|)^{2\mu}  \,\quad\text{when }\quad q<0 . \end{cases} \\    
       &\leq&\begin{cases} 2 (1+|q|)^{1+2\gamma} \quad\text{when }\quad q>0 , \\
       2  \,\quad\text{when }\quad q<0 . \end{cases} \\    
       &\leq& 2 w(q)
\eeaa
Thus
\beaa
\widetilde{w} ( q)& \sim & w(q) \; .
\eeaa
\end{proof}

\begin{lemma}\label{derivativeoftildwandrelationtotildew}
Let $\widehat{w}$ be defined as in Definition \ref{defwidehatw}.
We have, for $\ga \neq - \frac{1}{2} $ and $\mu \neq 0$,  
\beaa
\widehat{w}^{\prime}(q) \sim \frac{\widehat{w}(q)}{(1+|q|)} \; .
\eeaa

\end{lemma}
\begin{proof}

We have
\beaa
\widehat{w}(q) &=&\begin{cases} (1+ q)^{1+2\gamma} \quad\text{when }\quad q>0 , \\
        (1 - q)^{2\mu}  \,\quad\text{when }\quad q<0 . \end{cases} 
\eeaa
We compute,
\beaa
\widehat{w}^{\prime}(q)&=&\begin{cases} (1+2\gamma)(1+|q|)^{2\gamma} \quad\text{when }\quad q>0 , \\
         - 2\mu(1+|q|)^{2\mu-1} \,\quad\text{when }\quad q<0 . \end{cases} \\
          &=&\begin{cases}         (1+2\gamma)   \frac{w(q)}{(1+|q|)}   \quad\text{when }\quad q>0 , \\
         -  2\mu      \frac{w(q)}{(1+|q|)}    \,\quad\text{when }\quad q<0 . \end{cases} \\
\eeaa
Thus,
\beaa
\min \{ (1+2\gamma), -2\mu\}  \cdot \frac{\widehat{w}(q)}{(1+|q|)} 
 \leq \widehat{w}^{\prime}(q) \leq \max \{ (1+2\gamma), -2\mu\} \cdot \frac{\widehat{w}(q)}{(1+|q|)} \; ,
\eeaa
and hence, for  $\min \{ (1+2\gamma), -2\mu\}  \neq 0 $ and $\max \{ (1+2\gamma), -2\mu\} \neq 0$\,,  
\beaa
\widehat{w}^{\prime}(q) \sim \frac{\widehat{w}(q)}{(1+|q|)} \; .
\eeaa
\end{proof}

We now establish a conservation law with the weight $\widetilde{w}$. 

   \begin{lemma}\label{weightedconservationlawintheexteriorwiththeenergymomuntumtensorcontarctedwithvectordt}
We have
    \bea  \notag
  && \int_{N_{t_1}^{t_2} }  \big( T_{\hat{L} t}^{(\bf{g})} \big)  \cdot  \widetilde{w} (q) \cdot dv^{(\bf{m})}_N  +   \int_{\Sigma^{ext}_{t_2} }   T_{t t}^{(\bf{g})}  \cdot \widetilde{w} (q) \cdot d^{n}x \\
   \notag
     &=&   \int_{\Sigma^{ext}_{t_1} }  T_{ t t}^{(\bf{g})}   \cdot \widetilde{w} (q)  \cdot d^{n}x  -  \int_{t_1}^{t_2}  \int_{\Sigma^{ext}_{\tau} }  \Big( T_{t  t}^{(\bf{g})} +   T_{r  t}^{(\bf{g})} \Big) \cdot d\tau \cdot \widetilde{w} ^\prime (q)  \cdot  d^{n}x \\
     \notag
   &&  -  \int_{t_1}^{t_2}  \int_{\Sigma^{ext}_{\tau} } \Big(   \derm^{\mu }T_{\mu t}^{(\bf{g})}    \Big) \cdot d\tau \cdot \widetilde{w} (q) \cdot d^{n}x \; .  \\
 \eea
 
 \end{lemma}
 
 \begin{proof}
 Considering again the Minkowski metric $m$ in the coordinates $\{t, x^1, \ldots, x^n \}$, and instead of contracting $T_{ \mu\nu}^{(\bf{g})}$  with $\frac{\pa}{\pa t} $ with respect to the second component $\nu$, we contract with the weighted vector
 \bea\label{The weightedvectorproportionaltodt}
 X = \widetilde{w} (q) \frac{\pa}{\pa t} \;.
 \eea
 By then, we have in the coordinates $\mu, \nu \in \{t, x^1, \ldots, x^n \}$, 
  \beaa
 \big( \derm^{\mu}  \big( \widetilde{w} (q) \frac{\pa}{\pa t}  \big)^{\nu} \big) &=& \widetilde{w} ^\prime(q) \cdot \derm^{\mu} (q) \cdot  \big( \frac{\pa}{\pa t}  \big)^{\nu} + \widetilde{w} q) \derm^{\mu}  \big( \frac{\pa}{\pa t}  \big)^{\nu} \\
&=&  \widetilde{w} ^\prime(q) \cdot m^{\mu\a} \derm_{\a} (q) \cdot  \big( \frac{\pa}{\pa t}  \big)^{\nu} \; .
 \eeaa
 
 For $\mu = t = x^0$, we have since $q = r-t$, 
 \beaa
m^{\mu\a} \derm_{\a} (q) = m^{tt} \derm_{t} (q) = - (-1) = 1 \; .
\eeaa

 For $\mu = x^j$, we have since $q = r-t$, 
  \beaa
m^{\mu\a} \derm_{\a} (q) = m^{jj} \derm_{j} (q) =  \frac{x^{j}}{r}\; .
\eeaa
Thus,
 \bea
 \notag
 \big( \derm^{\mu} \big( \widetilde{w} (q)  \frac{\pa}{\pa t}  \big)^{\nu} \big) \cdot T_{\mu\nu}^{(\bf{g})}   &=&  \Big( \widetilde{w} ^\prime(q) \cdot m^{\mu\a} \derm_{\a} (q) \cdot  \big( \frac{\pa}{\pa t}  \big)^{\nu} \Big) \cdot T_{\mu\nu}^{(\bf{g})} \\
  \notag
 &=&   \widetilde{w} ^\prime(q) \cdot T_{t  t}^{(\bf{g})} +  \widetilde{w} ^\prime(q) \cdot \frac{x^{j}}{r} \cdot T_{j  t}^{(\bf{g})} \\
 &=&   \widetilde{w} ^\prime(q) \cdot \Big( T_{t  t}^{(\bf{g})} +   T_{r  t}^{(\bf{g})} \Big) \; .
 \eea
We still have
 \beaa
n^{(\bf{m}), \nu}_{\Sigma}  &=&  \big( \frac{\pa}{\pa t}  \big)^{\nu} \; , \\
 dv^{(\bf{m})}_\Sigma &=& dx^1 \ldots dx^n  := d^{n} x \; \\
  n^{(\bf{m}), \nu}_{N} &=& \hat{L}^{\nu} \; .
 \eeaa
Consequently, the conservation law with the weighted vector $w(q) \frac{\pa}{\pa t}$ contracted with the second component of the non-symmetric tensor $T_{\mu\nu}^{(\bf{g})}$, gives the following equality
 \bea
 \notag
  && \int_{t_1}^{t_2}  \int_{\Sigma^{ext}_{\tau} }  \Big( T_{t  t}^{(\bf{g})} +   T_{r  t}^{(\bf{g})} \Big) \cdot d\tau \cdot \widetilde{w}^\prime (q)\cdot  d^{n}x +  \int_{t_1}^{t_2}  \int_{\Sigma^{ext}_{\tau} } \Big(   \derm^{\mu }T_{\mu t}^{(\bf{g})}    \Big) \cdot d\tau \cdot \widetilde{w}(q) \cdot  d^{n}x \\
   \notag
     &=&   \int_{\Sigma^{ext}_{t_1} }  T_{ t t}^{(\bf{g})}   \cdot \widetilde{w}(q)  \cdot d^{n}x -  \int_{\Sigma^{ext}_{t_2} (q)}   T_{t t}^{(\bf{g})}  \cdot \widetilde{w}(q) \cdot d^{n}x  \\
  \notag
 &&- \int_{N_{t_1}^{t_2} }  \big( T_{\hat{L} t}^{(\bf{g})} \big)  \cdot  \widetilde{w}(q) \cdot dv^{(\bf{m})}_N  \; .
 \eea

 \end{proof}

 \begin{corollary}\label{WeightedenergyestimateusingHandnosmallnessyet}
 We have
    \beaa  \notag
  && \int_{N_{t_1}^{t_2} }  \big( T_{\hat{L} t}^{(\bf{g})} \big)  \cdot  \widetilde{w}(q) \cdot dv^{(\bf{m})}_N  \\
  &&+   \int_{\Sigma^{ext}_{t_2} }  \big(    - \frac{1}{2} (m^{t t} + H^{tt} ) < \derm_t \Phi_{V} ,   \derm_t \Phi_{V} >   \\
  &&\quad\quad\quad \quad  + \frac{1}{2} (m^{j i} +H^{j i} )   < \derm_j \Phi_{V} ,   \derm_i \Phi_{V} > \big)  \cdot \widetilde{w}(q) \cdot d^{n}x \\
   \notag
     &=&   \int_{\Sigma^{ext}_{t_1} } \big(    - \frac{1}{2} (m^{t t} + H^{tt} ) < \derm_t \Phi_{V} ,   \derm_t \Phi_{V} >    \\
  &&\quad\quad\quad  + \frac{1}{2} (m^{j i} +H^{j i} )   < \derm_j \Phi_{V} ,   \derm_i \Phi_{V} > \big)   \cdot \widetilde{w}(q)  \cdot d^{n}x  \\
     && -  \int_{t_1}^{t_2}  \int_{\Sigma^{ext}_{\tau} }  \Big( T_{t  t}^{(\bf{g})} +   T_{r  t}^{(\bf{g})} \Big) \cdot d\tau \cdot \widetilde{w}^\prime (q) d^{n}x \\
     \notag
   &&  -  \int_{t_1}^{t_2}  \int_{\Sigma^{ext}_{\tau} } \Big(  < g^{\mu\a} \derm_{\mu } \derm_\a \Phi_{V} ,   \derm_t \Phi_{V} >  +( \derm_{\mu } H^{\mu\a} ) \cdot < \derm_\a \Phi_{V} ,   \derm_t \Phi_{V} > \\
 \notag
&&- \frac{1}{2} m^{\mu}_{\;\;\; t}  \cdot  ( \derm_{\mu } H^{\a\b} ) \cdot < \derm_\a \Phi_{V} ,   \derm_\b \Phi_{V} >   \Big) \cdot d\tau \cdot \widetilde{w}(q) d^{n}x \; .  \\
 \eeaa
 \end{corollary}
 
 \begin{proof}
We want to evaluate the terms in \eqref{weightedconservationlawintheexteriorwiththeenergymomuntumtensorcontarctedwithvectordt}. We have shown based on \eqref{evlautingTttforourenergymomentumtensor}, that
   \beaa
 \notag
  T_{ t t}^{(\bf{g})}  &=&  - \frac{1}{2} g^{t t}< \derm_t \Phi ,   \derm_t \Phi >   + \frac{1}{2} g^{j i}  < \derm_j \Phi ,   \derm_i \Phi >   \\
&=&    - \frac{1}{2} (m^{t t} + H^{tt} ) < \derm_t \Phi ,   \derm_t \Phi >   + \frac{1}{2} (m^{j i} +H^{j i} )   < \derm_j \Phi ,   \derm_i \Phi > \; ,
  \eeaa
and based on \eqref{expressionofdivergenceofTmutwithoutdecompisinginaframe}, that
 \beaa
  \notag
&& \derm^{\mu }T_{\mu t}^{(\bf{g})}   \\
\notag
 &=& < g^{\mu\a} \derm_{\mu } \derm_\a \Phi ,   \derm_t \Phi >  +( \derm_{\mu } H^{\mu\a} ) \cdot < \derm_\a \Phi ,   \derm_t \Phi > \\
 \notag
&&- \frac{1}{2} m^{\mu}_{\;\;\; t}  \cdot  ( \derm_{\mu } H^{\a\b} ) \cdot < \derm_\a \Phi ,   \derm_\b \Phi > \; . \\
\eeaa
Injecting these in Lemma \eqref{weightedconservationlawintheexteriorwiththeenergymomuntumtensorcontarctedwithvectordt}, we get the stated result.

  \end{proof}

Now, we would like to evaluate in Corollary \ref{WeightedenergyestimateusingHandnosmallnessyet} the term with the weight $w^\prime (q)$\,.

 \begin{lemma}\label{expressionofTttplusTtr}
 We have
    \bea
   \notag
 && T_{ t t}^{(\bf{g})}  +  T_{ r t}^{(\bf{g})} \\
    \notag
 &=&  \frac{1}{2} \Big(  | \derm_t  \Phi_{V}  + \derm_r \Phi_{V}  |^2  +  \de^{ij}  | \derm_i - \frac{x_i}{r} \derm_{r}  )\Phi_{V}  |^2 \Big)   \\
     \notag
 &&  - \frac{1}{2} H^{t t}< \derm_t \Phi_{V}  ,   \derm_t \Phi_{V}  >   + \frac{1}{2} H^{j i}  < \derm_j \Phi_{V}  ,   \derm_i \Phi_{V}  >  \\
    \notag
&&  + H^{r t}< \derm_t \Phi_{V}  ,   \derm_t \Phi_{V}  > + H^{r j}< \derm_j \Phi_{V}  ,   \derm_t \Phi_{V}  >  \; . \\
  \eea
  
 \end{lemma}

   \begin{proof}
   
 We compute
  \beaa
T_{ r t}^{(\bf{g})} = m_{rr} \cdot {T^{(\bf{g})}}^{r}_{\;\;\; t} = {T^{(\bf{g})}}^{r}_{\;\;\; t}  \; .
 \eeaa

\bea
\notag
 {T^{(\bf{g})}}^{r}_{\;\;\; t}   &=&  g^{r\a}< \derm_\a \Phi ,   \derm_t \Phi > - \frac{1}{2} m^{r}_{\;\;\; t}  \cdot g^{\a\b}  < \derm_\a \Phi ,   \derm_\b \Phi >  \\
&=&  g^{r\a}< \derm_\a \Phi ,   \derm_t \Phi > \; .
 \eea

Thus,
\bea
\notag
 T_{ r t}^{(\bf{g})} =  g^{r t}< \derm_t \Phi ,   \derm_t \Phi > + g^{r j}< \derm_j \Phi ,   \derm_t \Phi >   \; . \\
 \eea

  Consequently,
      \beaa
      \notag
 && T_{ t t}^{(\bf{g})}  +  T_{ r t}^{(\bf{g})} \\
       \notag
&=& - \frac{1}{2} g^{t t}< \derm_t \Phi ,   \derm_t \Phi >   + \frac{1}{2} g^{j i}  < \derm_j \Phi ,   \derm_i \Phi >  \\
&&  + g^{r t}< \derm_t \Phi ,   \derm_t \Phi > + g^{r j}< \derm_j \Phi ,   \derm_t \Phi >  \\
&=& - \frac{1}{2} m^{t t}< \derm_t \Phi ,   \derm_t \Phi >   + \frac{1}{2} m^{j i}  < \derm_j \Phi ,   \derm_i \Phi >  \\
&& + m^{r t}< \derm_t \Phi ,   \derm_t \Phi > + m^{r j}< \derm_j \Phi ,   \derm_t \Phi >  \\
&& - \frac{1}{2} H^{t t}< \derm_t \Phi ,   \derm_t \Phi >   + \frac{1}{2} H^{j i}  < \derm_j \Phi ,   \derm_i \Phi >  \\
&&  + H^{r t}< \derm_t \Phi ,   \derm_t \Phi > + H^{r j}< \derm_j \Phi ,   \derm_t \Phi >   \; .
  \eeaa

 Thus,
       \beaa
      \notag
 && T_{ t t}^{(\bf{g})}  +  T_{r t}^{(\bf{g})} \\
       \notag
&=&  \frac{1}{2} < \derm_t \Phi ,   \derm_t \Phi >   + \frac{1}{2} \de^{j i}  < \derm_j \Phi ,   \derm_i \Phi >   + m^{r j}< \derm_j \Phi ,   \derm_t \Phi >  \\
&& - \frac{1}{2} H^{t t}< \derm_t \Phi ,   \derm_t \Phi >   + \frac{1}{2} H^{j i}  < \derm_j \Phi ,   \derm_i \Phi >  \\
&&  + H^{r t}< \derm_t \Phi ,   \derm_t \Phi > + H^{r j}< \derm_j \Phi ,   \derm_t \Phi > \; .
\eeaa
 Yet, 
  \beaa
m^{r j} = m^{rr} m^{ij} m_{r i} = m_{r j} = \frac{x^j}{r} 
  \eeaa
   Therefore, 
   \beaa
   m^{r j}< \derm_j \Phi ,   \derm_t \Phi > &=& \frac{x^j}{r} < \derm_j \Phi ,   \derm_t \Phi > = < \derm_r \Phi ,   \derm_t \Phi > \; .
   \eeaa
   
Therefore, we have
\bea
\notag
 T_{ t t}^{(\bf{g})}  +  T_{  r t}^{(\bf{g})} &=&  \frac{1}{2} | \derm \Phi |^2  + < \derm_r \Phi ,   \derm_t \Phi >  \\
 \notag
&& - \frac{1}{2} H^{t t}< \derm_t \Phi ,   \derm_t \Phi >   + \frac{1}{2} H^{j i}  < \derm_j \Phi ,   \derm_i \Phi >  \\
\notag
&&  + H^{r t}< \derm_t \Phi ,   \derm_t \Phi > + H^{r j}< \derm_j \Phi ,   \derm_t \Phi > \; . \\
  \eea

However, we have shown that for a scalar, we have $ \de^{ij}  < \pa_i \Phi , \pa_j \Phi > = \de^{ij}  < ( \pa_i - \frac{x_i}{r} \pa_{r}  )\Phi , (\pa_j - \frac{x_j}{r} \pa_{r} ) \Phi > +  < \pa_{r}  \Phi, \pa_{r}  \Phi > $. However, since $\derm$ is the Minkowski covariant derivative, computing the trace with respect to wave coordinates $\{t, x^1, \ldots, x^n \}$, we get
\bea
\notag
 && \de^{ij}  < \derm_i \Phi , \derm_j \Phi >  \\
 \notag
 &=& \de^{ij}  < ( \derm_i - \frac{x_i}{r} \derm_{r}  )\Phi , (\derm_j - \frac{x_j}{r} \derm_{r} ) \Phi > +  < \derm_{r}  \Phi, \derm_{r}  \Phi > \; . \\
\eea

 Hence,
 
 \bea
\notag
&&  < \derm_t  \Phi + \derm_r \Phi, \derm_t  \Phi + \derm_r \Phi >     +  \de^{ij}  < ( \derm_i - \frac{x_i}{r} \derm_{r}  )\Phi , (\derm_j - \frac{x_j}{r} \derm_{r} ) \Phi >  \\
\notag
 &=& < \derm_t  \Phi, \derm_t  \Phi >  + < \derm_r  \Phi, \derm_r  \Phi >  + 2  < \derm_t  \Phi, \derm_r  \Phi > \\
 \notag
 && + \de^{ij}  < \derm_i \Phi , \derm_j \Phi > -  < \derm_{r}  \Phi, \derm_{r}  \Phi > \\
 \notag
 &=&  < \derm_t  \Phi, \derm_t  \Phi >   + 2  < \derm_t  \Phi, \derm_r  \Phi > + \de^{ij}  < \derm_i \Phi , \derm_j \Phi > \\
 &=& |\derm \Phi |^2  + 2  < \derm_t  \Phi, \derm_r  \Phi > \; .
   \eea
   
   Therefore,
   \beaa
   \notag
 && T_{ t t}^{(\bf{g})}  +  T_{ r t}^{(\bf{g})} \\
    \notag
  &=&  \frac{1}{2} \Big(  < \derm_t  \Phi + \derm_r \Phi, \derm_t  \Phi + \derm_r \Phi >    \\
     \notag
  && +  \de^{ij}  < ( \derm_i - \frac{x_i}{r} \derm_{r}  )\Phi , (\derm_j - \frac{x_j}{r} \derm_{r} ) \Phi > \Big)   \\
     \notag
 &&  - \frac{1}{2} H^{t t}< \derm_t \Phi ,   \derm_t \Phi >   + \frac{1}{2} H^{j i}  < \derm_j \Phi ,   \derm_i \Phi >  \\
    \notag
&&  + H^{r t}< \derm_t \Phi ,   \derm_t \Phi > + H^{r j}< \derm_j \Phi ,   \derm_t \Phi > \; . \\
  \eeaa
  
  \end{proof}

  We recapitulate the following lemma that he fact that we proved in \cite{G4}.
  
  \begin{lemma}\label{howtogetthedesirednormintheexpressionofenergyestimate}
Assume that the perturbation of the Minkowski metric is such that $ H^{\mu\nu} = g^{\mu\nu}-m^{\mu\nu}$ is bounded by a constant $C < \frac{1}{n}$\;, where $n$ is the space dimension, i.e.
\bea\label{AssumptiononHforgettingthenormintheexpressionofenergyestimate}
| H| \leq C < \frac{1}{n} \; .
\eea
Then we have
 \beaa
 | \derm \Phi_{V}  |^2 \sim  - (m^{t t} + H^{t t} ) < \derm_t \Phi_{V} , \derm_t \Phi_{V}  > +   (m^{ij} + H^{ij} ) < \derm_i \Phi_{V}  , \derm_j \Phi_{V}  > \; ,
  \eeaa
where the scalar product of the partial derivatives is as in Definition \ref{definitionofthescalarproductoftwopartialderivatives}.

\end{lemma}

 \begin{lemma}\label{TheenerhyestimatewithtermsinvolvingHandderivativeofH}
 For $ H^{\mu\nu} = g^{\mu\nu}-m^{\mu\nu}$ satisfying 
\bea
| H| < \frac{1}{n}\; ,
\eea
where $n$ is the space dimension, and for $\Phi$ decaying sufficiently fast at spatial infinity, we have
   \beaa
 &&     \int_{\Sigma^{ext}_{t_2} }  |\derm \Phi_{V}  |^2     \cdot \widetilde{w}(q)  \cdot d^{n}x   \\
 \notag
 &&+ \int_{t_1}^{t_2}  \int_{\Sigma^{ext}_{\tau} }  \Big(    \frac{1}{2} \Big(  | \derm_t  \Phi_{V}  + \derm_r \Phi_{V}  |^2  +  \de^{ij}  | ( \derm_i - \frac{x_i}{r} \derm_{r}  )\Phi_{V}  |^2 \Big)  \cdot d\tau \cdot \widetilde{w}^\prime (q) d^{n}x \\
  \notag
   &\les &       \int_{\Sigma^{ext}_{t_1} }  |\derm \Phi_{V}  |^2     \cdot \widetilde{w}(q)  \cdot d^{n}x   + \int_{t_1}^{t_2}  \int_{\Sigma^{ext}_{\tau} }    | H |  \cdot  | \derm \Phi_{V}  |^2     \cdot \widetilde{w}^\prime (q) \cdot d^{n}x \cdot d\tau\\
     \notag
     &&+ \int_{t_1}^{t_2}  \int_{\Sigma^{ext}_{\tau} }  \Big(   | g^{\mu\a} \derm_{\mu } \derm_\a \Phi_{V}  | \cdot |  \derm \Phi_{V}  |  + | \derm H  | \cdot | \derm \Phi_{V}  |^2       \Big)    \cdot \widetilde{w}(q) \cdot d^{n}x \cdot d\tau  \; .\\
 \eeaa

 \end{lemma}

\begin{proof}

By injecting the expression obtained in Lemma \ref{expressionofTttplusTtr} and the expression obtained in \ref{evaluationofthecovariantdivergenceofourenergymomentumtensor} using \eqref{expressionofdivergenceofTmutwithoutdecompisinginaframe}, in Corollary  \ref{WeightedenergyestimateusingHandnosmallnessyet}, we obtain the following conservation law 

  \bea
  \notag
   &&   \int_{\Sigma^{ext}_{t_2} } \Big(  - \frac{1}{2} (m^{t t} + H^{tt} ) < \derm_t \Phi ,   \derm_t \Phi >   + \frac{1}{2} (m^{j i} +H^{j i} )   < \derm_j \Phi ,   \derm_i \Phi > \Big)  \cdot \widetilde{w}(q)  \cdot d^{n}x \\
   \notag
&&  + \int_{t_1}^{t_2}  \int_{\Sigma^{ext}_{\tau} }  \Big(    \frac{1}{2} \Big(  < \derm_t  \Phi + \derm_r \Phi, \derm_t  \Phi + \derm_r \Phi >    \\
   \notag
  && +  \de^{ij}  < ( \derm_i - \frac{x_i}{r} \derm_{r}  )\Phi , (\derm_j - \frac{x_j}{r} \derm_{r} ) \Phi > \Big)   \\
    \notag
 &&  - \frac{1}{2} H^{t t}< \derm_t \Phi ,   \derm_t \Phi >   + \frac{1}{2} H^{j i}  < \derm_j \Phi ,   \derm_i \Phi >  \\
   \notag
&&  + H^{r t}< \derm_t \Phi ,   \derm_t \Phi > + H^{r j}< \derm_j \Phi ,   \derm_t \Phi >   \Big) \cdot d\tau \cdot \widetilde{w}^\prime (q) d^{n}x \\
     \notag
   &=&   \int_{\Sigma^{ext}_{t_1} }   \Big(  - \frac{1}{2}  (m^{t t} + H^{tt} ) < \derm_t \Phi ,   \derm_t \Phi >   + \frac{1}{2} (m^{j i} +H^{j i} )  < \derm_j \Phi ,   \derm_i \Phi >  \Big)  \cdot \widetilde{w}(q)\cdot d^{n}x  \\
     \notag
     &&- \int_{t_1}^{t_2}  \int_{\Sigma^{ext}_{\tau} } \Big(  < g^{\mu\a} \derm_{\mu } \derm_\a \Phi ,   \derm_t \Phi > +  \frac{1}{2} \cdot   ( \derm_{t } g^{t \a} ) \cdot < \derm_\a \Phi ,   \derm_t \Phi > \\
     \notag
     && +( \derm_{j } g^{j \a} ) \cdot < \derm_\a \Phi ,   \derm_t \Phi >  - \frac{1}{2} \cdot  ( \derm_{t } g^{j \b} ) \cdot < \derm_j \Phi ,   \derm_\b \Phi >     \Big) \cdot d\tau \cdot \widetilde{w}(q) d^{n}x  \\
     \notag
&&     -  \int_{N_{t_1}^{t_2} }  \big( T_{\hat{L} t}^{(\bf{g})} \big)  \cdot  \widetilde{w}(q) \cdot dv^{(\bf{m})}_N  \; .\\
 \eea
 
 We get,
   \beaa
     \notag
 &&  \int_{\Sigma^{ext}_{t_2} }   \Big(  - \frac{1}{2}  (m^{t t} + H^{tt} ) < \derm_t \Phi ,   \derm_t \Phi >   + \frac{1}{2} (m^{j i} +H^{j i} )  < \derm_j \Phi ,   \derm_i \Phi >  \Big)  \cdot \widetilde{w}(q)\cdot d^{n}x  \\
 \notag
  && + \int_{t_1}^{t_2}  \int_{\Sigma^{ext}_{\tau} }  \Big(    \frac{1}{2} \Big(  | \derm_t  \Phi + \derm_r \Phi |^2  +  \de^{ij}  | ( \derm_i - \frac{x_i}{r} \derm_{r}  )\Phi |^2 \Big)  \cdot d\tau \cdot \widetilde{w}^\prime (q) d^{n}x \\
     \notag
  &= &       \int_{\Sigma^{ext}_{t_1} } \Big(  - \frac{1}{2} (m^{t t} + H^{tt} ) < \derm_t \Phi ,   \derm_t \Phi >   + \frac{1}{2} (m^{j i} +H^{j i} )   < \derm_j \Phi ,   \derm_i \Phi > \Big)  \cdot \widetilde{w}(q)  \cdot d^{n}x \\
   \notag
&&  - \int_{t_1}^{t_2}  \int_{\Sigma^{ext}_{\tau} }  \Big(    - \frac{1}{2} H^{t t}< \derm_t \Phi ,   \derm_t \Phi >   + \frac{1}{2} H^{j i}  < \derm_j \Phi ,   \derm_i \Phi > \\
\notag
&& + H^{r t}< \derm_t \Phi ,   \derm_t \Phi > + H^{r j}< \derm_j \Phi ,   \derm_t \Phi >   \Big) \cdot d\tau \cdot \widetilde{w}^\prime (q) d^{n}x \\
     \notag
     &&- \int_{t_1}^{t_2}  \int_{\Sigma^{ext}_{\tau} } \Big(  < g^{\mu\a} \derm_{\mu } \derm_\a \Phi ,   \derm_t \Phi > +  \frac{1}{2} \cdot   ( \derm_{t } g^{t \a} ) \cdot < \derm_\a \Phi ,   \derm_t \Phi > \\
     \notag
     && +( \derm_{j } g^{j \a} ) \cdot < \derm_\a \Phi ,   \derm_t \Phi >  - \frac{1}{2} \cdot  ( \derm_{t } g^{j \b} ) \cdot < \derm_j \Phi ,   \derm_\b \Phi >     \Big) \cdot d\tau \cdot \widetilde{w}(q) d^{n}x  \\
  \notag
 && -  \int_{N_{t_1}^{t_2} }  \big( T_{\hat{L} t}^{(\bf{g})} \big)  \cdot  \widetilde{w}(q) \cdot dv^{(\bf{m})}_N  \; . \\
 \eeaa
 
Based on Lemma \ref{howtogetthedesirednormintheexpressionofenergyestimate}, we have that for $ | H| \leq C < \frac{1}{n} $\;, the following equivalence,
 \bea
 \notag
   && - \frac{1}{2} (m^{t t} + H^{tt} ) < \derm_t \Phi_{V} ,   \derm_t \Phi_{V} >   + \frac{1}{2} (m^{j i} +H^{j i} )   < \derm_j \Phi_{V} ,   \derm_i \Phi_{V} > \\
   &\sim & |\derm \Phi_{V} |^2 \geq 0 \; . 
  \eea
By choosing choosing the vectors $U\,,\, V$ to be wave coordinates vector fields, and by we summing over all of them, we get the following energy estimate,

   \beaa
 &&     \int_{\Sigma^{ext}_{t_2} }  |\derm \Phi |^2     \cdot w(q)  \cdot d^{n}x    + \int_{N_{t_1}^{t_2} }  \big( T_{\hat{L} t}^{(\bf{g})} \big)  \cdot  \widetilde{w}(q) \cdot dv^{(\bf{m})}_N \\
 \notag
 &&+ \int_{t_1}^{t_2}  \int_{\Sigma^{ext}_{\tau} }  \Big(    \frac{1}{2} \Big(  | \derm_t  \Phi + \derm_r \Phi |^2  +  \de^{ij}  | ( \derm_i - \frac{x_i}{r} \derm_{r}  )\Phi |^2 \Big)  \cdot d\tau \cdot \widetilde{w}^\prime (q) d^{n}x \\
  \notag
  &\les &       \int_{\Sigma^{ext}_{t_1} }  |\derm \Phi |^2     \cdot \widetilde{w}(q)  \cdot d^{n}x \\
    \notag
&&  + \int_{t_1}^{t_2}  \int_{\Sigma^{ext}_{\tau} }   | \Big(       - \frac{1}{2} H^{t t}< \derm_t \Phi ,   \derm_t \Phi >   + \frac{1}{2} H^{j i}  < \derm_j \Phi ,   \derm_i \Phi >  \\
   \notag
&&  + H^{r t}< \derm_t \Phi ,   \derm_t \Phi > + H^{r j}< \derm_j \Phi ,   \derm_t \Phi >   \Big) |  \cdot d\tau \cdot \widetilde{w}^\prime (q) d^{n}x \\
     \notag
     &&+ \int_{t_1}^{t_2}  \int_{\Sigma^{ext}_{\tau} } | \Big(  < g^{\mu\a} \derm_{\mu } \derm_\a \Phi ,   \derm_t \Phi > +  \frac{1}{2} \cdot   ( \derm_{t } H^{t \a} ) \cdot < \derm_\a \Phi ,   \derm_t \Phi > \\
     \notag
     && +( \derm_{j } H^{j \a} ) \cdot < \derm_\a \Phi ,   \derm_t \Phi >  - \frac{1}{2} \cdot  ( \derm_{t } H^{j \b} ) \cdot < \derm_j \Phi ,   \derm_\b \Phi >     \Big) |  \cdot d\tau \cdot \widetilde{w}(q) d^{n}x  \; .
 \eeaa
 
Using the fact that we have by construction  $T_{\hat{L} t}^{(\bf{g})} \geq 0$ (see Definition \ref{definitionofwidetildeLfordivergencetheoremuse} and also Definition \ref{definitionofLtilde}), we then get the result.

\end{proof}

 \begin{corollary}\label{TheenerhyestimatewithtermsinvolvingHandderivativeofHandwithwandhatw}
  For 
\bea
| H| < \frac{1}{n}\; ,
\eea
where $n$ is the space dimension, and for $\Phi$ decaying sufficiently fast at spatial infinity, we have
  \beaa
 &&     \int_{\Sigma^{ext}_{t_2} }  |\derm \Phi_{V} |^2     \cdot w(q)  \cdot d^{n}x   \\
 \notag
 &&+ \int_{t_1}^{t_2}  \int_{\Sigma^{ext}_{\tau} }  \Big(    \frac{1}{2} \Big(  | \derm_t  \Phi_{V} + \derm_r \Phi_{V} |^2  +  \de^{ij}  | ( \derm_i - \frac{x_i}{r} \derm_{r}  )\Phi_{V} |^2 \Big)  \cdot d\tau \cdot \widehat{w}^\prime (q) d^{n}x \\
  \notag
   &\les &       \int_{\Sigma^{ext}_{t_1} }  |\derm \Phi_{V} |^2     \cdot w(q)  \cdot d^{n}x   + \int_{t_1}^{t_2}  \int_{\Sigma^{ext}_{\tau} }    | H |  \cdot  | \derm \Phi_{V} |^2     \cdot \widehat{w}^\prime (q) \cdot d^{n}x \cdot d\tau\\
     \notag
     &&+ \int_{t_1}^{t_2}  \int_{\Sigma^{ext}_{\tau} }  \Big(   | g^{\mu\a} \derm_{\mu } \derm_\a \Phi_{V} | \cdot |  \derm \Phi_{V} |  + | \derm H  | \cdot | \derm \Phi_{V} |^2       \Big)    \cdot w(q) \cdot d^{n}x \cdot d\tau  \; .\\
 \eeaa
 
 \end{corollary}
 
 \begin{proof}
 Using Lemma \ref{equaivalenceoftildewandtildeandofderivativeoftildwandderivativeofhatw} and injecting in Lemma \ref{TheenerhyestimatewithtermsinvolvingHandderivativeofH}, we get the stated result.

 \end{proof}

\section{Ingredients of the proof of the exterior stability of the Minkowski space-time for $n\geq 4$}

\subsection{The Minkowski vector fields}\ \label{TheMinkwoskivectorfieldsdefinition}

First, we refer the reader to \cite{G4} for more details. Let 
\beaa
x_{\b} &=& m_{\mu\b} x^{\mu} \;, \\
Z_{\a\b} &=& x_{\b} \pa_{\a} - x_{\a} \pa_{\b}  \;,  \\
S &=& t \pa_t + \sum_{i=1}^{n} x^i \pa_{i}  \; .
\eeaa

The Minkowski vector fields are the vectors of the following set
\bea
{\cal Z}  := \big\{ Z_{\a\b}\;,\; S\;,\; \pa_{\a} \, \,  | \, \,   \a\;,\; \b \in \{ 0, \ldots, n \} \big\}  \; .
\eea
Vectors belonging to ${\cal Z}$ will be denoted by $Z$\;.

\begin{definition} \label{DefinitionofZI}
We define
\bea
Z^I :=Z^{\iota_1} \ldots Z^{\iota_k} \quad \text{for} \quad I=(\iota_1, \ldots,\iota_k),  
\eea
where $\iota_i$ is an $\frac{(n^2 + 3n + 4)}{2}$-dimensional integer index, with $|\iota_i|=1$, and $Z^{\iota_i}$ representing each a vector field from the family ${\cal Z}$.

For a tensor $T$, of arbitrary order, either a scalar or valued in the Lie algebra, we define the Lie derivative as
\bea
\Lie_{Z^I} T :=\Lie_{Z^{\iota_1}} \ldots \Lie_{Z^{\iota_k}} T \quad \text{for} \quad I=(\iota_1, \ldots,\iota_k) .
\eea
\end{definition}

\subsection{The bootstrap argument in the exterior}\

We look at the case where $n \geq 4$\;. As in \cite{G4}, we define

 \bea
 h_{\mu\nu} = g_{\mu\nu} - m_{\mu\nu}   \; .
\eea

We then define the weighted energy as follows, yet this time, it is restricted to the exterior region,
\bea\label{definitionoftheweightedenergy}
\notag
\E_{|I|}^{ext} (\tau) :=  \sum_{|J|\leq |I|} \big( \|w^{1/2}   \derm ( \Lie_{Z^J} h   (t,\cdot) )  \|_{L^2 (\Sigma^{ext}_{t} ) } +  \|w^{1/2}   \derm ( \Lie_{Z^J}  A   (t,\cdot) )  \|_{L^2 (\Sigma^{ext}_{t} ) } \big) \, ,
\eea
where $\Sigma^{ext}_{t}$ is definied in Definition \ref{definitionoftheexteriorslicesigmat}. We will run a bootstrap argument on $\E_{|I|}^{ext}$\;. We assume for some $|I| \geq N$\,, which we will determine later that
\bea \label{bootstrapexter}
\E_{ N }^{ext} (t)  \leq E^{ext}  (N )  \cdot \eps \cdot (1 +t)^\delta ,
\eea
where $E^{ext}  ( N )$ is a constant that depends on $ N $ to be chosen later. In this paper, we choose
\bea\label{delataqualtozero}
\delta &=& 0 \;, \\ 
\eps &=& 1 \; . \label{epsequaltoone}
\eea
We will then show that we can improve the constant $E^{ext}  (N )$ to  $\frac{1}{2} \cdot E^{ext}  ( N )$, and obtain
\beaa
\E_{ N }^{ext}  (t)  \leq \frac{ E^{ext}  ( N )}{2}   \cdot \eps \cdot (1 +t)^\delta \; .
\eeaa

We will now make use of the Klainerman-Sobolev inequality which holds also true in the exterior region $\overline{C}$, the complementary of $C$ (the future causal domain of dependance for the metric $g$, of the compact $K \subset \Sigma_{t_1}$). We will then run the same bootstrap argument for the exterior energy. 

\subsection{Weighted Klainerman-Sobolev inequality in the exterior}\

The weight is defined again as being $w$ (see Definition \eqref{defoftheweightw}), for some $\gamma > 0$\;. However, the integration for the $L^2$ will be supported only on the exterior regions $\Sigma^{ext}_{t} = \Sigma_{t} \cap \overline{C}$. Then, we have globally the following pointwise estimate in the exterior region $\overline{C}$ for any smooth scalar function $\phi$ vanishing at spatial infinity, i.e. $\lim_{r \to \infty} \phi (t, x^1, \ldots, x^n) = 0$,
\bea
\notag
|\phi(t,x)| \cdot (1+t+|q|)^{\frac{(n-1)}{2}} \cdot \big[ (1+|q|) \cdot w(q)\big]^{1/2} \leq
C\sum_{|I|\leq  \lfloor  \frac{n}{2} \rfloor  +1 } \|\big(w(q)\big)^{1/2} Z^I \phi(t,\cdot)\|_{L^2 (\Sigma^{ext}_{t} ) } \;, \\
\eea
where here the $L^2(\Sigma^{ext}_{t} )$ norm is taken on $\Sigma^{ext}_{t} $ slice.

\subsection{The a priori estimates}\

Based on the calculations that we showed in \cite{G4}, we have the following exterior versions of the a prior estimates that we derived in \cite{G4}, by using the Klainerman-Svolev inequality in the exterior with the energy defined in the exterior.

\begin{lemma} \label{aprioriestimatesongradientoftheLiederivativesofthefields}
Under the bootstrap assumption \eqref{bootstrapexter}, taken for $N =  |I| +  \lfloor  \frac{n}{2} \rfloor  +1$\;, if for all $\mu, \nu \in (t, x^1, \ldots, x^n)$\;, and for any functions $ \pa_\mu \Lie_{Z^I} h^1_\nu \; , \; \pa_\mu \Lie_{Z^I} A_\nu  \in C^\infty_0(\R^n)$\;, then we have
 \bea
 \notag
|\derm  ( \Lie_{Z^I}  A ) (t,x)  |           &\leq& \begin{cases} C ( |I| ) \cdot E^{ext} ( |I| + \lfloor  \frac{n}{2} \rfloor  +1 )  \cdot \frac{\eps }{(1+t+|q|)^{\frac{(n-1)}{2}-\delta} (1+|q|)^{1+\gamma}},\quad\text{when }\quad q>0 \; ,\\
           \notag
       C ( |I| ) \cdot E^{ext} ( |I| + \lfloor  \frac{n}{2} \rfloor  +1)  \cdot \frac{\eps  }{(1+t+|q|)^{\frac{(n-1)}{2}-\delta}(1+|q|)^{\frac{1}{2} }}  \,\quad\text{when }\quad q<0 \;  , \end{cases} \\
      \eea
and 
 \bea
 \notag
|\derm ( \Lie_{Z^I}  h ) (t,x)  |   &\leq& \begin{cases} C ( |I| ) \cdot E^{ext} ( |I| + \lfloor  \frac{n}{2} \rfloor  +1)  \cdot \frac{\eps }{(1+t+|q|)^{\frac{(n-1)}{2}-\delta} (1+|q|)^{1+\gamma}},\quad\text{when }\quad q>0 \;  ,\\
       C ( |I| ) \cdot E^{ext} ( |I| + \lfloor  \frac{n}{2} \rfloor  +1)  \cdot \frac{\eps  }{(1+t+|q|)^{\frac{(n-1)}{2}-\delta}(1+|q|)^{\frac{1}{2} }}  \,\quad\text{when }\quad q<0 \; . \end{cases} \\
      \eea

\end{lemma}

\begin{lemma}\label{aprioriestimatefrombootstraponzerothderivativeofAandh1}
For $k =  |I| +  \lfloor  \frac{n}{2} \rfloor  +1$\,, and for $\ga >0$\, and with initial data such that
\beaa
|   A (0,x) |  +  |   h^1 (0,x) |   &\les&   \frac{\eps }{ (1+r)^{\frac{(n-1)}{2}+\gamma-\delta}} \; ,
\eeaa
then, we have for all $ | I  | $\,,
 \bea
 \notag
&&| \Lie_{Z^I}  A  (t,x)  |   +| \Lie_{Z^I}  h  (t,x)  |   \\
 \notag 
 &\leq& \begin{cases} c (\gamma) \cdot  C ( |I| ) \cdot E^{ext} ( |I| +  \lfloor  \frac{n}{2} \rfloor  +1)  \cdot \frac{\eps }{(1+t+|q|)^{\frac{(n-1)}{2}-\delta} (1+|q|)^{\gamma}},\quad\text{when }\quad q>0\; ,\\
       C ( |I| ) \cdot E^{ext} ( |I| +  \lfloor  \frac{n}{2} \rfloor  +1)  \cdot \frac{\eps \cdot (1+| q |   )^\frac{1}{2} }{(1+t+|q|)^{\frac{(n-1)}{2}-\delta} }  \,\quad\text{when }\quad q<0 \; . \end{cases} \\
      \eea

\end{lemma}

\begin{remark}\label{remarkontheinteriorregionbeingcontainedinqleqq_0}
Under the bootstrap assumption, and therefore under the a priori estimates in Lemmas \ref{aprioriestimatesongradientoftheLiederivativesofthefields} and \ref{aprioriestimatefrombootstraponzerothderivativeofAandh1}, we have that for any $q_0 \in \R$\;, there exists a point $(t, r=0)$ such that $N$ (defined in Definition \ref{definitionofNandofNtruncatedbtweentwots}) whose tip is $(t, r=0)$ is contained in the region $\{(t, x) \;|\; q:= r-t \leq q_0 \}$\;. 
\end{remark}

\begin{definition}
Based on Remark \ref{remarkontheinteriorregionbeingcontainedinqleqq_0}, we have that the exterior region $\overline{C} $ includes the region $\{(t, x) \;|\; q:= r- t \geq q_0 \}$. We refer in what follows to $q_0$ as being a choice of such $q$\;, from which we construct the exterior region as containing $\{ q \geq q_0 \}$\;.
\end{definition}

\subsection{The main exterior energy estimate}\

We now fix the space dimensions as being n $\geq 4$.
 \begin{lemma}\label{theenergyestimateforngeq4withoutestimatingthecommutatorterm}
 For $ H^{\mu\nu} = g^{\mu\nu}-m^{\mu\nu}$ satisfying 
\bea
| H| \leq  \frac{1}{n} \; ,
\eea
and for $\Phi$ decaying sufficiently fast at spatial infinity, for $\ga > 0 $ and $\mu < 0$,  we have
       \beaa
 &&     \int_{\Sigma^{ext}_{t_2} }  |\derm \Phi |^2     \cdot w(q)  \cdot d^{n}x   \\
 \notag
 &&+ \int_{t_1}^{t_2}  \int_{\Sigma^{ext}_{\tau} }  \Big(    \frac{1}{2} \Big(  | \derm_t  \Phi + \derm_r \Phi |^2  +  \de^{ij}  | ( \derm_i - \frac{x_i}{r} \derm_{r}  )\Phi |^2 \Big) \cdot \frac{\widehat{w} (q)}{(1+|q|)} \cdot d^{n}x  \cdot d\tau  \\
\notag
       &\les &       \int_{\Sigma^{ext}_{t_1} }  |\derm \Phi |^2     \cdot w(q)  \cdot d^{n}x   \\
       \notag
     &&  +   C(q_0) \cdot  c (\delta) \cdot c (\gamma) \cdot E^{ext}  (   \lfloor  \frac{n}{2} \rfloor  +1) \cdot \int_{t_1}^{t_2}   \frac{\eps}{ (1+ t  )^{\frac{3}{2} } }  \cdot  \int_{\Sigma^{ext}_{\tau} }    | \derm \Phi |^2    \cdot \frac{w (q)}{(1+|q|)}  \cdot d^{n}x  \cdot d\tau \\
     \notag
     &&+ \int_{t_1}^{t_2}  \int_{\Sigma^{ext}_{\tau} }   | g^{\mu\a} \derm_{\mu } \derm_\a \Phi | \cdot |  \derm \Phi |    \cdot w(q) \cdot d^{n}x  \cdot d\tau \; .\\
 \eeaa

 \end{lemma}

  \begin{proof}
 
 Using the bootstrap assumption on $H$ in the exterior, combined with the Klainerman-Sobolev inequality in the exterior region $q \geq q_0$, we obtain in the exterior region $\overline{C}$, as shown in \cite{G4}, that
 
              \bea\label{aprioriestimateondermcovariantderivativeofBigHwithLiederivativeL}
 \notag
 |\derm ( \Lie_{Z^I} H ) (t,x)  |    &\leq& \begin{cases} C ( |I| ) \cdot E^{ext} (|I| +  \lfloor  \frac{n}{2} \rfloor  +1)  \cdot \frac{\eps }{(1+t+|q|)^{\frac{(n-1)}{2}-\delta} (1+|q|)^{1+\ga}},\quad\text{when }\quad q>0,\\
       C ( |I| ) \cdot E^{ext}  (|I| +  \lfloor  \frac{n}{2} \rfloor  +1)  \cdot \frac{\eps  }{(1+t+|q|)^{\frac{(n-1)}{2}-\delta}(1+|q|)^{\frac{1}{2} }}  \,\quad\text{when }\quad q<0 , \end{cases} \\
      \eea
      and
       \bea\label{aprioriestimateonBigHwithLiederivativeL}
 \notag
|  \Lie_{Z^I} H (t,x)  | &\leq& \begin{cases} c (\delta) \cdot c (\gamma) \cdot C ( |I| ) \cdot E^{ext}  ( |I| +  \lfloor  \frac{n}{2} \rfloor  +1) \cdot  \frac{\eps}{ (1+ t + | q | )^{\frac{(n-1)}{2}-\delta }  \cdot  (1+| q |   )^{\ga}}  ,\quad\text{when }\quad q>0,\\
\notag
    C ( |I| ) \cdot E^{ext}  (|I| +  \lfloor  \frac{n}{2} \rfloor  +1) \cdot  \frac{\eps}{ (1+ t + | q | )^{\frac{(n-1)}{2}-\delta }  } \cdot (1+| q |   )^{\frac{1}{2} }  , \,\quad\text{when }\quad q<0 . \end{cases} \\ 
    \eea
    
           Taking $\delta = 0$, we have $0  \leq \frac{(4-2)}{2} = 1  \leq  \frac{(n-2)}{2}  $, for $n \geq 4$. Thus, for $n\geq 4$, in $\overline{C}$, we have
                         \bea\label{aprioriestimateondermcovariantderivativeofBigHwithLiederivativeLintheexteriorregion}
 \notag
 |\derm ( \Lie_{Z^I} H ) (t,x)  |    &\leq&  C(q_0) \cdot C ( |I| ) \cdot E^{ext} (|I| +  \lfloor  \frac{n}{2} \rfloor  +1)  \cdot \frac{\eps }{(1+t+|q|)^{\frac{3}{2}} \cdot  (1+|q|)^{1+\ga}}  \; , \\
      \eea
      and
       \bea\label{aprioriestimateonBigHwithLiederivativeLintheexteriorregion}
 \notag
|  \Lie_{Z^I} H (t,x)  | &\leq&  C(q_0) \cdot  c (\delta) \cdot c (\gamma) \cdot C ( |I| ) \cdot E^{ext}  ( |I| +  \lfloor  \frac{n}{2} \rfloor  +1) \cdot  \frac{\eps}{ (1+ t + | q | )^{\frac{3}{2} } \cdot  (1+| q |   )^{\ga}}  \; . \\
    \eea
   Yet, given the weighted energy estimate that we showed in Corollary \ref{TheenerhyestimatewithtermsinvolvingHandderivativeofHandwithwandhatw}, by taking $n \geq 4$ and by injecting the a priori estimates, considering that $\ga > 0$, we get that for 
   \bea
| H| < \frac{1}{n} \; ,
\eea
and for $\Phi$ decaying sufficiently fast at spatial infinity, that
    
       \beaa
 &&     \int_{\Sigma^{ext}_{t_2} }  |\derm \Phi |^2     \cdot w(q)  \cdot d^{n}x   \\
 \notag
 &&+ \int_{t_1}^{t_2}  \int_{\Sigma^{ext}_{\tau} }  \Big(    \frac{1}{2} \Big(  | \derm_t  \Phi + \derm_r \Phi |^2  +  \de^{ij}  | ( \derm_i - \frac{x_i}{r} \derm_{r}  )\Phi |^2 \Big)  \cdot \widehat{w}^\prime (q) \cdot  d^{n}x  \cdot d\tau\\
  \notag
   &\les &       \int_{\Sigma^{ext}_{t_1} }  |\derm \Phi |^2     \cdot w(q)  \cdot d^{n}x   + \int_{t_1}^{t_2}  \int_{\Sigma^{ext}_{\tau} }    | H |  \cdot  | \derm \Phi |^2    \cdot \widehat{w}^\prime (q) \cdot  d^{n}x  \cdot d\tau \\
     \notag
     &&+ \int_{t_1}^{t_2}  \int_{\Sigma^{ext}_{\tau} }  \Big(   | g^{\mu\a} \derm_{\mu } \derm_\a \Phi | \cdot |  \derm \Phi |  + | \derm H  | \cdot | \derm \Phi |^2       \Big)    \cdot w(q) \cdot d^{n}x  \cdot d\tau \\
     \notag
       &\les &       \int_{\Sigma^{ext}_{t_1} }  |\derm \Phi |^2     \cdot w(q)  \cdot d^{n}x   \\
       \notag
     &&  +   C(q_0) \cdot  c (\delta) \cdot c (\gamma) \cdot E^{ext}  (   \lfloor  \frac{n}{2} \rfloor  +1) \cdot \int_{t_1}^{t_2}  \int_{\Sigma^{ext}_{\tau} }     \frac{\eps}{ (1+ t  )^{\frac{3}{2} } }  \cdot  | \derm \Phi |^2     \cdot \big( \widehat{w}^\prime (q) +  \frac{w(q)}{(1+|q|)} \big) \cdot d^{n}x  \cdot d\tau \\
     \notag
     &&+ \int_{t_1}^{t_2}  \int_{\Sigma^{ext}_{\tau} }   | g^{\mu\a} \derm_{\mu } \derm_\a \Phi | \cdot |  \derm \Phi |       \cdot w(q)  \cdot d^{n}x  \cdot d\tau  \; .\\
 \eeaa

However, we showed in Lemma \ref{derivativeoftildwandrelationtotildew}, that for
\beaa
\widehat{w} (q)&:=&\begin{cases} (1+|q|)^{1+2\gamma} \quad\text{when }\quad q>0 , \\
        (1+|q|)^{2\mu}  \,\quad\text{when }\quad q<0 , \end{cases} 
\eeaa
we have, for $\ga \neq - \frac{1}{2} $ and $\mu \neq 0$,  
\beaa
\widehat{w}^{\prime}(q) \sim \frac{\widehat{w}(q)}{(1+|q|)} \; .
\eeaa
Furthermore, for $\mu < 0$, we have $\widehat{w} (q) \leq w(q)$, thus,
\beaa
\widehat{w}^{\prime}(q) \leq \frac{w(q)}{(1+|q|)} \; .
\eeaa
Consequently, for $\ga > 0 $ and $\mu < 0$,
       \bea\label{estimatetoreferencetoexplainthedefinitionoftildewandhatw}
       \notag
 &&     \int_{\Sigma^{ext}_{t_2} }  |\derm \Phi |^2     \cdot w(q)  \cdot d^{n}x   \\
 \notag
 &&+ \int_{t_1}^{t_2}  \int_{\Sigma^{ext}_{\tau} }  \Big(    \frac{1}{2} \Big(  | \derm_t  \Phi + \derm_r \Phi |^2  +  \de^{ij}  | ( \derm_i - \frac{x_i}{r} \derm_{r}  )\Phi |^2 \Big) \cdot \widehat{w}^\prime (q) \cdot d^{n}x  \cdot d\tau  \\
\notag
       &\les &       \int_{\Sigma^{ext}_{t_1} }  |\derm \Phi |^2     \cdot w(q)  \cdot d^{n}x   \\
       \notag
     &&  +   C(q_0) \cdot  c (\delta) \cdot c (\gamma) \cdot E^{ext}  (   \lfloor  \frac{n}{2} \rfloor  +1) \cdot \int_{t_1}^{t_2}   \frac{\eps}{ (1+ t  )^{\frac{3}{2} } }  \cdot  \int_{\Sigma^{ext}_{\tau} }    | \derm \Phi |^2    \cdot \frac{w (q)}{(1+|q|)}  \cdot d^{n}x  \cdot d\tau \\
     \notag
     &&+ \int_{t_1}^{t_2}  \int_{\Sigma^{ext}_{\tau} }   | g^{\mu\a} \derm_{\mu } \derm_\a \Phi | \cdot |  \derm \Phi |    \cdot w(q) \cdot d^{n}x  \cdot d\tau \; .\\
 \eea

 \end{proof}

We now state the following lemma that is an exterior version of an estimate on the commutator term that we showed in \cite{G4}.

 \begin{lemma}\label{estimatethespace-timeintegralofthecommutatortermneededfortheenergyestimate}
For $n\geq 4$\,, let $H$ such that for all time $t$\,, for $\ga \neq 0$ and $ 0 < \la \leq \frac{1}{2}$\,, 
\bea
 \int_{\SSS^{n-1}} \lim_{r \to \infty} \Big( \frac{r^{n-1}}{(1+t+r)^{2-\la} \cdot (1+|q|) } \cdot  w(q) \cdot |H|^2  \Big)  d\si^{n-1} (t ) &=& 0 \; ,
\eea
and let $h$ such that for all time $t$\,, for all $|K| \leq |I|$\,,
\bea
 \int_{\SSS^{n-1}} \lim_{r \to \infty} \Big( \frac{r^{n-1}}{ (1+|q|) }  \cdot  w(q) \cdot  | \Lie_{Z^K} h  |^2  \Big)  d\si^{n-1} (t ) &=& 0 \; ,
\eea
then, for $\de = 0$\,, for either $\Phi = H$ or $\Phi = A$\,, using the bootstrap assumption on $\Phi$\,, we have
   \bea
      \notag
  &&   \int_0^t \Big(  \int_{\Si_{t}^{ext}}  (1+t )^{1+\la}  \cdot  | g^{\a\b} \derm_{\a}   \derm_{\b}   ( \Lie_{ Z^I}   \Phi ) |^2  \cdot w  \cdot dx^1 \ldots dx^n  \Big) \cdot dt  \\
 \notag
&\les&  \int_0^t   \frac{\eps}{(1+t)^{2-\la}}  \cdot  C ( |I| ) \cdot E ( \lfloor \frac{|I|}{2} \rfloor+  \lfloor  \frac{n}{2} \rfloor  +1)  \cdot  c(\ga)  \\
      \notag
&&  \times \Big(  \sum_{|J|\leq |I|}   \int_{\Si_{t}} \big(  | \derm ( \Lie_{Z^{J}} h ) |^2 + |\derm ( \Lie_{Z^{K}} \Phi )  |^2 \big) \cdot   w  \cdot dx^1 \ldots dx^n   \Big) \cdot dt   \\
      \notag
&& + \int_0^t  \Big( \int_{\Si_{t}^{ext}}  (1+t )^{1+\la}  \cdot \sum_{|K| \leq |I| }  | \Lie_{Z^K}  g^{\a\b} \derm_{\a}   \derm_{\b}     \Phi |^2   \cdot w   \cdot dx^1 \ldots dx^n  \Big) \cdot dt  \; . \\
\eea

 \end{lemma}

\begin{proof}
Based on our previous calculations in \cite{G4} for $n \geq 4$ for the commutator term, also using the exterior Hardy type inequality that we showed in \cite{G4} and that we re-state here in Corollary \ref{HardytypeinequalityforintegralstartingatROm}), it is straightforward to show the stated exterior estimate on the commutator term. 

\end{proof}

Now, we have the following exterior estimate on the commutator term. 
  \begin{lemma}
For $n\geq 4$, let $H$ such that for all time $t$, for $\ga \neq 0$ and $ 0 < \la \leq \frac{1}{2}$, 
\bea
 \int_{\SSS^{n-1}} \lim_{r \to \infty} \Big( \frac{r^{n-1}}{(1+t+r)^{2-\la} \cdot (1+|q|) } w(q) \cdot |H|^2  \Big)  d\si^{n-1} (t ) &=& 0 \; ,
\eea
and such that that for all $|K| \leq |I| $,
 \bea
\notag
 \int_{\SSS^{n-1}} \lim_{r \to \infty} \Big( \frac{r^{n-1}}{(1+t+r)^{2-\la} \cdot (1+|q|) } \cdot w(q) \cdot   \Big( |   \Lie_{Z^K} A  |^2 + |   \Lie_{Z^K} h  |^2  \Big)  d\si^{n-1} (t ) &=& 0  \; , \\
\eea

then, for $\de = 0$, using the bootstrap assumption on $A$ and $h$, we have
   \bea
      \notag
  &&   \int_0^t \Big(  \int_{\Si_{t}^{ext}}  \frac{(1+t )^{1+\la}}{\eps}  \cdot  | g^{\a\b} \derm_{\a}   \derm_{\b}   ( \Lie_{ Z^I}   A ) |^2  \cdot w  \cdot dx^1 \ldots dx^n  \Big) \cdot dt  \\
        \notag
  && +  \int_0^t \Big(  \int_{\Si_{t}^{ext}}  \frac{(1+t )^{1+\la}}{\eps}  \cdot  | g^{\a\b} \derm_{\a}   \derm_{\b}   ( \Lie_{ Z^I}   h ) |^2  \cdot w  \cdot dx^1 \ldots dx^n  \Big) \cdot dt  \\
 \notag
&\les&  C ( |I| ) \cdot E^{ext} ( \lfloor \frac{|I|}{2} \rfloor+  \lfloor  \frac{n}{2} \rfloor  +1)  \cdot  c(\ga) \cdot \int_0^t   \frac{\eps}{(1+\tau)^{2-\la}}  \cdot  \E_{|I|}^{ext} (\tau)   \cdot d\tau   \\
      \notag
&& + \int_0^t  \Big( \int_{\Si_{t}^{ext}}  \frac{(1+\tau )^{1+\la}}{\eps}  \cdot \sum_{|K| \leq |I| }  | \Lie_{Z^K}  g^{\a\b} \derm_{\a}   \derm_{\b}  A  |^2   \cdot w   \cdot dx^1 \ldots dx^n  \Big) \cdot d\tau   \\
\notag
&& + \int_0^t  \Big( \int_{\Si_{t}^{ext}}  \frac{(1+\tau )^{1+\la}}{\eps}  \cdot \sum_{|K| \leq |I| }  | \Lie_{Z^K}  g^{\a\b} \derm_{\a}   \derm_{\b}   h |^2   \cdot w   \cdot dx^1 \ldots dx^n  \Big) \cdot d\tau  \; . \\
\eea

 \end{lemma}

\begin{proof}
Using Lemma \ref{estimatethespace-timeintegralofthecommutatortermneededfortheenergyestimate}, we get the desired result. We notice that since the needed estimates on the metric hold for $n \geq 4$ everywhere (i.e. in the interior as well as in the exterior), we do not have dependance on $q_0$, i.e. there is no constant $C(q_0)$. More precisely,

For $n\geq 4$, let $H$ such that for all time $t$, for $\ga \neq 0$ and $ 0 < \la \leq \frac{1}{2}$, 
\bea
 \int_{\SSS^{n-1}} \lim_{r \to \infty} \Big( \frac{r^{n-1}}{(1+t+r)^{2-\la} \cdot (1+|q|) } w(q) \cdot |H|^2  \Big)  d\si^{n-1} (t ) &=& 0 \; ,
\eea
then, for $\de = 0$, for $\Phi = H$ or $\Phi = A$, using the bootstrap assumption on $\Phi$, we have
   \bea
      \notag
  &&   \int_0^t \Big(  \int_{\Si_{t}^{ext}}  \frac{(1+\tau)^{1+\la}}{\eps}  \cdot  | g^{\a\b} \derm_{\a}   \derm_{\b}   ( \Lie_{ Z^I}   \Phi ) |^2  \cdot w  \cdot dx^1 \ldots dx^n  \Big) \cdot d\tau  \\
 \notag
&\les&  \int_0^t   \frac{\eps}{(1+\tau)^{2-\la}}  \cdot  C ( |I| ) \cdot E^{ext} ( \lfloor \frac{|I|}{2} \rfloor+  \lfloor  \frac{n}{2} \rfloor  +1)  \cdot  c(\ga)  \\
      \notag
&& \cdot \Big(  \sum_{|J|\leq |I|}   \int_{\Si_{t}^{ext}} \big(  | \derm ( \Lie_{Z^{J}} h ) |^2 + |\derm ( \Lie_{Z^{K}} \Phi )  |^2 \big) \cdot   w  \cdot dx^1 \ldots dx^n   \Big) \cdot d\tau  \\
      \notag
&& + \int_0^t  \Big( \int_{\Si_{t}^{ext}}  \frac{(1+\tau )^{1+\la}}{\eps}  \cdot \sum_{|K| \leq |I| }  | \Lie_{Z^K}  g^{\a\b} \derm_{\a}   \derm_{\b}     \Phi |^2   \cdot w   \cdot dx^1 \ldots dx^n  \Big) \cdot d\tau \; . \\
\eea

\end{proof}

 \begin{lemma}\label{TheenergyestimateinGronwallformonbothAandhforneq4}
 For $ H^{\mu\nu} = g^{\mu\nu}-m^{\mu\nu}$ satisfying 
\bea
| H| \leq  \frac{1}{n} \; ,
\eea
and such that for $n\geq 4$\,, for $\ga \neq 0$ and $ 0 < \la \leq \frac{1}{2}$\,, 
\bea
 \int_{\SSS^{n-1}} \lim_{r \to \infty} \Big( \frac{r^{n-1}}{(1+t+r)^{2-\la} \cdot (1+|q|) } w(q) \cdot |H|^2  \Big)  d\si^{n-1} (t ) &=& 0 \; ,
\eea
and such that that for all $|K| \leq |I| $,
 \bea
\notag
 \int_{\SSS^{n-1}} \lim_{r \to \infty} \Big( \frac{r^{n-1}}{(1+t+r)^{2-\la} \cdot (1+|q|) } \cdot w(q) \cdot   \Big( |   \Lie_{Z^K} A  |^2 + |   \Lie_{Z^K} h  |^2  \Big)  d\si^{n-1} (t ) &=& 0  \; . \\
\eea
Then, for $\Lie_{Z^J} A $ and $\Lie_{Z^J} h^1$ decaying sufficiently fast at spatial infinity, for $\ga >  0$ and for $ 0 < \la \leq \frac{1}{2}$\,, we have
         \beaa
 &&     {(\E_{|I|}^{ext})}^2 (t_2)    \\
 \notag
       &\les & {(\E_{|I|}^{ext})}^2 (t_1)     +  C ( |I| ) \cdot C(q_0) \cdot  c(\ga) \cdot  E^{ext} ( \lfloor \frac{|I|}{2} \rfloor+  \lfloor  \frac{n}{2} \rfloor  +1)    \cdot \int_{t_1}^{t_2}   \frac{\eps}{ (1+ t  )^{1+\la}  }  \cdot  {(\E_{|I|}^{ext})}^2 (\tau) \cdot d\tau \\
     \notag
&& + C(|I|) \cdot \int_0^t  \Big( \int_{\Si_{t}^{ext}}  \frac{(1+\tau )^{1+\la}}{\eps}  \cdot \sum_{|K| \leq |I| }  | \Lie_{Z^K}  g^{\a\b} \derm_{\a}   \derm_{\b}    A |^2   \cdot w   \cdot dx^1 \ldots dx^n  \Big) \cdot d\tau   \\
      \notag
&& + C(|I|) \cdot  \int_0^t  \Big( \int_{\Si_{t}^{ext}}  \frac{(1+\tau )^{1+\la}}{\eps}  \cdot \sum_{|K| \leq |I| }  | \Lie_{Z^K}  g^{\a\b} \derm_{\a}   \derm_{\b}    h |^2   \cdot w   \cdot dx^1 \ldots dx^n  \Big) \cdot d\tau \; , \\
\eeaa
where
   \beaa
\E_{|I|}^{ext} (\tau) :=  \sum_{|J|\leq |I|} \big( \|w^{1/2}   \derm ( \Lie_{Z^J} h^1   (t,\cdot) )  \|_{L^2 (\Sigma^{ext}_{t} ) } +  \|w^{1/2}   \derm ( \Lie_{Z^J}  A   (t,\cdot) )  \|_{L^2 (\Sigma^{ext}_{t} ) } \big) \, .
\eeaa

 \end{lemma}

  \begin{proof}

By taking $\Phi = \Lie_{Z^J} A $ and another $\Phi = \Lie_{Z^J} h^1 $, decaying sufficiently fast at spatial infinity, and using the energy estimate that we have shown in Lemma \ref{theenergyestimateforngeq4withoutestimatingthecommutatorterm}, for $ | H| \leq \frac{1}{n} $, where $n$ is the space dimension, and for $\ga > 0$\,, we have
       \beaa
 &&     \int_{\Sigma^{ext}_{t_2} }  \big( |\derm  ( \Lie_{Z^J} A ) |^2   +  |\derm  ( \Lie_{Z^J} h ) |^2  \big)  \cdot w(q)  \cdot d^{n}x   \\
 \notag
 &&+ \int_{t_1}^{t_2}  \int_{\Sigma^{ext}_{\tau} }  \Big(    \frac{1}{2} \Big(  | \derm_t  ( \Lie_{Z^J} A ) + \derm_r ( \Lie_{Z^J} A ) |^2  +  \de^{ij}  | ( \derm_i - \frac{x_i}{r} \derm_{r}  )( \Lie_{Z^J} A ) |^2  \\
  \notag
 &&+   | \derm_t  ( \Lie_{Z^J} h^1 ) + \derm_r ( \Lie_{Z^J} h ) |^2  +  \de^{ij}  | ( \derm_i - \frac{x_i}{r} \derm_{r}  ) ( \Lie_{Z^J} h ) |^2 \Big) \cdot \frac{\widehat{w} (q)}{(1+|q|)} \cdot d^{n}x  \cdot d\tau  \\
\notag
       &\les &       \int_{\Sigma^{ext}_{t_1} }   \big( |\derm  ( \Lie_{Z^J} A ) |^2   +  |\derm  ( \Lie_{Z^J} h ) |^2  \big)     \cdot w(q)  \cdot d^{n}x   \\
       \notag
     &&  +   C(q_0) \cdot  c (\delta) \cdot c (\gamma) \cdot E^{ext}  (   \lfloor  \frac{n}{2} \rfloor  +1) \\
     && \cdot \int_{t_1}^{t_2}   \frac{\eps}{ (1+ t  )^{\frac{3}{2} } }  \cdot  \int_{\Sigma^{ext}_{\tau} }     \big( |\derm  ( \Lie_{Z^J} A ) |^2   +  |\derm  ( \Lie_{Z^J} h ) |^2  \big)    \cdot \frac{w (q)}{(1+|q|)}  \cdot d^{n}x  \cdot d\tau \\
     \notag
     &&+ \int_{t_1}^{t_2}  \int_{\Sigma^{ext}_{\tau} } \big( \frac{(1+t )^{1+\la}}{\eps} \cdot  | g^{\mu\a} \derm_{\mu } \derm_\a ( \Lie_{Z^J} A ) |^2 + \frac{\eps}{(1+t)^{1+\la}} |  \derm ( \Lie_{Z^J} A ) |^2   \big)  \cdot w(q) \cdot d^{n}x  \cdot d\tau  \\
       \notag
     &&+ \int_{t_1}^{t_2}  \int_{\Sigma^{ext}_{\tau} } \big( \frac{(1+t )^{1+\la}}{\eps} \cdot | g^{\mu\a} \derm_{\mu } \derm_\a ( \Lie_{Z^J} h ) |^2  +\frac{\eps}{(1+t)^{1+\la}}   |  \derm ( \Lie_{Z^J} h ) |^2    \big) \cdot w(q) \cdot d^{n}x  \cdot d\tau \; .\\
 \eeaa
 Hence, for $ 0 < \la \leq \frac{1}{2}$\,,
        \beaa
 &&     {(\E_{|I|}^{ext})}^2 (t_2)    \\
 \notag
       &\les &{(\E_{|I|}^{ext})}^2 (t_1)     +   C(q_0) \cdot  c (\delta) \cdot c (\gamma) \cdot E^{ext}  (   \lfloor  \frac{n}{2} \rfloor  +1)  \cdot \int_{t_1}^{t_2}   \frac{\eps}{ (1+ t  )^{1+\la}  }  \cdot  \int_{\Sigma^{ext}_{\tau} }   {(\E_{|I|}^{ext})}^2 (\tau) \cdot d\tau \\
     \notag
     &&+\sum_{|J|\leq |I|} \int_{t_1}^{t_2}  \int_{\Sigma^{ext}_{\tau} }  \frac{(1+t )^{1+\la}}{\eps} \cdot  | g^{\mu\a} \derm_{\mu } \derm_\a ( \Lie_{Z^J} A ) |^2    \cdot w(q) \cdot d^{n}x  \cdot d\tau  \\
       \notag
     &&+\sum_{|J|\leq |I|}  \int_{t_1}^{t_2}  \int_{\Sigma^{ext}_{\tau} }  \frac{(1+t )^{1+\la}}{\eps} \cdot | g^{\mu\a} \derm_{\mu } \derm_\a ( \Lie_{Z^J} h ) |^2    \cdot w(q) \cdot d^{n}x  \cdot d\tau \; .\\
 \eeaa
 
Now, for $n\geq 4$, let $H$ such that for all time $t$, for $\ga \neq 0$ and $ 0 < \la \leq \frac{1}{2}$, 
\bea
 \int_{\SSS^{n-1}} \lim_{r \to \infty} \Big( \frac{r^{n-1}}{(1+t+r)^{2-\la} \cdot (1+|q|) } w(q) \cdot |H|^2  \Big)  d\si^{n-1} (t ) &=& 0 \; ,
\eea
we get based on the exterior estimate we have established on the commutator term, that

        \beaa
 &&     {(\E_{|I|}^{ext})}^2 (t_2)    \\
 \notag
       &\les & {(\E_{|I|}^{ext})}^2 (t_1)     +   C(q_0) \cdot  c (\delta) \cdot c (\gamma) \cdot E^{ext}  (   \lfloor  \frac{n}{2} \rfloor  +1)  \cdot \int_{t_1}^{t_2}   \frac{\eps}{ (1+ t  )^{1+\la}  }  \cdot  \int_{\Sigma^{ext}_{\tau} }    {(\E_{|I|}^{ext})}^2 (\tau) \cdot d\tau \\
     \notag
&& +  C ( |I| ) \cdot E^{ext} ( \lfloor \frac{|I|}{2} \rfloor+  \lfloor  \frac{n}{2} \rfloor  +1)  \cdot  c(\ga) \cdot \int_0^t   \frac{\eps}{(1+\tau)^{2-\la}}  \cdot  {(\E_{|I|}^{ext})}^2(\tau)   \cdot d\tau   \\
      \notag
&& + \sum_{|K| \leq |I| } \int_0^t  \Big( \int_{\Si_{t}^{ext}}  \frac{(1+\tau )^{1+\la}}{\eps}  \cdot  | \Lie_{Z^K}  g^{\a\b} \derm_{\a}   \derm_{\b}    A |^2   \cdot w   \cdot dx^1 \ldots dx^n  \Big) \cdot d\tau   \\
      \notag
&& +  \sum_{|K| \leq |I| } \int_0^t  \Big( \int_{\Si_{t}^{ext}}  \frac{(1+\tau )^{1+\la}}{\eps}  \cdot   | \Lie_{Z^K}  g^{\a\b} \derm_{\a}   \derm_{\b}    h |^2   \cdot w   \cdot dx^1 \ldots dx^n  \Big) \cdot d\tau \; .  \\
\eeaa
Fixing $\de =0$ and  $ 0 < \la \leq \frac{1}{2}$, we get the result.

 \end{proof}

\subsection{The source terms for $n\geq 4$}\

We proved the following two lemmas in \cite{G4}.

\begin{lemma}
We have
      \beaa
   \notag
 && |  \Lie_{Z^I}   ( g^{\la\mu} \derm_{\la}   \derm_{\mu}   A  )   | \\
   \notag
 &\les&   \sum_{|K|\leq |I|}  \Big( |    \derm (\Lie_{Z^K} A ) | \Big) \cdot E (   \lfloor \frac{|I|}{2} \rfloor + \lfloor  \frac{n}{2} \rfloor  + 1)\\
 \notag
 && \cdot \Big( \Big( \begin{cases}  \frac{\eps }{(1+t+|q|)^{\frac{(n-1)}{2}-\delta} (1+|q|)^{1+\gamma}},\quad\text{when }\quad q>0,\\
           \notag
      \frac{\eps  }{(1+t+|q|)^{\frac{(n-1)}{2}-\delta}(1+|q|)^{\frac{1}{2} }}  \,\quad\text{when }\quad q<0 . \end{cases} \Big) \\
      \notag
      && +   \Big( \begin{cases}  \frac{\eps }{(1+t+|q|)^{\frac{(n-1)}{2}-\delta} (1+|q|)^{\gamma}},\quad\text{when }\quad q>0,\\
           \notag
      \frac{\eps  \cdot (1+|q|)^{\frac{1}{2} } }{(1+t+|q|)^{\frac{(n-1)}{2}-\delta}}  \,\quad\text{when }\quad q<0 . \end{cases} \Big) \\
         \notag
         && + \Big( \begin{cases}  \frac{\eps }{(1+t+|q|)^{\frac{(n-1)}{2}-\delta} (1+|q|)^{\frac{(n-1)}{2}-\delta + 1+2\gamma}},\quad\text{when }\quad q>0,\\
           \notag
      \frac{\eps  }{ (1+t+|q|)^{\frac{(n-1)}{2}-\delta} \cdot (1+ |q|)^{\frac{(n-1)}{2}-\delta}  }  \,\quad\text{when }\quad q<0 . \end{cases} \Big) \\
      \notag
   && +    \Big(  \begin{cases}  \frac{\eps }{(1+t+|q|)^{\frac{(n-1)}{2}-\delta} (1+|q|)^{\frac{(n-1)}{2}-\delta +2\gamma}},\quad\text{when }\quad q>0,\\
           \notag
      \frac{\eps \cdot (1+ |q|)  }{(1+t+|q|)^{\frac{(n-1)}{2}-\delta} \cdot (1+ |q|)^{\frac{(n-1)}{2}-\delta} } \,\quad\text{when }\quad q<0 . \end{cases} \Big) \Big) 
               \notag
  \eeaa
    \beaa
   \notag
 &&  + \sum_{|K|\leq |I|}  \Big( |   \Lie_{Z^K} A  | \Big) \cdot E (   \lfloor \frac{|I|}{2} \rfloor + \lfloor  \frac{n}{2} \rfloor  + 1) \\
 \notag
 &&  \cdot \Big( \Big(  \begin{cases}  \frac{\eps }{(1+t+|q|)^{\frac{(n-1)}{2}-\delta} (1+|q|)^{\frac{(n-1)}{2}-\delta+1+2\gamma}},\quad\text{when }\quad q>0,\\
           \notag
      \frac{\eps  }{(1+t+|q|)^{\frac{(n-1)}{2}-\delta} \cdot ( 1+|q| )^{\frac{(n-1)}{2}-\delta} } \,\quad\text{when }\quad q<0 . \end{cases} \Big) \\
      \notag 
      && + \Big(  \begin{cases}  \frac{\eps }{(1+t+|q|)^{\frac{(n-1)}{2}-\delta} (1+|q|)^{1+\gamma}},\quad\text{when }\quad q>0,\\
           \notag
      \frac{\eps  }{(1+t+|q|)^{\frac{(n-1)}{2}-\delta}(1+|q|)^{\frac{1}{2} }}  \,\quad\text{when }\quad q<0 . \end{cases} \Big) \\
         \notag
       && +  \Big(  \begin{cases}  \frac{\eps }{(1+t+|q|)^{\frac{(n-1)}{2}-\delta} (1+|q|)^{\frac{(n-1)}{2}-\delta + 2\gamma}},\quad\text{when }\quad q>0,\\
           \notag
      \frac{\eps \cdot (1+ |q|)  }{(1+t+|q|)^{\frac{(n-1)}{2}-\delta } \cdot (1+|q|)^{\frac{(n-1)}{2}-\delta}} \,\quad\text{when }\quad q<0 . \end{cases} \Big)  \\
               \notag
     && +             \Big(  \begin{cases}  \frac{\eps }{(1+t+|q|)^{\frac{(n-1)}{2}-\delta} \cdot (1+|q|)^{(n-1)-2\delta + 1+3\gamma}},\quad\text{when }\quad q>0,\\
           \notag
      \frac{\eps  \cdot(1+|q|)^{\frac{1}{2} } }{(1+t+|q|)^{\frac{(n-1)}{2}-\delta} \cdot (1+|q|)^{(n-1)-2\delta}}  \,\quad\text{when }\quad q<0 . \end{cases} \Big)  \\
      \notag
  &&     +    \Big(  \begin{cases}  \frac{\eps }{(1+t+|q|)^{\frac{(n-1)}{2}-\delta} (1+|q|)^{\frac{(n-1)}{2}-\delta +1+2\gamma}},\quad\text{when }\quad q>0,\\
           \notag
      \frac{\eps  }{(1+t+|q|)^{\frac{(n-1)}{2}-\delta} \cdot (1+|q|)^{\frac{(n-1)}{2}-\delta}}  \,\quad\text{when }\quad q<0 . \end{cases} \Big)  \\
      \notag
  && +      \Big(  \begin{cases}  \frac{\eps }{(1+t+|q|)^{\frac{(n-1)}{2}-\delta} (1+|q|)^{(n-1)-2\delta + 3\gamma}},\quad\text{when }\quad q>0,\\
           \notag
      \frac{\eps \cdot (1+|q|)^{\frac{3}{2} }  }{(1+t+|q|)^{\frac{(n-1)}{2}-\delta} \cdot (1+|q|)^{(n-1)-2\delta}}  \,\quad\text{when }\quad q<0 . \end{cases} \Big) \Big)
               \notag
  \eeaa
        \beaa
   \notag
&&  + \Big( \sum_{|K|\leq |I|}   |    \derm (\Lie_{Z^K} h ) | \Big) \cdot E (   \lfloor \frac{|I|}{2} \rfloor + \lfloor  \frac{n}{2} \rfloor  + 1)  \\
 \notag
 && \cdot  \Big( \Big( \begin{cases}  \frac{\eps }{(1+t+|q|)^{\frac{(n-1)}{2}-\delta} (1+|q|)^{1+\gamma}},\quad\text{when }\quad q>0,\\
           \notag
      \frac{\eps  }{(1+t+|q|)^{\frac{(n-1)}{2}-\delta}(1+|q|)^{\frac{1}{2} }}  \,\quad\text{when }\quad q<0 . \end{cases} \Big) \\
      \notag
    && +   \Big(  \begin{cases}  \frac{\eps }{(1+t+|q|)^{\frac{(n-1)}{2}-\delta} (1+|q|)^{\frac{(n-1)}{2}-\delta + 2\gamma}},\quad\text{when }\quad q>0,\\
           \notag
      \frac{\eps \cdot (1+ |q|)  }{(1+t+|q|)^{\frac{(n-1)}{2}-\delta} \cdot (1+ |q|)^{\frac{(n-1)}{2}-\delta} }  \,\quad\text{when }\quad q<0 . \end{cases} \Big) \\
      \notag
&& +        \Big( \begin{cases}  \frac{\eps }{(1+t+|q|)^{\frac{(n-1)}{2}-\delta} (1+|q|)^{\frac{(n-1)}{2}-\delta+1+2\gamma}},\quad\text{when }\quad q>0,\\
           \notag
      \frac{\eps  }{(1+t+|q|)^{\frac{(n-1)}{2}-\delta} \cdot (1+|q|)^{\frac{(n-1)}{2}-\delta}  } ,\quad\text{when }\quad q<0 . \end{cases} \Big) \\
     \notag
     && + \Big( \begin{cases}  \frac{\eps }{(1+t+|q|)^{\frac{(n-1)}{2}-\delta} (1+|q|)^{(n-1) - 2\de + 3\gamma}},\quad\text{when }\quad q>0,\\
           \notag
      \frac{\eps \cdot (1+|q|)^{\frac{3}{2} }  }{(1+t+|q|)^{\frac{(n-1)}{2}-\delta} \cdot (1+|q|)^{(n-1)-2\delta}}  \,\quad\text{when }\quad q<0 . \end{cases} \Big) \Big) 
               \notag
  \eeaa
       \beaa
   \notag
 && + \Big( \sum_{|K|\leq |I|}   |   \Lie_{Z^K} h  | \Big) \cdot E (   \lfloor \frac{|I|}{2} \rfloor + \lfloor  \frac{n}{2} \rfloor  + 1)  \\
 \notag
 && \cdot \Big( \Big(  \begin{cases}  \frac{\eps }{(1+t+|q|)^{(n-1)-2\delta} (1+|q|)^{2+2\gamma}},\quad\text{when }\quad q>0,\\
           \notag
      \frac{\eps  }{(1+t+|q|)^{(n-1)-2\delta}(1+|q|)}  \,\quad\text{when }\quad q<0 . \end{cases} \Big)  \\
      \notag
  && +      \Big( \begin{cases}  \frac{\eps }{(1+t+|q|)^{(n-1)-2\delta} \cdot (1+|q|)^{\frac{(n-1)}{2}-\delta + 1+3\gamma}},\quad\text{when }\quad q>0,\\
           \notag
      \frac{\eps  \cdot(1+|q|)^{\frac{1}{2} } }{(1+t+|q|)^{(n-1)-2\delta} \cdot (1+|q|)^{\frac{(n-1)}{2}-\delta} }  \,\quad\text{when }\quad q<0 . \end{cases} \Big)  \\
      \notag
   && +   \Big(  \begin{cases}  \frac{\eps }{(1+t+|q|)^{(n-1)-2\delta} (1+|q|)^{1+2\gamma}},\quad\text{when }\quad q>0,\\
           \notag
      \frac{\eps  }{(1+t+|q|)^{(n-1)-2\delta}}  \,\quad\text{when }\quad q<0 . \end{cases} \Big) \\
      \notag
   && +    \Big(  \begin{cases}  \frac{\eps }{(1+t+|q|)^{(n-1)-2\delta} (1+|q|)^{\frac{(n-1)}{2}-\delta +3\gamma}},\quad\text{when }\quad q>0,\\
           \notag
      \frac{\eps \cdot (1+|q|)^{\frac{3}{2} }  }{(1+t+|q|)^{(n-1)-2\delta} \cdot (1+ |q|)^{\frac{(n-1)}{2}-\delta} }  \,\quad\text{when }\quad q<0 . \end{cases} \Big) \Big) \; .
               \notag
  \eeaa

\end{lemma}

\begin{lemma}
We have

\beaa
   \notag
 &&  | \Lie_{Z^I}   ( g^{\la\mu} \derm_{\la}   \derm_{\mu}    h ) |   \\
\notag
 &\les &\Big( \sum_{|K|\leq |I|}    | \derm ( \Lie_{Z^K} A) |\Big) \cdot  E (   \lfloor \frac{|I|}{2} \rfloor + \lfloor  \frac{n}{2} \rfloor  + 1) \\
 \notag
 &&\cdot  \Big(  \Big(\begin{cases}  \frac{\eps }{(1+t+|q|)^{\frac{(n-1)}{2}-\delta} (1+|q|)^{1+\gamma}},\quad\text{when }\quad q>0,\\
           \notag
      \frac{\eps  }{(1+t+|q|)^{\frac{(n-1)}{2}-\delta}(1+|q|)^{\frac{1}{2} }}  \,\quad\text{when }\quad q<0 . \end{cases} \Big) \\
      \notag
      && + \Big( \begin{cases}  \frac{\eps }{(1+t+|q|)^{\frac{(n-1)}{2}-\delta} (1+|q|)^{\frac{(n-1)}{2}-\delta +2\gamma}},\quad\text{when }\quad q>0,\\
           \notag
      \frac{\eps \cdot (1+ |q|)  }{(1+t+|q|)^{\frac{(n-1)}{2}-\delta} \cdot (1+|q|)^{\frac{(n-1)}{2}-\delta}}  \,\quad\text{when }\quad q<0 . \end{cases} \Big) \\
      \notag
    &&  +  \Big( \begin{cases}  \frac{\eps }{(1+t+|q|)^{\frac{(n-1)}{2}-\delta} (1+|q|)^{\frac{(n-1)}{2}-\delta+1+2\gamma}},\quad\text{when }\quad q>0,\\
           \notag
      \frac{\eps  }{(1+t+|q|)^{\frac{(n-1)}{2}-\delta} \cdot (1+|q|)^{\frac{(n-1)}{2}-\delta}}  \,\quad\text{when }\quad q<0 . \end{cases} \Big) \\
      \notag
   && + \Big(  \begin{cases}  \frac{\eps }{(1+t+|q|)^{\frac{(n-1)}{2}-\delta} (1+|q|)^{(n-1)-2\delta +3\gamma}},\quad\text{when }\quad q>0,\\
           \notag
      \frac{\eps \cdot (1+|q|)^{\frac{3}{2} }  }{(1+t+|q|)^{\frac{(n-1)}{2}-\delta} \cdot (1+|q|)^{(n-1)-2\delta}}  \,\quad\text{when }\quad q<0 . \end{cases} \Big) \Big) 
               \notag
\eeaa

\beaa
\notag
 &&+ \Big( \sum_{|K|\leq |I|}    |  \Lie_{Z^K} A | \Big) \cdot E (   \lfloor \frac{|I|}{2} \rfloor + \lfloor  \frac{n}{2} \rfloor  + 1)   \\
 \notag
 &&   \cdot \Big(   \Big(  \begin{cases}  \frac{\eps }{(1+t+|q|)^{(n-1)-2\delta} (1+|q|)^{1+2\gamma}},\quad\text{when }\quad q>0,\\
           \notag
      \frac{\eps  }{(1+t+|q|)^{(n-1)-2\delta}}  \,\quad\text{when }\quad q<0 . \end{cases} \Big) \\
      \notag
  &&  +   \Big(  \begin{cases}  \frac{\eps }{(1+t+|q|)^{(n-1)-2\delta} (1+|q|)^{\frac{(n-1)}{2}-\delta + 3\gamma}},\quad\text{when }\quad q>0,\\
           \notag
      \frac{\eps \cdot (1+|q|)^{\frac{3}{2} }  }{(1+t+|q|)^{(n-1)-2\delta } \cdot (1+|q|)^{\frac{(n-1)}{2}-\delta }}  \,\quad\text{when }\quad q<0 . \end{cases} \Big) \\
               \notag
 &&   +   \Big( \begin{cases}  \frac{\eps }{(1+t+|q|)^{(n-1)-2\delta} (1+|q|)^{\frac{(n-1)}{2}-\delta +1+3\gamma}},\quad\text{when }\quad q>0,\\
           \notag
      \frac{\eps  \cdot(1+|q|)^{\frac{1}{2} } }{(1+t+|q|)^{(n-1)-2\delta} \cdot (1+|q|)^{\frac{(n-1)}{2}-\delta}}  \,\quad\text{when }\quad q<0 . \end{cases} \Big)  \\
      \notag
   &&  + \Big(  \begin{cases}  \frac{\eps }{(1+t+|q|)^{(n-1)-2\delta} \cdot (1+|q|)^{(n-1)-2\delta +4\gamma}},\quad\text{when }\quad q>0,\\
           \notag
      \frac{\eps \cdot (1+|q|)^{2 }  }{(1+t+|q|)^{(n-1)-2\delta} \cdot (1+|q|)^{(n-1)-2\delta} }  \,\quad\text{when }\quad q<0 . \end{cases} \Big) \Big)  
               \notag
 \eeaa
 
\beaa
\notag
 &&+ \Big( \sum_{|K|\leq |I|}    | \derm(  \Lie_{Z^K} h ) | \Big) \cdot E (   \lfloor \frac{|I|}{2} \rfloor + \lfloor  \frac{n}{2} \rfloor  + 1) \\
 \notag
 && \cdot \Big(   \Big(  \begin{cases}  \frac{\eps }{(1+t+|q|)^{\frac{(n-1)}{2}-\delta} (1+|q|)^{1+\gamma}},\quad\text{when }\quad q>0,\\
 \notag
      \frac{\eps  }{(1+t+|q|)^{\frac{(n-1)}{2}-\delta}(1+|q|)^{\frac{1}{2} }}  \,\quad\text{when }\quad q<0 . \end{cases} \Big) \\
      \notag
   && + \Big(  \begin{cases}  \frac{\eps }{(1+t+|q|)^{\frac{(n-1)}{2}-\delta} (1+|q|)^{\frac{(n-1)}{2}-\delta+1+2\gamma}},\quad\text{when }\quad q>0,\\
           \notag
      \frac{\eps  }{(1+t+|q|)^{\frac{(n-1)}{2}-\delta} \cdot (1+|q|)^{\frac{(n-1)}{2}-\delta}}  \,\quad\text{when }\quad q<0 . \end{cases} \Big) \Big) 
      \notag
 \eeaa

\beaa
\notag
 &&+ \Big(   \sum_{|K|\leq |I|}    |   \Lie_{Z^K} h  |\Big) \cdot  E (   \lfloor \frac{|I|}{2} \rfloor + \lfloor  \frac{n}{2} \rfloor  + 1) \\
 \notag
   &&\Big(   \Big(  \begin{cases}  \frac{\eps }{(1+t+|q|)^{(n-1)-2\delta} (1+|q|)^{2+2\gamma}},\quad\text{when }\quad q>0,\\
           \notag
      \frac{\eps  }{(1+t+|q|)^{(n-1)-2\delta}(1+|q|)}  \,\quad\text{when }\quad q<0 . \end{cases} \Big) \\
      \notag
      && + \Big(  \begin{cases}  \frac{\eps }{(1+t+|q|)^{(n-1)-2\delta} (1+|q|)^{\frac{(n-1)}{2}-\delta+1+3\gamma}},\quad\text{when }\quad q>0,\\
           \notag
      \frac{\eps  \cdot(1+|q|)^{\frac{1}{2} } }{(1+t+|q|)^{(n-1)-2\delta} \cdot (1+|q|)^{\frac{(n-1)}{2}-\delta}}  \,\quad\text{when }\quad q<0 . \end{cases} \Big)  \\
      \notag
  &&  + \Big(  \begin{cases}  \frac{\eps }{(1+t+|q|)^{(n-1)-2\delta} \cdot (1+|q|)^{(n-1)-2\delta +4\gamma}},\quad\text{when }\quad q>0,\\
           \notag
      \frac{\eps \cdot (1+|q|)^{2 }  }{(1+t+|q|)^{(n-1)-2\delta} \cdot (1+|q|)^{(n-1)-2\delta} }  \,\quad\text{when }\quad q<0 . \end{cases} \Big) \Big)  \; .\\
               \notag
 \eeaa

\end{lemma}

Now, we look at the case, where $n\geq4$. 
 
\begin{lemma}
For $n \geq 4$, we have
 \beaa
 \notag
&& \frac{(1+t )}{\eps} \cdot |  \Lie_{Z^I}   ( g^{\la\mu} \derm_{\la}   \derm_{\mu}   A  )   |^2 \, \\
   \notag
   &\les&  \sum_{|K|\leq |I|}  \Big( |    \derm (\Lie_{Z^K} A ) |^2 +  |    \derm (\Lie_{Z^K} h ) |^2 \Big) \cdot E (   \lfloor \frac{|I|}{2} \rfloor + \lfloor  \frac{n}{2} \rfloor  + 1)\\
 \notag
   && \cdot \Big(  \begin{cases}  \frac{\eps }{(1+t+|q|)^{2-2\delta} (1+|q|)^{2\gamma}},\quad\text{when }\quad q>0,\\
           \notag
      \frac{\eps  }{(1+t+|q|)^{2-2\delta} \cdot ( 1+|q| )^{-1} } \,\quad\text{when }\quad q<0 , \end{cases} \Big) 
      \notag 
   \eeaa
     \beaa
     \notag
   && + \sum_{|K|\leq |I|}  \Big( |   \Lie_{Z^K} A  |^2 +|   \Lie_{Z^K} h  |^2 \Big) \cdot E (   \lfloor \frac{|I|}{2} \rfloor + \lfloor  \frac{n}{2} \rfloor  + 1) \\
 \notag
   && \cdot \Big(  \begin{cases}  \frac{\eps }{(1+t+|q|)^{2-2\delta} (1+|q|)^{2+2\gamma}},\quad\text{when }\quad q>0,\\
           \notag
      \frac{\eps  }{(1+t+|q|)^{2-2\delta} \cdot ( 1+|q| )^{1-2\de} } \,\quad\text{when }\quad q<0 . \end{cases} \Big) \; .
      \notag 
   \eeaa

\end{lemma}

\begin{proof}

For $n \geq 4$, we examine one by one the terms in $ \frac{(1+t )}{\eps} \cdot |  \Lie_{Z^I}   ( g^{\la\mu} \derm_{\la}   \derm_{\mu}   A  )   |^2 $, we get

\beaa
   \notag
 &&   \sum_{|K|\leq |I|}  \Big( |    \derm (\Lie_{Z^K} A ) |^2 \Big) \cdot E (   \lfloor \frac{|I|}{2} \rfloor + \lfloor  \frac{n}{2} \rfloor  + 1)\\
 \notag
 && \cdot \Big( \begin{cases}  \frac{\eps }{(1+t+|q|)^{2-2\delta} (1+|q|)^{2+2\gamma}},\quad\text{when }\quad q>0,\\
           \notag
      \frac{\eps  }{(1+t+|q|)^{2-2\delta}(1+|q|)}  \,\quad\text{when }\quad q<0 . \end{cases} \Big) \\
      \notag
      && +   \Big( \begin{cases}  \frac{\eps }{(1+t+|q|)^{2-2\delta} (1+|q|)^{2\gamma}},\quad\text{when }\quad q>0,\\
           \notag
      \frac{\eps  \cdot (1+|q|)^{ } }{(1+t+|q|)^{2-2\delta}}  \,\quad\text{when }\quad q<0 . \end{cases} \Big) \\
         \notag
         && + \Big( \begin{cases}  \frac{\eps }{(1+t+|q|)^{2-2\delta} (1+|q|)^{5 + 2(\ga-\delta) +2\gamma}},\quad\text{when }\quad q>0,\\
           \notag
      \frac{\eps  }{ (1+t+|q|)^{2-2\delta} \cdot (1+ |q|)^{3-2\delta}  }  \,\quad\text{when }\quad q<0 . \end{cases} \Big) \\
      \notag
   && +    \Big(  \begin{cases}  \frac{\eps }{(1+t+|q|)^{2-2\delta} (1+|q|)^{3+2(\ga-\delta)+2\gamma}},\quad\text{when }\quad q>0,\\
           \notag
      \frac{\eps \cdot (1+ |q|)  }{(1+t+|q|)^{2-2\delta} \cdot (1+ |q|)^{3-2\delta} } \,\quad\text{when }\quad q<0 . \end{cases} \Big) \Big) 
               \notag
  \eeaa

     \beaa
     \notag
   &\les&  \sum_{|K|\leq |I|}  \Big( |    \derm (\Lie_{Z^K} A ) |^2 \Big) \cdot E (   \lfloor \frac{|I|}{2} \rfloor + \lfloor  \frac{n}{2} \rfloor  + 1)\\
 \notag
   && \cdot \Big(  \begin{cases}  \frac{\eps }{(1+t+|q|)^{2-2\delta} (1+|q|)^{2\gamma}},\quad\text{when }\quad q>0,\\
           \notag
      \frac{\eps  }{(1+t+|q|)^{2-2\delta} \cdot ( 1+|q| )^{-1} } \,\quad\text{when }\quad q<0 . \end{cases} \Big) \;.
      \notag 
   \eeaa

And,
       \beaa
   \notag
 &&   \sum_{|K|\leq |I|}  \Big( |   \Lie_{Z^K} A  |^2 \Big) \cdot E (   \lfloor \frac{|I|}{2} \rfloor + \lfloor  \frac{n}{2} \rfloor  + 1) \\
 \notag
 &&  \cdot \Big( \Big(  \begin{cases}  \frac{\eps }{(1+t+|q|)^{2-2\delta} (1+|q|)^{5+2(\ga-\delta)+2\gamma}},\quad\text{when }\quad q>0,\\
           \notag
      \frac{\eps  }{(1+t+|q|)^{2-2\delta} \cdot ( 1+|q| )^{3-2\delta} } \,\quad\text{when }\quad q<0 . \end{cases} \Big) \\
      \notag 
      && + \Big(  \begin{cases}  \frac{\eps }{(1+t+|q|)^{2-2\delta} (1+|q|)^{2+2\gamma}},\quad\text{when }\quad q>0,\\
           \notag
      \frac{\eps  }{(1+t+|q|)^{2-2\delta}(1+|q|)^{ }}  \,\quad\text{when }\quad q<0 . \end{cases} \Big) \\
         \notag
       && +  \Big(  \begin{cases}  \frac{\eps }{(1+t+|q|)^{2-2\delta} (1+|q|)^{3+2(\ga-\delta)+ 2\gamma}},\quad\text{when }\quad q>0,\\
           \notag
      \frac{\eps \cdot (1+ |q|)^2  }{(1+t+|q|)^{2-2\delta } \cdot (1+|q|)^{3-2\delta}} \,\quad\text{when }\quad q<0 . \end{cases} \Big)  \\
               \notag
     && +             \Big(  \begin{cases}  \frac{\eps }{(1+t+|q|)^{2-2\delta} \cdot (1+|q|)^{8+4(\ga-\delta) +2\gamma}},\quad\text{when }\quad q>0,\\
           \notag
      \frac{\eps  \cdot(1+|q|)^{ } }{(1+t+|q|)^{2-2\delta} \cdot (1+|q|)^{6-4\delta}}  \,\quad\text{when }\quad q<0 . \end{cases} \Big)  \\
      \notag
  &&     +    \Big(  \begin{cases}  \frac{\eps }{(1+t+|q|)^{2-2\delta} (1+|q|)^{5+2(\ga-\delta) +2\gamma}},\quad\text{when }\quad q>0,\\
           \notag
      \frac{\eps  }{(1+t+|q|)^{2-2\delta} \cdot (1+|q|)^{3-2\delta}}  \,\quad\text{when }\quad q<0 . \end{cases} \Big)  \\
      \notag
  && +      \Big(  \begin{cases}  \frac{\eps }{(1+t+|q|)^{2-2\delta} (1+|q|)^{6+4(\ga-\delta) + 2\gamma}},\quad\text{when }\quad q>0,\\
           \notag
      \frac{\eps \cdot (1+|q|)^{3 }  }{(1+t+|q|)^{2-2\delta} \cdot (1+|q|)^{6-4\delta}}  \,\quad\text{when }\quad q<0 . \end{cases} \Big) 
               \notag
  \eeaa
 
     \beaa
     \notag
   &\les&  \sum_{|K|\leq |I|}  \Big( |   \Lie_{Z^K} A  |^2 \Big) \cdot E (   \lfloor \frac{|I|}{2} \rfloor + \lfloor  \frac{n}{2} \rfloor  + 1) \\
 \notag
   && \cdot \Big(  \begin{cases}  \frac{\eps }{(1+t+|q|)^{2-2\delta} (1+|q|)^{2+2\gamma}},\quad\text{when }\quad q>0,\\
           \notag
      \frac{\eps  }{(1+t+|q|)^{2-2\delta} \cdot ( 1+|q| )^{1-2\de} } \,\quad\text{when }\quad q<0 . \end{cases} \Big) 
      \notag 
   \eeaa
   
       (where we used the fact that $\ga \geq \delta$).

 And,
 
           \beaa
   \notag
&&   \Big( \sum_{|K|\leq |I|}   |    \derm (\Lie_{Z^K} h ) |^2 \Big) \cdot E (   \lfloor \frac{|I|}{2} \rfloor + \lfloor  \frac{n}{2} \rfloor  + 1)  \\
 \notag
 && \cdot  \Big( \Big( \begin{cases}  \frac{\eps }{(1+t+|q|)^{2-2\delta} (1+|q|)^{2+2\gamma}},\quad\text{when }\quad q>0,\\
           \notag
      \frac{\eps  }{(1+t+|q|)^{2-2\delta}(1+|q|)^{ }}  \,\quad\text{when }\quad q<0 . \end{cases} \Big) \\
      \notag
    && +   \Big(  \begin{cases}  \frac{\eps }{(1+t+|q|)^{2-2\delta} (1+|q|)^{3+2(\ga-\delta) + 2\gamma}},\quad\text{when }\quad q>0,\\
           \notag
      \frac{\eps \cdot (1+ |q|)^2  }{(1+t+|q|)^{2-2\delta} \cdot (1+ |q|)^{3-2\delta} }  \,\quad\text{when }\quad q<0 . \end{cases} \Big) \\
      \notag
&& +        \Big( \begin{cases}  \frac{\eps }{(1+t+|q|)^{2-2\delta} (1+|q|)^{5+2(\ga-\delta)+2\gamma}},\quad\text{when }\quad q>0,\\
           \notag
      \frac{\eps  }{(1+t+|q|)^{2-2\delta} \cdot (1+|q|)^{3-2\delta}  } ,\quad\text{when }\quad q<0 . \end{cases} \Big) \\
     \notag
     && + \Big( \begin{cases}  \frac{\eps }{(1+t+|q|)^{2-2\delta} (1+|q|)^{6 +4(\ga-\delta)+ 2\gamma}},\quad\text{when }\quad q>0,\\
           \notag
      \frac{\eps \cdot (1+|q|)^{3 }  }{(1+t+|q|)^{2-2\delta} \cdot (1+|q|)^{6-4\delta}}  \,\quad\text{when }\quad q<0 . \end{cases} \Big) \Big) 
               \notag
  \eeaa
 
      \beaa
     \notag
   &\les&  \Big( \sum_{|K|\leq |I|}   |    \derm (\Lie_{Z^K} h ) |^2 \Big) \cdot E (   \lfloor \frac{|I|}{2} \rfloor + \lfloor  \frac{n}{2} \rfloor  + 1)  \\
 \notag
   && \cdot \Big(  \begin{cases}  \frac{\eps }{(1+t+|q|)^{2-2\delta} (1+|q|)^{2+2\gamma}},\quad\text{when }\quad q>0,\\
           \notag
      \frac{\eps  }{(1+t+|q|)^{2-2\delta} \cdot ( 1+|q| )^{1-2\de} } \,\quad\text{when }\quad q<0 . \end{cases} \Big) 
      \notag 
   \eeaa
      (using the fact that $\ga \geq \delta$).

Also,
  
       \beaa
   \notag
 &&  \Big( \sum_{|K|\leq |I|}   |   \Lie_{Z^K} h  |^2 \Big) \cdot E (   \lfloor \frac{|I|}{2} \rfloor + \lfloor  \frac{n}{2} \rfloor  + 1)  \\
 \notag
 && \cdot \Big( \Big(  \begin{cases}  \frac{\eps }{(1+t+|q|)^{5-4\delta} (1+|q|)^{4+4\gamma}},\quad\text{when }\quad q>0,\\
           \notag
      \frac{\eps  }{(1+t+|q|)^{5-4\delta}(1+|q|)^{2}}  \,\quad\text{when }\quad q<0 . \end{cases} \Big)  \\
      \notag
  && +      \Big( \begin{cases}  \frac{\eps }{(1+t+|q|)^{5-4\delta} \cdot (1+|q|)^{5+2(\ga-\delta) +4\gamma}},\quad\text{when }\quad q>0,\\
           \notag
      \frac{\eps  \cdot(1+|q|)^{ } }{(1+t+|q|)^{5-4\delta} \cdot (1+|q|)^{3-2\delta} }  \,\quad\text{when }\quad q<0 . \end{cases} \Big)  \\
      \notag
   && +   \Big(  \begin{cases}  \frac{\eps }{(1+t+|q|)^{5-4\delta} (1+|q|)^{2+4\gamma}},\quad\text{when }\quad q>0,\\
           \notag
      \frac{\eps  }{(1+t+|q|)^{5-4\delta}}  \,\quad\text{when }\quad q<0 . \end{cases} \Big) \\
      \notag
   && +    \Big(  \begin{cases}  \frac{\eps }{(1+t+|q|)^{5-4\delta} (1+|q|)^{3+2(\ga-\delta) +4\gamma}},\quad\text{when }\quad q>0,\\
           \notag
      \frac{\eps \cdot (1+|q|)^{3 }  }{(1+t+|q|)^{5-4\delta} \cdot (1+ |q|)^{3-2\delta} }  \,\quad\text{when }\quad q<0 . \end{cases} \Big) \Big) \; .
               \notag
  \eeaa
  
     \beaa
     \notag
   &\les&  \Big( \sum_{|K|\leq |I|}   |   \Lie_{Z^K} h  |^2 \Big) \cdot E (   \lfloor \frac{|I|}{2} \rfloor + \lfloor  \frac{n}{2} \rfloor  + 1)   \\
 \notag
   && \cdot \Big(  \begin{cases}  \frac{\eps }{(1+t+|q|)^{5-4\delta} (1+|q|)^{2+4\gamma}},\quad\text{when }\quad q>0,\\
           \notag
      \frac{\eps  }{(1+t+|q|)^{5-4\delta} (1+|q|)^{-2\de} } \,\quad\text{when }\quad q<0 . \end{cases} \Big) \;.
      \notag 
   \eeaa
   
\end{proof}

\begin{lemma}
For $n \geq 4$, 
\beaa
   \notag
 &&  \frac{(1+t )}{\eps} \cdot  | \Lie_{Z^I}   ( g^{\la\mu} \derm_{\la}   \derm_{\mu}    h ) |^2   \\
     \notag
   &\les&  \sum_{|K|\leq |I|}  \Big( |    \derm (\Lie_{Z^K} A ) |^2 + |    \derm (\Lie_{Z^K} h ) |^2 \Big) \cdot E (   \lfloor \frac{|I|}{2} \rfloor + \lfloor  \frac{n}{2} \rfloor  + 1)\\
 \notag
   && \cdot \Big(  \begin{cases}  \frac{\eps }{(1+t+|q|)^{2-2\delta} (1+|q|)^{2+2\gamma}},\quad\text{when }\quad q>0,\\
           \notag
      \frac{\eps  }{(1+t+|q|)^{2-2\delta} \cdot ( 1+|q| )^{1-2\de} } \,\quad\text{when }\quad q<0 . \end{cases} \Big) \\
      \notag 
   &&+  \sum_{|K|\leq |I|}  \Big( |   \Lie_{Z^K} A  |^2 + |   \Lie_{Z^K} h  |^2 \Big) \cdot E (   \lfloor \frac{|I|}{2} \rfloor + \lfloor  \frac{n}{2} \rfloor  + 1)\\
 \notag
   && \cdot \Big(  \begin{cases}  \frac{\eps }{(1+t+|q|)^{5-4\delta} (1+|q|)^{2+4\gamma}},\quad\text{when }\quad q>0,\\
           \notag
      \frac{\eps  }{(1+t+|q|)^{5-4\delta} (1+|q|)^{-2\de} } \,\quad\text{when }\quad q<0 . \end{cases} \Big) \; .
      \notag 
   \eeaa
   
\end{lemma}

\begin{proof}

For $n \geq 4$, we examine the terms in $\frac{(1+t )}{\eps} \cdot  | \Lie_{Z^I}   ( g^{\la\mu} \derm_{\la}   \derm_{\mu}    h ) |^2 $, one by one. We have

\beaa
   \notag
&& \Big( \sum_{|K|\leq |I|}    | \derm ( \Lie_{Z^K} A) |^2 \Big) \cdot  E (   \lfloor \frac{|I|}{2} \rfloor + \lfloor  \frac{n}{2} \rfloor  + 1) \\
 \notag
 &&\cdot  \Big(  \Big(\begin{cases}  \frac{\eps }{(1+t+|q|)^{2-2\delta} (1+|q|)^{2+2\gamma}},\quad\text{when }\quad q>0,\\
           \notag
      \frac{\eps  }{(1+t+|q|)^{2-2\delta}(1+|q|)^{ }}  \,\quad\text{when }\quad q<0 . \end{cases} \Big) \\
      \notag
      && + \Big( \begin{cases}  \frac{\eps }{(1+t+|q|)^{2-2\delta} (1+|q|)^{3+2(\ga-\delta) +2\gamma}},\quad\text{when }\quad q>0,\\
           \notag
      \frac{\eps \cdot (1+ |q|)^2  }{(1+t+|q|)^{2-2\delta} \cdot (1+|q|)^{3-2\delta}}  \,\quad\text{when }\quad q<0 . \end{cases} \Big) \\
      \notag
    &&  +  \Big( \begin{cases}  \frac{\eps }{(1+t+|q|)^{2-2\delta} (1+|q|)^{5+2(\ga-\delta)+2\gamma}},\quad\text{when }\quad q>0,\\
           \notag
      \frac{\eps  }{(1+t+|q|)^{2-2\delta} \cdot (1+|q|)^{3-2\delta}}  \,\quad\text{when }\quad q<0 . \end{cases} \Big) \\
      \notag
   && + \Big(  \begin{cases}  \frac{\eps }{(1+t+|q|)^{2-2\delta} (1+|q|)^{6+4(\ga-\delta) +2\gamma}},\quad\text{when }\quad q>0,\\
           \notag
      \frac{\eps \cdot (1+|q|)^{3 }  }{(1+t+|q|)^{2-2\delta} \cdot (1+|q|)^{6-4\delta}}  \,\quad\text{when }\quad q<0 . \end{cases} \Big) \Big) 
               \notag
\eeaa

     \beaa
     \notag
   &\les&  \sum_{|K|\leq |I|}  \Big( |    \derm (\Lie_{Z^K} A ) |^2 \Big) \cdot E (   \lfloor \frac{|I|}{2} \rfloor + \lfloor  \frac{n}{2} \rfloor  + 1)\\
 \notag
   && \cdot \Big(  \begin{cases}  \frac{\eps }{(1+t+|q|)^{2-2\delta} (1+|q|)^{2+2\gamma}},\quad\text{when }\quad q>0,\\
           \notag
      \frac{\eps  }{(1+t+|q|)^{2-2\delta} \cdot ( 1+|q| )^{1-2\de} } \,\quad\text{when }\quad q<0 . \end{cases} \Big) \; .
      \notag 
   \eeaa

And,
 \beaa
\notag
 && \Big( \sum_{|K|\leq |I|}    |  \Lie_{Z^K} A |^2 \Big) \cdot E (   \lfloor \frac{|I|}{2} \rfloor + \lfloor  \frac{n}{2} \rfloor  + 1)   \\
 \notag
 &&   \cdot \Big(   \Big(  \begin{cases}  \frac{\eps }{(1+t+|q|)^{5-4\delta} (1+|q|)^{2+4\gamma}},\quad\text{when }\quad q>0,\\
           \notag
      \frac{\eps  }{(1+t+|q|)^{5-4\delta}}  \,\quad\text{when }\quad q<0 . \end{cases} \Big) \\
      \notag
  &&  +   \Big(  \begin{cases}  \frac{\eps }{(1+t+|q|)^{5-4\delta} (1+|q|)^{3+2(\ga-\delta) + 4\gamma}},\quad\text{when }\quad q>0,\\
           \notag
      \frac{\eps \cdot (1+|q|)^{3 }  }{(1+t+|q|)^{5-4\delta } \cdot (1+|q|)^{3-2\delta }}  \,\quad\text{when }\quad q<0 . \end{cases} \Big) \\
               \notag
 &&   +   \Big( \begin{cases}  \frac{\eps }{(1+t+|q|)^{5-4\delta} (1+|q|)^{5+2(\ga-\delta) +4\gamma}},\quad\text{when }\quad q>0,\\
           \notag
      \frac{\eps  \cdot(1+|q|)^{ } }{(1+t+|q|)^{5-4\delta} \cdot (1+|q|)^{3-2\delta}}  \,\quad\text{when }\quad q<0 . \end{cases} \Big)  \\
      \notag
   &&  + \Big(  \begin{cases}  \frac{\eps }{(1+t+|q|)^{5-4\delta} \cdot (1+|q|)^{6+4(\ga-\delta) +4\gamma}},\quad\text{when }\quad q>0,\\
           \notag
      \frac{\eps \cdot (1+|q|)^{4 }  }{(1+t+|q|)^{5-4\delta} \cdot (1+|q|)^{6-4\delta} }  \,\quad\text{when }\quad q<0 . \end{cases} \Big) \Big)  
               \notag
 \eeaa

     \beaa
     \notag
   &\les&  \sum_{|K|\leq |I|}  \Big( |   \Lie_{Z^K} A  |^2 \Big) \cdot E (   \lfloor \frac{|I|}{2} \rfloor + \lfloor  \frac{n}{2} \rfloor  + 1)\\
 \notag
   && \cdot \Big(  \begin{cases}  \frac{\eps }{(1+t+|q|)^{5-4\delta} (1+|q|)^{2+4\gamma}},\quad\text{when }\quad q>0,\\
           \notag
      \frac{\eps  }{(1+t+|q|)^{5-4\delta} (1+|q|)^{-2\de} } \,\quad\text{when }\quad q<0 . \end{cases} \Big) \; .
      \notag 
   \eeaa

And,
 \beaa
\notag
 && \Big( \sum_{|K|\leq |I|}    | \derm(  \Lie_{Z^K} h ) |^2 \Big) \cdot E (   \lfloor \frac{|I|}{2} \rfloor + \lfloor  \frac{n}{2} \rfloor  + 1) \\
 \notag
 && \cdot \Big(   \Big(  \begin{cases}  \frac{\eps }{(1+t+|q|)^{2-2\delta} (1+|q|)^{2+2\gamma}},\quad\text{when }\quad q>0,\\
 \notag
      \frac{\eps  }{(1+t+|q|)^{2-2\delta}(1+|q|)^{ }}  \,\quad\text{when }\quad q<0 . \end{cases} \Big) \\
      \notag
   && + \Big(  \begin{cases}  \frac{\eps }{(1+t+|q|)^{2-2\delta} (1+|q|)^{4+2(\ga-\delta)+2\gamma}},\quad\text{when }\quad q>0,\\
           \notag
      \frac{\eps  }{(1+t+|q|)^{2-2\delta} \cdot (1+|q|)^{3-2\delta}}  \,\quad\text{when }\quad q<0 . \end{cases} \Big) \Big) 
      \notag
 \eeaa

     \beaa
     \notag
   &\les&  \sum_{|K|\leq |I|}  \Big( |    \derm (\Lie_{Z^K} h ) |^2 \Big) \cdot E (   \lfloor \frac{|I|}{2} \rfloor + \lfloor  \frac{n}{2} \rfloor  + 1)\\
 \notag
   && \cdot \Big(  \begin{cases}  \frac{\eps }{(1+t+|q|)^{2-2\delta} (1+|q|)^{2+2\gamma}},\quad\text{when }\quad q>0,\\
           \notag
      \frac{\eps  }{(1+t+|q|)^{2-2\delta} \cdot ( 1+|q| )^{} } \,\quad\text{when }\quad q<0 . \end{cases} \Big) \; .
      \notag 
   \eeaa

Also,
 \beaa
\notag
 && \Big(   \sum_{|K|\leq |I|}    |   \Lie_{Z^K} h  |^2 \Big) \cdot  E (   \lfloor \frac{|I|}{2} \rfloor + \lfloor  \frac{n}{2} \rfloor  + 1) \\
 \notag
   &&\Big(   \Big(  \begin{cases}  \frac{\eps }{(1+t+|q|)^{5-4\delta} (1+|q|)^{4+4\gamma}},\quad\text{when }\quad q>0,\\
           \notag
      \frac{\eps  }{(1+t+|q|)^{5-4\delta}(1+|q|)^2}  \,\quad\text{when }\quad q<0 . \end{cases} \Big) \\
      \notag
      && + \Big(  \begin{cases}  \frac{\eps }{(1+t+|q|)^{5-4\delta} (1+|q|)^{5+2(\ga-\delta)+4\gamma}},\quad\text{when }\quad q>0,\\
           \notag
      \frac{\eps  \cdot(1+|q|)^{ } }{(1+t+|q|)^{5-4\delta} \cdot (1+|q|)^{3-2\delta}}  \,\quad\text{when }\quad q<0 . \end{cases} \Big)  \\
      \notag
  &&  + \Big(  \begin{cases}  \frac{\eps }{(1+t+|q|)^{5-4\delta} \cdot (1+|q|)^{6+4(\ga-\delta) +4\gamma}},\quad\text{when }\quad q>0,\\
           \notag
      \frac{\eps \cdot (1+|q|)^{4 }  }{(1+t+|q|)^{5-4\delta} \cdot (1+|q|)^{6-4\delta} }  \,\quad\text{when }\quad q<0 . \end{cases} \Big) \Big)  
               \notag
 \eeaa

     \beaa
     \notag
   &\les&  \sum_{|K|\leq |I|}  \Big( |   \Lie_{Z^K} h  |^2 \Big) \cdot E (   \lfloor \frac{|I|}{2} \rfloor + \lfloor  \frac{n}{2} \rfloor  + 1)\\
 \notag
   && \cdot \Big(  \begin{cases}  \frac{\eps }{(1+t+|q|)^{5-4\delta} (1+|q|)^{4+4\gamma}},\quad\text{when }\quad q>0,\\
           \notag
      \frac{\eps  }{(1+t+|q|)^{5-4\delta} \cdot ( 1+|q| )^{2-4\de} } \,\quad\text{when }\quad q<0 . \end{cases} \Big) \; .
   \eeaa

\end{proof}

\section{The proof of exterior stability for $n\geq 4$}

Now, we look $n \geq 4$ and $\de = 0$ and we are interested only in the exterior region $\overline{C}$. We fix $q_0$ such that $\overline{C} \subseteq \{ q \geq q_0 \}$. 

\subsection{Using the Hardy type inequality for the space-time integrals of the source terms for $n \geq 4$}\
      
\begin{lemma}
For $n \geq 4$, $\de = 0$, for $q \geq q_0$, we have

   \beaa
 \notag
&&  \frac{(1+t )^{1+\la}}{\eps} \cdot |  \Lie_{Z^I}   ( g^{\la\mu} \derm_{\la}   \derm_{\mu}    A ) |^2 \, \\
   \notag
   &\les&  C(q_0) \cdot E^{ext} (   \lfloor \frac{|I|}{2} \rfloor + \lfloor  \frac{n}{2} \rfloor  + 1) \\
   \notag
 &&   \cdot \sum_{|K|\leq |I|}  \Big( |    \derm (\Lie_{Z^K} A ) |^2 +  |    \derm (\Lie_{Z^K} h ) |^2 \Big)  \cdot   \Big(  \frac{\eps  }{(1+t+|q|)^{2-\la} \cdot (1+|q|)^{2\gamma} } \Big) \\
     \notag
   && + C(q_0)  \cdot E^{ext} (   \lfloor \frac{|I|}{2} \rfloor + \lfloor  \frac{n}{2} \rfloor  + 1)  \\
   \notag
&&  \cdot  \sum_{|K|\leq |I|}  \Big( |   \Lie_{Z^K} A  |^2 +|   \Lie_{Z^K} h  |^2 \Big)\cdot \Big( \frac{\eps }{(1+t+|q|)^{2-\la} \cdot (1+|q|)^{2+2\gamma} }  \Big)  \; .\\
      \notag 
   \eeaa

\end{lemma}

\begin{proof}
Based on what we have shown using the Klainerman-Sobolev inequality in the exterior, we get for $n \geq 4$, $\de =0$, that for all points in the exterior region $\overline{C}$, we have
 \beaa
 \notag
&& \frac{(1+t )}{\eps} \cdot |  \Lie_{Z^I}   ( g^{\la\mu} \derm_{\la}   \derm_{\mu}    A ) |^2 \, \\
   \notag
   &\les&  \sum_{|K|\leq |I|}  \Big( |    \derm (\Lie_{Z^K} A ) |^2 +  |    \derm (\Lie_{Z^K} h ) |^2 \Big) \cdot E^{ext} (   \lfloor \frac{|I|}{2} \rfloor + \lfloor  \frac{n}{2} \rfloor  + 1)\\
 \notag
   && \cdot \Big(  \begin{cases}  \frac{\eps }{(1+t+|q|)^{2} \cdot (1+|q|)^{2\gamma}},\quad\text{when }\quad q>0,\\
           \notag
      \frac{\eps  }{(1+t+|q|)^{2} \cdot ( 1+|q| )^{-1} } \,\quad\text{when }\quad q<0 . \end{cases} \Big) \\
     \notag
   && + \sum_{|K|\leq |I|}  \Big( |   \Lie_{Z^K} A  |^2 +|   \Lie_{Z^K} h  |^2 \Big) \cdot E^{ext} (   \lfloor \frac{|I|}{2} \rfloor + \lfloor  \frac{n}{2} \rfloor  + 1) \\
 \notag
   && \cdot \Big(  \begin{cases}  \frac{\eps }{(1+t+|q|)^{2} \cdot (1+|q|)^{2+2\gamma}},\quad\text{when }\quad q>0,\\
           \notag
      \frac{\eps  }{(1+t+|q|)^{2} \cdot ( 1+|q| )^{1} } \,\quad\text{when }\quad q<0 . \end{cases} \Big) \; .
      \notag 
   \eeaa
  Hence, for $q \geq q_0$, we have
   \beaa
 \notag
&&  \frac{(1+t )}{\eps} \cdot  |  \Lie_{Z^I}   ( g^{\la\mu} \derm_{\la}   \derm_{\mu}    A ) |^2 \, \\
   \notag
   &\les&  \sum_{|K|\leq |I|}  \Big( |    \derm (\Lie_{Z^K} A ) |^2 +  |    \derm (\Lie_{Z^K} h ) |^2 \Big)\cdot C(q_0) \cdot E^{ext}  (   \lfloor \frac{|I|}{2} \rfloor + \lfloor  \frac{n}{2} \rfloor  + 1)  \cdot   \Big(  \frac{\eps  }{(1+t+|q|)^{2} \cdot (1+|q|)^{2\gamma} } \Big) \\
     \notag
   && + \sum_{|K|\leq |I|}  \Big( |   \Lie_{Z^K} A  |^2 +|   \Lie_{Z^K} h  |^2 \Big)\cdot C(q_0)  \cdot E^{ext}  (   \lfloor \frac{|I|}{2} \rfloor + \lfloor  \frac{n}{2} \rfloor  + 1)  \cdot \Big( \frac{\eps }{(1+t+|q|)^{2} \cdot (1+|q|)^{2+2\gamma} }  \Big) \; .
      \notag 
   \eeaa

\end{proof}

\begin{lemma}
For $n \geq 4$, $\de = 0$, and $q \geq q_0$, we have
   \beaa
   \notag
 &&  \frac{(1+t )^{1+\la}}{\eps} \cdot  | \Lie_{Z^I}   ( g^{\la\mu} \derm_{\la}   \derm_{\mu}    h ) |^2   \\
     \notag
   &\les&   C(q_0)  \cdot E^{ext}  (   \lfloor \frac{|I|}{2} \rfloor + \lfloor  \frac{n}{2} \rfloor  + 1) \\
   && \cdot \sum_{|K|\leq |I|}  \Big( |    \derm (\Lie_{Z^K} A ) |^2 + |    \derm (\Lie_{Z^K} h ) |^2 \Big) \cdot \Big(   \frac{\eps }{(1+t+|q|)^{2-\la} \cdot (1+|q|)^{2+2\gamma}}  \Big)  \\
      \notag 
   &&+ C(q_0) \cdot E^{ext}  (   \lfloor \frac{|I|}{2} \rfloor + \lfloor  \frac{n}{2} \rfloor  + 1)  \\
   \notag
&& \cdot    \sum_{|K|\leq |I|}  \Big( |   \Lie_{Z^K} A  |^2 + |   \Lie_{Z^K} h  |^2 \Big) \cdot \Big(   \frac{\eps }{(1+t+|q|)^{5-\la} \cdot (1+|q|)^{2+4\gamma}} \Big)  \; .
      \notag 
   \eeaa

\end{lemma}

\begin{proof}
Using Klainerman-Sobolev inequality in the exterior, we get based on what we showed for $n \geq 4$, $\de =0$, 
\beaa
   \notag
 &&  \frac{(1+t )}{\eps} \cdot  | \Lie_{Z^I}   ( g^{\la\mu} \derm_{\la}   \derm_{\mu}    h ) |^2   \\
     \notag
   &\les&  \sum_{|K|\leq |I|}  \Big( |    \derm (\Lie_{Z^K} A ) |^2 + |    \derm (\Lie_{Z^K} h ) |^2 \Big) \cdot E^{ext} (   \lfloor \frac{|I|}{2} \rfloor + \lfloor  \frac{n}{2} \rfloor  + 1)\\
 \notag
   && \cdot \Big(  \begin{cases}  \frac{\eps }{(1+t+|q|)^{2} \cdot (1+|q|)^{2+2\gamma}},\quad\text{when }\quad q>0,\\
           \notag
      \frac{\eps  }{(1+t+|q|)^{2} \cdot ( 1+|q| )^{1} } \,\quad\text{when }\quad q<0 . \end{cases} \Big) \\
      \notag 
   &&+  \sum_{|K|\leq |I|}  \Big( |   \Lie_{Z^K} A  |^2 + |   \Lie_{Z^K} h  |^2 \Big) \cdot E^{ext}  (   \lfloor \frac{|I|}{2} \rfloor + \lfloor  \frac{n}{2} \rfloor  + 1)\\
 \notag
   && \cdot \Big(  \begin{cases}  \frac{\eps }{(1+t+|q|)^{5} \cdot (1+|q|)^{2+4\gamma}},\quad\text{when }\quad q>0,\\
           \notag
      \frac{\eps  }{(1+t+|q|)^{5-4\delta} } \,\quad\text{when }\quad q<0 . \end{cases} \Big) \; .
      \notag 
   \eeaa
Thus, we obtain for $q \geq q_0$,
   \beaa
   \notag
 &&  \frac{(1+t )}{\eps} \cdot  | \Lie_{Z^I}   ( g^{\la\mu} \derm_{\la}   \derm_{\mu}    h ) |^2   \\
     \notag
   &\les&   C(q_0)  \cdot E^{ext}  (   \lfloor \frac{|I|}{2} \rfloor + \lfloor  \frac{n}{2} \rfloor  + 1)   \\
   \notag
   &&   \sum_{|K|\leq |I|}  \Big( |    \derm (\Lie_{Z^K} A ) |^2 + |    \derm (\Lie_{Z^K} h ) |^2 \Big) \cdot \Big(   \frac{\eps }{(1+t+|q|)^{2} \cdot (1+|q|)^{2+2\gamma}}  \Big) \\
      \notag 
   &&+ C(q_0) \cdot E^{ext}  (   \lfloor \frac{|I|}{2} \rfloor + \lfloor  \frac{n}{2} \rfloor  + 1) \cdot \sum_{|K|\leq |I|}  \Big( |   \Lie_{Z^K} A  |^2 + |   \Lie_{Z^K} h  |^2 \Big)  \cdot \Big(   \frac{\eps }{(1+t+|q|)^{5} \cdot (1+|q|)^{2+4\gamma}} \Big) \; .
      \notag 
   \eeaa

\end{proof}

We recapitulate the following corollary from \cite{G4}.

\begin{corollary}\label{HardytypeinequalityforintegralstartingatROm}
Let $w$ defined as in Definition \ref{defoftheweightw}, where $\ga > 0$.
Let  $\Phi$ a tensor that decays fast enough at spatial infinity for all time $t$\,, such that
\bea
 \int_{\SSS^{n-1}} \lim_{r \to \infty} \Big( \frac{r^{n-1}}{(1+t+r)^{a} \cdot (1+|q|) } w(q) \cdot <\Phi, \Phi>  \Big)  d\si^{n-1} (t ) &=& 0 \; .
\eea
Let $R(\Om)  \geq 0 $\,, be a function of $\Om \in \SSS^{n-1}$\,. Then, since $\ga \neq 0$\,, we have for $0 \leq a \leq n-1$\,, that 
\bea
\notag
 &&   \int_{\SSS^{n-1}} \int_{r=R(\Om)}^{r=\infty} \frac{r^{n-1}}{(1+t+r)^{a}} \cdot   \frac{w (q)}{(1+|q|)^2} \cdot <\Phi, \Phi>   \cdot dr  \cdot d\si^{n-1}  \\
 \notag
 &\leq& c(\ga) \cdot  \int_{\SSS^{n-1}} \int_{r=R(\Om)}^{r=\infty}  \frac{ r^{n-1}}{(1+t+r)^{a}} \cdot w(q) \cdot <\pa_r\Phi, \pa_r \Phi>  \cdot  dr  \cdot d\si^{n-1}  \; , \\
 \eea
 where the constant $c(\ga)$ does not depend on $R(\Om)$\,.
 \end{corollary}

\begin{lemma}\label{estimateonthesourcetermsforAsuitableforneq4}
For $n \geq 4$, $\de = 0$, and for fields decaying fast enough at spatial infinity, such that for all time $t$, for $|K| \leq |I| $
\bea
\notag
 \int_{\SSS^{n-1}} \lim_{r \to \infty} \Big( \frac{r^{n-1}}{(1+t+r)^{2-\la} \cdot (1+|q|) } w(q) \cdot   \Big( |   \Lie_{Z^K} A  |^2 + |   \Lie_{Z^K} h  |^2  \Big)  d\si^{n-1} (t ) &=& 0  \; , \\
\eea
then, for $\ga \neq 0$ and $\mu \neq \frac{1}{2}$, we have
 \bea
   \notag
 &&   \int_0^t \Big(  \int_{\Si_{t}^{ext}}  \frac{(1+t )^{1+\la}}{\eps} \cdot  | \Lie_{Z^I}   ( g^{\la\mu} \derm_{\la}   \derm_{\mu}    A ) |^2  \cdot   w \Big) \cdot dt   \\
     \notag
    &\les& c(\ga, \mu) \cdot C(q_0) \cdot E^{ext} (   \lfloor \frac{|I|}{2} \rfloor + \lfloor  \frac{n}{2} \rfloor  + 1)  \\
    \notag
    &&  \times   \int_0^t          \frac{\eps  }{(1+t)^{2-\la} \ }  \cdot  \Big(  \int_{\Si_{t}^{ext}}   \sum_{|K|\leq |I|}  \Big( |    \derm (\Lie_{Z^K} A ) |^2 + |    \derm (\Lie_{Z^K} h ) |^2 \Big)   \cdot   w \Big) \cdot dt \; . \\
   \eea

\end{lemma}

\begin{proof}
  We showed that for $n \geq 4$, $\de = 0$, for $q \geq q_0$, we have

   \beaa
 \notag
&&  \frac{(1+t )^{1+\la}}{\eps} \cdot  |  \Lie_{Z^I}   ( g^{\la\mu} \derm_{\la}   \derm_{\mu}    A ) |^2 \, \\
   \notag
   &\les& C(q_0) \cdot E^{ext} (   \lfloor \frac{|I|}{2} \rfloor + \lfloor  \frac{n}{2} \rfloor  + 1)  \\
&&   \times \sum_{|K|\leq |I|}  \Big( |    \derm (\Lie_{Z^K} A ) |^2 +  |    \derm (\Lie_{Z^K} h ) |^2 \Big)  \cdot   \Big(  \frac{\eps  }{(1+t+|q|)^{2-\la} \cdot (1+|q|)^{2\gamma} } \Big)\\
     \notag
   && +  C(q_0)  \cdot E^{ext} (   \lfloor \frac{|I|}{2} \rfloor + \lfloor  \frac{n}{2} \rfloor  + 1)  \cdot \sum_{|K|\leq |I|}  \Big( |   \Lie_{Z^K} A  |^2 +|   \Lie_{Z^K} h  |^2 \Big) \cdot \Big( \frac{\eps }{(1+t+|q|)^{2-\la} \cdot (1+|q|)^{2+2\gamma} }  \Big) \; .
      \notag 
   \eeaa

Based on the Hardy inequality that we have shown in Corollary \ref{HardytypeinequalityforintegralstartingatROm}, we get that for $\ga \neq 0$ and $ 0 < \la \leq \frac{1}{2}$ (and therefore $2-\la \leq 3 \leq n-1$ for $n \geq 4$), under the assumption again that $\Lie_{Z^K} A$ and $\Lie_{Z^K} h$ decay fast enough at spatial infinity for all time $t$, for $|K| \leq |I| $, such that
\bea
\notag
 \int_{\SSS^{n-1}} \lim_{r \to \infty} \Big( \frac{r^{n-1}}{(1+t+r)^{2-\la} \cdot (1+|q|) } w(q) \cdot   \Big( |   \Lie_{Z^K} A  |^2 + |   \Lie_{Z^K} h  |^2  \Big)  d\si^{n-1} (t ) &=& 0  \; , \\
\eea
that
\bea
\notag
 &&   \int_{\SSS^{n-1}} \int_{r=R(\Om)}^{r=\infty} \frac{r^{n-1}}{(1+t+r)^{2-\la}} \cdot   \frac{w (q)}{(1+|q|)^2} \cdot <\Phi, \Phi>   \cdot dr  \cdot d\si^{n-1}  \\
 \notag
 &\leq& c(\ga, \mu) \cdot  \int_{\SSS^{n-1}} \int_{r=R(\Om)}^{r=\infty}  \frac{ r^{n-1}}{(1+t+r)^{2-\la}} \cdot w(q) \cdot <\pa_r\Phi, \pa_r \Phi>  \cdot  dr  \cdot d\si^{n-1}  \; . \\
 \eea

By choosing $R(\Om)$ such that when $\Om$ spans $\SSS^{n-1}$, we obtain the intersection of $\Sigma_t$ and $N_{t_1}^{t_2}$ (the null boundary for the metric $g$ of $\overline{C}$). We get
  \beaa
\notag
  && \int_{\Si_{t}^{ext}} \frac{1}{(1+t+|q| )^{2-\la}(1+|q|)^2} \cdot \Big( |   \Lie_{Z^K} A  |^2 + |   \Lie_{Z^K} h  |^2  \Big) \cdot   w    \\
     &\leq& c(\ga) \cdot  \int_{\Si_{t}^{ext}}  \frac{1}{(1+t+|q|)^{2-\la}}  \cdot  \Big( |  \derm( \Lie_{Z^K} A ) |^2 + | \derm(    \Lie_{Z^K} h )  |^2  \Big)  \cdot   w    \; .
 \eeaa
 As a result,
        \beaa
 \notag
&& \int_0^t \Big( \int_{\Si_{t}^{ext}} \frac{(1+t )^{1+\la}}{\eps} \cdot  |  \Lie_{Z^I}   ( g^{\la\mu} \derm_{\la}   \derm_{\mu}    A ) |^2 \cdot w \Big) \cdot dt\, \\
   \notag
   &\les&   C(q_0) \cdot E^{ext} (   \lfloor \frac{|I|}{2} \rfloor + \lfloor  \frac{n}{2} \rfloor  + 1)  \\
   &&  \cdot  \int_0^t  \Big( \int_{\Si_{t}^{ext}} \sum_{|K|\leq |I|}  \Big( |    \derm (\Lie_{Z^K} A ) |^2 +  |    \derm (\Lie_{Z^K} h ) |^2 \Big)   \cdot       \frac{\eps  }{(1+t+|q|)^{2-\la}  } \cdot w \Big)  \cdot dt \\
     \notag
   && +  C(q_0) \cdot E^{ext} (   \lfloor \frac{|I|}{2} \rfloor + \lfloor  \frac{n}{2} \rfloor  + 1)   \\
 && \cdot \int_0^t    \Big( \int_{\Si_{t}^{ext}} \sum_{|K|\leq |I|}  \Big( |   \Lie_{Z^K} A  |^2 +|   \Lie_{Z^K} h  |^2 \Big)\cdot   \frac{\eps  }{(1+t+|q|)^{2-\la} \cdot ( 1+|q| )^{2} }  \cdot w \Big)  \cdot dt \\
      \notag 
         &\les&  C(q_0) \cdot E^{ext} (   \lfloor \frac{|I|}{2} \rfloor + \lfloor  \frac{n}{2} \rfloor  + 1)  \\
         && \cdot \int_0^t   \Big( \int_{\Si_{t}^{ext}} \sum_{|K|\leq |I|}  \Big( |    \derm (\Lie_{Z^K} A ) |^2 +  |    \derm (\Lie_{Z^K} h ) |^2 \Big)  \cdot       \frac{\eps  }{(1+t+|q|)^{2-\la}  } \cdot w \Big) \cdot dt  \\
     \notag
   && +   C(q_0) \cdot E^{ext} (   \lfloor \frac{|I|}{2} \rfloor + \lfloor  \frac{n}{2} \rfloor  + 1) \\
&&  \cdot  \int_0^t  \Big( \int_{\Si_{t}^{ext}} \sum_{|K|\leq |I|}  \Big( | \derm (   \Lie_{Z^K} A ) |^2 +| \derm (  \Lie_{Z^K} h )   |^2 \Big) \cdot      \frac{ c(\ga) \cdot\eps  }{(1+t+|q|)^{2-\la}  }  \cdot w \Big)   \cdot dt \; .
      \notag 
   \eeaa
   
   \end{proof}

\begin{lemma}\label{Theestimateonthesourcetermsforhsuitableforneq4}
For $n \geq 4$, $\de = 0$, and for fields decaying fast enough at spatial infinity, such that for all time $t$, for $|K| \leq |I| $,
\bea
\notag
 \int_{\SSS^{n-1}} \lim_{r \to \infty} \Big( \frac{r^{n-1}}{(1+t+r)^{2-\la} \cdot (1+|q|) } \cdot w(q) \cdot   \Big( |   \Lie_{Z^K} A  |^2 + |   \Lie_{Z^K} h  |^2  \Big)  d\si^{n-1} (t ) &=& 0  \; , \\
\eea
then, for $\ga \neq 0$ we have
 \beaa
   \notag
 &&   \int_0^t \Big(  \int_{\Si_{t}^{ext}}  \frac{(1+t )^{1+\la}}{\eps} \cdot  | \Lie_{Z^I}   ( g^{\la\mu} \derm_{\la}   \derm_{\mu}    h ) |^2  \cdot   w \Big) \cdot dt   \\
     \notag
    &\les& c(\ga) \cdot  C(q_0) \cdot E^{ext} (   \lfloor \frac{|I|}{2} \rfloor + \lfloor  \frac{n}{2} \rfloor  + 1)  \\
    \notag
    &&  \times    \int_0^t          \frac{\eps  }{(1+t)^{2-\la} \ }  \cdot  \Big(  \int_{\Si_{t}^{ext}}   \sum_{|K|\leq |I|}  \Big( |    \derm (\Lie_{Z^K} A ) |^2 + |    \derm (\Lie_{Z^K} h ) |^2 \Big)   \cdot   w \Big) \cdot dt \; .  \\
   \eeaa

\end{lemma}

\begin{proof}

We have shown that for $n \geq 4$, $\de = 0$, and $q \geq q_0$,

   \beaa
   \notag
 &&  \frac{(1+t )^{1+\la}}{\eps} \cdot  | \Lie_{Z^I}   ( g^{\la\mu} \derm_{\la}   \derm_{\mu}    h ) |^2   \\
     \notag
   &\les&   C(q_0)  \cdot E^{ext}  (   \lfloor \frac{|I|}{2} \rfloor + \lfloor  \frac{n}{2} \rfloor  + 1)  \\
   &&  \times \sum_{|K|\leq |I|}  \Big( |    \derm (\Lie_{Z^K} A ) |^2 + |    \derm (\Lie_{Z^K} h ) |^2 \Big) \cdot \Big(   \frac{\eps }{(1+t+|q|)^{2-\la} \cdot (1+|q|)^{2+2\gamma}}  \Big)\\
      \notag 
   &&+ C(q_0) \cdot E^{ext}  (   \lfloor \frac{|I|}{2} \rfloor + \lfloor  \frac{n}{2} \rfloor  + 1) \\
   \notag
&&  \times  \sum_{|K|\leq |I|}  \Big( |   \Lie_{Z^K} A  |^2 + |   \Lie_{Z^K} h  |^2 \Big)  \cdot \Big(   \frac{\eps }{(1+t+|q|)^{5-\la} \cdot (1+|q|)^{2+4\gamma}} \Big) \; .
      \notag 
   \eeaa

Thus,
\beaa
   \notag
 &&  \frac{(1+t )^{1+\la}}{\eps} \cdot  | \Lie_{Z^I}   ( g^{\la\mu} \derm_{\la}   \derm_{\mu}    h ) |^2   \\
     \notag
   &\les&   C(q_0) \cdot E^{ext}(   \lfloor \frac{|I|}{2} \rfloor + \lfloor  \frac{n}{2} \rfloor  + 1) \cdot \sum_{|K|\leq |I|}  \Big( |    \derm (\Lie_{Z^K} A ) |^2 + |    \derm (\Lie_{Z^K} h ) |^2 \Big)   \cdot      \frac{\eps  }{(1+t+|q|)^{2-\la} \cdot ( 1+|q| )^{2} }  \\
      \notag 
   &&+   C(q_0) \cdot E^{ext} (   \lfloor \frac{|I|}{2} \rfloor + \lfloor  \frac{n}{2} \rfloor  + 1) \cdot \sum_{|K|\leq |I|}  \Big( |   \Lie_{Z^K} A  |^2 + |   \Lie_{Z^K} h  |^2 \Big)   \cdot        \frac{\eps  }{(1+t+|q|)^{2-\la} \cdot (1+|q|)^2  } \; .
      \notag 
   \eeaa

   Assuming that both $\Lie_{Z^K} A$ and $\Lie_{Z^K} h$ decay fast enough at spatial infinity for all time $t$, i.e. that
\bea
\notag
 \int_{\SSS^{n-1}} \lim_{r \to \infty} \Big( \frac{r^{n-1}}{(1+t+r)^{2-\la} \cdot (1+|q|) } w(q) \cdot   \Big( |   \Lie_{Z^K} A  |^2 + |   \Lie_{Z^K} h  |^2  \Big)  d\si^{n-1} (t ) &=& 0  \; . \\
\eea
Then, for $\ga \neq 0$ and $ 0 < \la \leq \frac{1}{2}$, we have for $0 \leq 2-\la \leq 3 \leq n-1$ (for $n \geq 4$), we get that
  \beaa
\notag
  && \int_{\Si_{t}^{ext}} \frac{1}{(1+t+|q| )^{2-\la}(1+|q|)^2} \cdot \Big( |   \Lie_{Z^K} A  |^2 + |   \Lie_{Z^K} h  |^2  \Big) \cdot   w    \\
     &\leq& c(\ga) \cdot  \int_{\Si_{t}^{ext}}  \frac{1}{(1+t+|q|)^{2-\la}}  \cdot  \Big( |  \derm( \Lie_{Z^K} A ) |^2 + | \derm(    \Lie_{Z^K} h )  |^2  \Big)  \cdot   w    \\
 \eeaa
 As a result,
 \beaa
   \notag
 &&  \int_{\Si_{t}}  \frac{(1+t )^{1+\la}}{\eps} \cdot  | \Lie_{Z^I}   ( g^{\la\mu} \derm_{\la}   \derm_{\mu}    h ) |^2  \cdot   w  \\
     \notag
   &\les&   C(q_0) \cdot E^{ext} (   \lfloor \frac{|I|}{2} \rfloor + \lfloor  \frac{n}{2} \rfloor  + 1)   \\
   \notag
  && \times  \int_{\Si_{t}}   \sum_{|K|\leq |I|}  \Big( |    \derm (\Lie_{Z^K} A ) |^2 + |    \derm (\Lie_{Z^K} h ) |^2 \Big) \cdot      \frac{\eps  }{(1+t+|q|)^{2-\la} \cdot ( 1+|q| )^{2} }  \cdot   w \\
      \notag 
   &&+ C(q_0) \cdot E^{ext} (   \lfloor \frac{|I|}{2} \rfloor + \lfloor  \frac{n}{2} \rfloor  + 1) \cdot   \int_{\Si_{t}}   \sum_{|K|\leq |I|}  \Big( |   \Lie_{Z^K} A  |^2 + |   \Lie_{Z^K} h  |^2 \Big)  \cdot        \frac{\eps  }{(1+t+|q|)^{2-\la} \cdot (1+|q|)^2  }  \cdot   w \\
      \notag 
       &\les&   c(\ga) \cdot  C(q_0) \cdot E^{ext}  (   \lfloor \frac{|I|}{2} \rfloor + \lfloor  \frac{n}{2} \rfloor  + 1)  \cdot      \frac{\eps  }{(1+t)^{2-\la} \ }  \cdot  \int_{\Si_{t}}   \sum_{|K|\leq |I|}  \Big( |    \derm (\Lie_{Z^K} A ) |^2 + |    \derm (\Lie_{Z^K} h ) |^2 \Big)   \cdot   w \\
      \notag 
   &&+  C(q_0) \cdot E^{ext} (   \lfloor \frac{|I|}{2} \rfloor + \lfloor  \frac{n}{2} \rfloor  + 1) \cdot  \int_{\Si_{t}}   \sum_{|K|\leq |I|}  \Big( |   \derm ( \Lie_{Z^K} A ) |^2 + |  \derm ( \Lie_{Z^K} h )  |^2 \Big) \cdot       \frac{\eps  }{(1+t+|q|)^{4-\la}  }  \cdot   w \\
   \notag
    &\les&   c(\ga) \cdot  C(q_0) \cdot E^{ext} (   \lfloor \frac{|I|}{2} \rfloor + \lfloor  \frac{n}{2} \rfloor  + 1)  \cdot      \frac{\eps  }{(1+t)^{2-\la} \ }  \cdot  \int_{\Si_{t}}   \sum_{|K|\leq |I|}  \Big( |    \derm (\Lie_{Z^K} A ) |^2 + |    \derm (\Lie_{Z^K} h ) |^2 \Big)   \cdot   w \; .
   \eeaa
   
   \end{proof}

\subsection{Grönwall type inequality on the exterior energy for $n\geq 4$}\
   
       \begin{lemma}\label{Gronwallinequalityintheexteriorontheenergyfornequal4}
       
        For $ H^{\mu\nu} = g^{\mu\nu}-m^{\mu\nu}$ satisfying 
\bea
| H| \leq  \frac{1}{n} \; ,
\eea
and for $\Lie_{Z^J} A $ and $\Lie_{Z^J} h^1$ decaying sufficiently fast at spatial infinity as in the bootstrap argument, and with the condition that for $\ga > 0 $ and for all $|K| \leq |I| $,
 \bea
\notag
 \int_{\SSS^{n-1}} \lim_{r \to \infty} \Big( \frac{r^{n-1}}{(1+t+r)^{2-\la} \cdot (1+|q|) } \cdot w(q) \cdot   \Big( |   \Lie_{Z^K} A  |^2 + |   \Lie_{Z^K} h  |^2  \Big)  d\si^{n-1} (t ) &=& 0  \; , \\
\eea

 then for $\de= 0$, and for $ 0 < \la \leq \frac{1}{2}$\,,

\bea
\notag
 &&  {(\E_{|I|}^{ext})}^2 (t_2)  \\
 \notag
  &\les&{(\E_{|I|}^{ext})}^2(t_1) +      C(q_0) \cdot c(\ga)  \cdot E^{ext} ( \lfloor \frac{|I|}{2} \rfloor+  \lfloor  \frac{n}{2} \rfloor  +1)   \cdot C(|I|)    \cdot \int_{t_1}^{t_2}   \frac{\eps}{ (1+ t  )^{1+\la}  }  \cdot   {(\E_{|I|}^{ext})}^2 (\tau) \cdot d\tau \; , \\
  \eea
  
where
\beaa
\E_{|I|}^{ext} (\tau) :=  \sum_{|J|\leq |I|} \big( \|w^{1/2}   \derm ( \Lie_{Z^J} h^1   (t,\cdot) )  \|_{L^2 (\Sigma^{ext}_{\tau})} +  \|w^{1/2}   \derm ( \Lie_{Z^J}  A   (t,\cdot) )  \|_{L^2 (\Sigma^{ext}_{\tau})} \big) \, ,
\eeaa
with $w$ defined as in Definition \ref{defoftheweightw}, with $\ga >0$\,.
\end{lemma}

\begin{proof}
Based on Lemmas \ref{estimateonthesourcetermsforAsuitableforneq4} and \ref{Theestimateonthesourcetermsforhsuitableforneq4} and injecting in Lemma \ref{TheenergyestimateinGronwallformonbothAandhforneq4}, we have under the stated assumptions, 
         \beaa
 &&    {(\E_{|I|}^{ext})}^2 (t_2)    \\
 \notag
       &\les & {(\E_{|I|}^{ext})}^2(t_1)     +  C ( |I| ) \cdot C(q_0) \cdot  c(\ga) \cdot  E^{ext} ( \lfloor \frac{|I|}{2} \rfloor+  \lfloor  \frac{n}{2} \rfloor  +1)    \cdot \int_{t_1}^{t_2}   \frac{\eps}{ (1+ t  )^{1+\la}  }  \cdot  {(\E_{|I|}^{ext})}^2 (\tau) \cdot d\tau \\
     \notag
&& + C(|I|) \cdot \int_0^t  \Big( \int_{\Si_{t}^{ext}}  \frac{(1+\tau )^{1+\la}}{\eps}  \cdot \sum_{|K| \leq |I| }  | \Lie_{Z^K}  g^{\a\b} \derm_{\a}   \derm_{\b}    A |^2   \cdot w   \cdot dx^1 \ldots dx^n  \Big) \cdot d\tau   \\
      \notag
&& + C(|I|) \cdot  \int_0^t  \Big( \int_{\Si_{t}^{ext}}  \frac{(1+\tau )^{1+\la}}{\eps}  \cdot \sum_{|K| \leq |I| }  | \Lie_{Z^K}  g^{\a\b} \derm_{\a}   \derm_{\b}    h |^2   \cdot w   \cdot dx^1 \ldots dx^n  \Big) \cdot d\tau \; \\
\notag
    &\les&  \E_{|I|}^{ext} (t_1)     +  C ( |I| ) \cdot C(q_0) \cdot  c(\ga) \cdot  E^{ext} ( \lfloor \frac{|I|}{2} \rfloor+  \lfloor  \frac{n}{2} \rfloor  +1)    \cdot \int_{t_1}^{t_2}   \frac{\eps}{ (1+ t  )^{1+\la}  }  \cdot  {(\E_{|I|}^{ext})}^2 (\tau) \cdot d\tau \\
    \notag
    && + C(|I|) \cdot c(\ga) \cdot C(q_0) \cdot E^{ext} (   \lfloor \frac{|I|}{2} \rfloor + \lfloor  \frac{n}{2} \rfloor  + 1)  \\
    \notag
    && \times   \int_0^t          \frac{\eps  }{(1+t)^{2-\la} \ }  \cdot  \Big(  \int_{\Si_{t}^{ext}}   \sum_{|K|\leq |I|}  \Big( |    \derm (\Lie_{Z^K} A ) |^2 + |    \derm (\Lie_{Z^K} h ) |^2 \Big)   \cdot   w \Big) \cdot dt \; . \\
   \eeaa

Hence, fixing $\delta =0$, and $ 0 < \la \leq \frac{1}{2}$, we obtain the result.

\end{proof}

\subsection{The proof of the theorem for $n\geq 4$}\

\begin{proposition}\label{Thepropositiononclosingthebootstrapargumenttoactuallyboundtheenergy}
Let $n \geq 4$. Consider initial data $\Lie_{Z^J} A $ and $\Lie_{Z^J} h^1$ decaying sufficiently fast at spatial infinity at $t=0$. For every $N \geq 2 \lfloor  \frac{n}{2} \rfloor  + 2$, for every constant $E^{ext} ( N )$ (to bound $\E_{N}^{ext}(t)$ in \eqref{boundonthegrowthoftheenergyinproposition}), there exists a constant $c_0$, that depends on $E^{ext} ( N )$\;, on $N$ and on $w$ (i.e. depends on $\gamma$), such that if
\beaa
\overline{\E}_{N} (0) \leq c_0 \; ,
\eeaa
then for all time $t$, we have
\beaa\label{boundonthegrowthoftheenergyinproposition}
\E_{N}^{ext}(t) \leq E^{ext} ( N ) \; ,
\eeaa
and consequently, in the the Lorenz gauge, the Yang-Mills fields decay to zero and the metric decays to the Minkowski metric in wave coordinates, for the initial value Cauchy problem for the Einstein Yang-Mills equations that we defined in the set-up, which will consequently admit a global solution in time $t$. More precisely, for all  $|J| \leq N -  \lfloor  \frac{n}{2} \rfloor  - 1$, we have in the exterior region $\overline{C}$, which is contained in $q \geq q_0$,

 \beaa
 \notag
 |\derm  ( \Lie_{Z^J}  A ) (t,x)  |     + |\derm  ( \Lie_{Z^J}  h ) (t,x)  |     &\leq&  C(q_0) \cdot   E^{ext} ( N )  \cdot \frac{\eps }{(1+t+|q|)^{\frac{(n-1)}{2}} (1+|q|)^{1+\gamma}} \; ,\\
      \eeaa
and
 \beaa
 \notag
 |\Lie_{Z^J} A (t,x)  | + |\Lie_{Z^J}  h (t,x)  |  &\leq&  C(q_0) \cdot E^{ext} (N)  \cdot \frac{\eps }{(1+t+|q|)^{\frac{(n-1)}{2}} (1+|q|)^{\gamma}} \; .
      \eeaa

\end{proposition}

\begin{proof}
We start with the bootstrap assumption on $\E_{ (  \lfloor \frac{N}{2} \rfloor + \lfloor  \frac{n}{2} \rfloor  + 1)}$. We have then, thanks to \eqref{aprioriestimateonBigHwithLiederivativeL}, for $n\geq 4$ and for $\de = 0$\;, that
       \beaa
 \notag
|   H (t,x)  | &\les& \begin{cases}  c (\gamma)  \cdot  \frac{ \E_{ ( \lfloor  \frac{n}{2} \rfloor  +1)}}{ (1+ t + | q | )^{2 }  (1+| q |   )^{\ga}}  ,\quad\text{when }\quad q>0,\\
\notag
    \frac{ \E_{ ( \lfloor  \frac{n}{2} \rfloor  +1)}}{ (1+ t + | q | )^{2 }  } (1+| q |   )^{\frac{1}{2} }  , \,\quad\text{when }\quad q<0 . \end{cases} \\ 
     &\les&  c (\gamma)  \cdot   \E_{ ( \lfloor  \frac{n}{2} \rfloor  +1)}\\
               &\les&  c (\gamma)  \cdot   E (  \lfloor  \frac{n}{2} \rfloor  +1)  \\
                    && \text{(where we used that we chose $\de = 0$ and $\eps = 1$\;,}\\
                    &&\text{see \eqref{delataqualtozero} and \eqref{epsequaltoone})}.
    \eeaa

By choosing $E (  \lfloor  \frac{n}{2} \rfloor  +1)$ small enough, depending on $\ga$ and on $n$\,, we have
\bea
 c (\gamma)  \cdot  E (  \lfloor  \frac{n}{2} \rfloor  +1) < \frac{1}{n} \; .
\eea
We take initial data decaying sufficiently fast at spatial infinity, and since the fields satisfy a wave equation, we claim that this spatial decay will propagate in time under the bootstrap assumption, and thus, they will be satisfied for all time $t$, in a way that we could use Lemma \ref{Gronwallinequalityintheexteriorontheenergyfornequal4}, where we fix  $ 0 < \la \leq \frac{1}{2}$ arbitrary. Consequently, we get
\bea
\notag
 && ( \E_{N}^{ext} )^2 (t)  \\
 \notag
  &\leq& C \cdot ( \E_{N}^{ext} )^2 (0) +    c(\ga) \cdot E (   \lfloor \frac{|I|}{2} \rfloor + \lfloor  \frac{n}{2} \rfloor  + 1)  \cdot  C(N) \cdot  \int_0^t          \frac{\eps  }{(1+\tau)^{1+\la} \ }  \cdot  \E_{N}^2 (\tau)  \cdot d\tau \;.
\eea
Now, using Grönwall lemma, we get
\bea
\notag
 &&  ( \E_{N}^{ext} )^2 (t)  \leq C\cdot \E_{N}^2 (0) \cdot \exp \Big(  \int_0^t    c(\ga) \cdot E (   \lfloor \frac{N}{2} \rfloor + \lfloor  \frac{n}{2} \rfloor  + 1)  \cdot  C(N) \cdot   \eps \cdot       \frac{1 }{(1+\tau)^{1+\la} \ }  \cdot d\tau     \Big) \\
\notag
 & \leq& C \cdot \E_{N}^2 (0) \cdot \exp \Big(     c(\ga) \cdot E (   \lfloor \frac{N}{2} \rfloor + \lfloor  \frac{n}{2} \rfloor  + 1)  \cdot  C(N) \cdot   \eps  \cdot \Big[  \frac{-1 }{\la (1+\tau)^{\la} }   \Big]^{\infty}_{0}      \Big) \\
  & \leq& C \cdot \E_{N}^2 (0) \cdot \exp \Big(    c(\ga) \cdot E (   \lfloor \frac{N}{2} \rfloor + \lfloor  \frac{n}{2} \rfloor  + 1)  \cdot  C(N) \cdot   \eps  \cdot \frac{1 }{\la }     \Big) \; ,
\eea
which also leads to, using that we chose $\eps \leq 1$ and that $E(k) \leq 1$, that
\beaa
  \E_{N}^{ext}  (t)  & \leq& C \cdot \E_{|I|} (0) \cdot \exp \Big(    c(\ga) \cdot E (   \lfloor \frac{N}{2} \rfloor + \lfloor  \frac{n}{2} \rfloor  + 1)  \cdot  C(N) \cdot   \eps  \cdot \frac{1 }{\la }     \Big) \\
  & \leq& C \cdot \E_{N} (0) \cdot \exp \Big(    c(\ga)   \cdot  C(N)   \cdot \frac{1 }{\la }     \Big) \; .
\eeaa
Thus, choosing an initial data such that
\bea
 \overline{\E}_{N} (0) \leq  \frac{1}{2\cdot C  \cdot \exp \Big(    c(\ga)  \cdot  C(N)   \cdot \frac{1 }{\la }     \Big)} \cdot  E (   \lfloor \frac{N}{2} \rfloor + \lfloor  \frac{n}{2} \rfloor  + 1) \; ,
\eea
implies that
\bea
 \E_{N} (0) \leq  \frac{1}{2\cdot C  \cdot \exp \Big(    c(\ga)  \cdot  C(N)   \cdot \frac{1 }{\la }     \Big)} \cdot  E (   \lfloor \frac{N}{2} \rfloor + \lfloor  \frac{n}{2} \rfloor  + 1) \; ,
\eea
This leads to
\beaa
\notag
 &&  \E_{N}^{ext}  (t)  \leq \frac{1}{2} \cdot E (   \lfloor \frac{N}{2} \rfloor + \lfloor  \frac{n}{2} \rfloor  + 1)  \; .
\eeaa
However, for $N \geq  \lfloor \frac{N}{2} \rfloor + \lfloor  \frac{n}{2} \rfloor  + 1$, which means for $\frac{N}{2} \geq  \lfloor  \frac{n}{2} \rfloor  + 1$, we have
\beaa
  \E_{   \lfloor \frac{N}{2} \rfloor + \lfloor  \frac{n}{2} \rfloor  + 1}^{ext}   (t)  \leq   \E_{|I|} (0)  \; .
 \eeaa
 Thus,
 \bea
  &&  \E_{   \lfloor \frac{N}{2} \rfloor + \lfloor  \frac{n}{2} \rfloor  + 1}^{ext}   (t)  \leq \frac{1}{2} \cdot E (   \lfloor \frac{N}{2} \rfloor + \lfloor  \frac{n}{2} \rfloor  + 1)    \; .
  \eea
This shows that the estimate $  \E_{   \lfloor \frac{N}{2} \rfloor + \lfloor  \frac{n}{2} \rfloor  + 1}  (t)  \leq E (   \lfloor \frac{N}{2} \rfloor + \lfloor  \frac{n}{2} \rfloor  + 1)    $\, is in fact a true estimate and therefore, we can close the bootstrap argument for  $  \E_{   \lfloor \frac{N}{2} \rfloor + \lfloor  \frac{n}{2} \rfloor  + 1}  (t)$\,, with $\eps=1$ and $\de=0$. For this, we have used the condition that
\beaa
N \geq  \lfloor \frac{N}{2} \rfloor + \lfloor  \frac{n}{2} \rfloor  + 1\;,
\eeaa
 which imposes that $N \geq 2 \lfloor  \frac{n}{2} \rfloor  + 2$, and we also got that
\bea
 &&  \E_{N}^{ext}  (t)  \leq \frac{1}{2} \cdot E (   \lfloor \frac{N}{2} \rfloor + \lfloor  \frac{n}{2} \rfloor  + 1) \; .
\eea
This in turn gives, using Lemmas \ref{aprioriestimatefrombootstraponzerothderivativeofAandh1} and \ref{aprioriestimatesongradientoftheLiederivativesofthefields}, the stated decay estimates on the fields.

\end{proof}

\end{document}